\documentclass[3p]{elsarticle}

\usepackage{amsmath,amssymb,amsthm}
\usepackage{bm}
\usepackage{xcolor}
\usepackage{tikz-cd}
\usepackage{lineno,hyperref} % lineno must be loaded after amsmath to get correct line numbers in paragraphs with displayed equations
\modulolinenumbers[5]
\journal{}
\newdefinition{algorithm}{Algorithm}

\newcommand{\dimM}{m}
\newcommand{\RKHS}{\mathcal{H}}
\newcommand{\VG}{A}
\newcommand{\KVKst}{W}

\newcommand{\Nyst}{\mathcal{N}}

\newcommand{\GV}{B}
\newcommand{\symV}{\tilde{V}}
\newcommand{\EigenW}{\zeta}
\newcommand{\EigenGV}{z}
\newcommand{\EigApp}{\tilde{z}}
\newcommand{\real}{\mathbb{R}}
\newcommand{\cmplx}{\mathbb{C}}
\newcommand{\num}{\mathbb{N}}

\newcommand{\SpectMes}{E}

\newcommand{\MesW}{\mathcal{E}}
\newcommand{\MesVComp}{\tilde{E}}
\newcommand{\Dirich}{\mathcal{D}}
\newcommand{\iotaC}{\iota^{(1)}}

\DeclareMathOperator{\tr}{tr}
\DeclareMathOperator{\Real}{Re}
\DeclareMathOperator{\grad}{grad}
\DeclareMathOperator{\vol}{vol}
\DeclareMathOperator{\spn}{span}
\DeclareMathOperator{\supp}{supp}
\DeclareMathOperator{\ran}{ran}
\DeclareMathOperator{\diag}{diag}

\DeclareMathOperator{\proj}{proj}
\DeclareMathOperator{\Id}{Id}

\newcommand{\Ep}{E_{p}}
\newcommand{\Ec}{E_{c}}
\newcommand{\Hp}{H_{p}}
\newcommand{\Hc}{H_{c}}
\newcommand{\BR}{\mathcal{B}(\mathbb{R})}
\newcommand{\Borel}{\mathcal{B}}

\newcommand{\blue}{\textcolor{black}}
\newtheorem{theorem}{Theorem}
\newtheorem{prop}[theorem]{Proposition}
\newtheorem{lemma}[theorem]{Lemma}
\newtheorem{cor}[theorem]{Corollary}
\newdefinition{exmp}[theorem]{Example}
\newdefinition{rk}[theorem]{Remark}
\newdefinition{defn}[theorem]{Definition}
\newtheorem{Assumption}{Assumption}
\theoremstyle{remark}
\newtheorem*{rk*}{Remark}

\newcounter{tempEnumi}

\biboptions{sort&compress}

\begin{document}

\begin{frontmatter}
\title{Reproducing kernel Hilbert space compactification \\of unitary evolution groups}
\author[nyu]{Suddhasattwa Das}
\author[nyu]{Dimitrios Giannakis\corref{mycorrespondingauthor}}
\ead{dimitris@cims.nyu.edu}
\cortext[mycorrespondingauthor]{Corresponding author}
\author[fin]{Joanna Slawinska}

\address[nyu]{Courant Institute of Mathematical Sciences, New York University, New York, NY 10012, USA}
\address[fin]{Finnish Center for Artificial Intelligence, Department of Computer Science, University of Helsinki, Helsinki, Finland}

\begin{abstract}
    A framework for coherent pattern extraction and prediction of observables of measure-preserving, ergodic dynamical systems with both atomic and continuous spectral components is developed. This framework is based on an approximation of the generator of the system by a compact operator $W_\tau$ on a reproducing kernel Hilbert space (RKHS). A key element of this approach is that $W_\tau$ is skew-adjoint (unlike regularization approaches based on the addition of diffusion), and thus  can be characterized by a unique projection-valued measure, discrete by compactness, and an associated orthonormal basis of eigenfunctions. These eigenfunctions can be ordered in terms of a Dirichlet energy on the RKHS, and provide a notion of coherent observables under the dynamics akin to the Koopman eigenfunctions associated with the atomic part of the spectrum. In addition, the regularized generator has a well-defined Borel functional calculus allowing the construction of a unitary evolution group $\{ e^{t W_\tau} \}_{t\in\mathbb{R}}$ on the RKHS, which approximates the unitary Koopman evolution group of the original system. We establish convergence results for the spectrum and Borel functional calculus of the regularized generator to those of the original system in the limit $\tau \to 0^+$. Convergence results are also established for a data-driven formulation, where these operators are approximated using finite-rank operators obtained from observed time series. An advantage of working in spaces of observables with an RKHS structure is that one can perform pointwise evaluation and interpolation through bounded linear operators, which is not possible in $L^p$ spaces. This enables the evaluation of data-approximated eigenfunctions on previously unseen states, as well as data-driven forecasts initialized with pointwise initial data (as opposed to probability densities in $L^p$).  The pattern extraction and prediction framework is numerically applied to ergodic dynamical systems with atomic and continuous spectra, namely a quasiperiodic torus rotation, the Lorenz 63 system, and the R\"ossler system.
\end{abstract}


\end{frontmatter}
%\linenumbers
%-_-_-_-_-_-_-_-_-_-_-_-_-_-_-_-_-_-_-_-_-_-_-_-_-_-_-_-_-_-_-_-_-_-_-_-_-_-_-_-_-_-_-_-_-_-_-_-_-_-_-_-_-_-_-_-_-_-_-_-_-_-_-_-_-_-_-_-_-_-_-_-_-_-_-_-
\section{Introduction} \label{sec:intro}

Characterizing and predicting the evolution of observables of dynamical systems is an important problem in the mathematical, physical, and engineering sciences, both theoretically and from an applications standpoint. A framework that has been gaining popularity \cite{DellnitzJunge99,DellnitzEtAl00,MezicBanaszuk04,Mezic05,FroylandPadberg09,RowleyEtAl09,Schmid10,Froyland2013,FroylandEtAl14b,FroylandEtAl14,TuEtAl14,BerryEtAl15,GiannakisEtAl15,WilliamsEtAl15,BruntonEtAl17,ArbabiMezic17,Giannakis17,DasGiannakis_delay_2019,GiannakisDas_tracers,KordaMezic18,KordaEtAl18,DasGiannakis_RKHS_2018} is the operator-theoretic approach to ergodic theory \cite{BudisicEtAl12,HadjighasemEtAl17,EisnerEtAl15},  where instead of directly studying the properties of the dynamical flow on state space, one characterizes the dynamics through its action on linear spaces of observables. The two classes of operators that have been predominantly employed in these approaches are the Koopman and Perron-Frobenius (transfer) operators, which are duals to one another when defined on appropriate spaces of functions and measures, respectively. It is a remarkable fact, realized in the work of Koopman in the 1930s  \cite{Koopman31}, that the action of a general nonlinear system on such spaces can be characterized through linear evolution operators, acting on observables by composition with the flow. Thus, despite the potentially nonlinear nature of the dynamics, many relevant problems, such as coherent pattern detection, statistical prediction, and control, can be formulated as intrinsically linear problems, making the full machinery of functional analysis available to construct stable and convergent approximation techniques.

The Koopman operator $ U^t $ associated with a continuous-time, continuous flow $ \Phi^t: M \to M $ on a manifold $M$ acts on functions by composition, $ U^t f = f \circ \Phi^t $. It is a \blue{contractive} operator on the Banach space $C^0(M)$ of bounded continuous functions on $M$, and  a unitary operator on the Hilbert space $L^2(\mu) $ associated with any invariant Borel probability measure $\mu$. Our main focus will be the latter Hilbert space setting, in which $ U = \{ U^t \}_{t\in\real} $ becomes a unitary evolution group. In this setting, it is merely a matter of convention to consider Koopman operators instead of transfer operators, for the action of the transfer operator at time $t$ on densities of measures in $L^2(\mu)$ is given by the adjoint $ U^{t*} = U^{-t} $ of $U^t$. \blue{We will also assume that the invariant measure $\mu$ is ergodic.} 

In this work, we seek to address the following two broad classes of problems: 
\begin{enumerate}[(i)]
    \item \emph{Coherent pattern extraction;} that is, identification of a collection of observables in $L^2(\mu)$ having high regularity and an \blue{approximately periodic evolution under $U^t$. A precise notion of coherent observables stated in terms of Koopman eigenfunctions, or approximate Koopman eigenfunctions, will be given in \eqref{eqKoopEig} and \eqref{eqAPS_Ut}, respectively.}
    \item \emph{Prediction;} that is, approximation of $U^{t} f$ at arbitrary $t \in \real$ for a fixed observable $f\in L^2(\mu)$.
\end{enumerate}
Throughout, we require that the methods to address these problems are data-driven; i.e., they only utilize information from the values of a function $ F : M \to Y $ taking values in a data space $ Y$, sampled finitely many times along an orbit of the dynamics. 

\subsection{Spectral characterization of unitary evolution groups} 

By Stone's theorem on one-parameter unitary groups  \cite{Stone1932,Schmudgen12}, the Koopman group $U$ is completely characterized by its generator---a densely defined, skew-adjoint, unbounded operator $ V : D(V) \to L^2(\mu) $ with domain $ D(V) \subset L^2(\mu) $ and 
\[
  V f = \lim_{t \to 0} \frac{ U^t f - f }{ t }, \quad f \in D(V).
\]
In particular, associated with $V$ is a unique projection-valued measure (PVM) $ \SpectMes: \BR \to \mathcal{L}(L^2(\mu)) $ acting on the Borel $\sigma$-algebra $ \BR $ on the real line and taking values in the space $\mathcal{L}(L^2(\mu)) $ of bounded operators on $L^2(\mu)$, such that 
\begin{equation}  \label{eqn:V_Spec}
  V = \int_\real i \omega \, dE(\omega), \quad U^t = \int_{\real} e^{i\omega t} \, dE(\omega).
\end{equation}
The latter relationship expresses the Koopman operator at time $ t $ as an exponentiation of the generator, $ U^t = e^{tV} $, which can be thought of as operator-theoretic analog of the exponentiation of a skew-symmetric matrix yielding a unitary matrix. In fact, the map $ V \mapsto e^{tV}$ is an instance of the Borel functional calculus, whereby one lifts a Borel-measurable function $ Z : i\real \to \cmplx$ on the imaginary line $i \real \subset \cmplx$, to an operator-valued function 
\begin{equation}\label{eqn:BorelFunc}
    Z(V) = \int_{\real} Z(i\omega) \, dE(\omega), 
\end{equation}
acting on the skew-adjoint operator $V$ via an integral against its corresponding PVM $E$. 

The spectral representation of the unitary Koopman group can be further refined by virtue of the fact that $L^2(\mu)$ admits the $ U^t $-invariant orthogonal splitting 
\begin{equation}  \label{eqHDecomp}
  L^2(\mu) = \Hp \oplus \Hc, \quad \Hc = \Hp^\perp, 
\end{equation}
where $\Hp $ and $ \Hc $ are closed orthogonal subspaces of $L^2(\mu)$ associated with the atomic (point) and continuous components of $ E $, respectively. On these subspaces, there exist unique PVMs $ \Ep : \BR \to \mathcal{L}(\Hp)$ and $ \Ec : \BR \to \mathcal{L}(\Hc) $, respectively, where $ \Ep $ is atomic and $ \Ec $ is continuous, yielding the decomposition 
\begin{equation}
  \label{eqSpecDecomp} 
  E = \Ep \oplus \Ec.
\end{equation}   
We will refer to $ \Ep $ and $ \Ec $ as the point and continuous spectral components of $ E $, respectively. 

The subspace $ \Hp $ is the closed linear span of the eigenspaces of $ V $ (and thus of $ U^t $). Correspondingly, the atoms of $ \Ep $, i.e., the singleton sets $ \{ \omega_j \} \subset \real $ for which $ \Ep(\{\omega_j\}) \neq 0 $, contain the eigenfrequencies of the generator. In particular, for every such $ \omega_j $, $ E_p( \{ \omega_j \} ) $ is equal to the orthogonal projector to the  eigenspace of $ V $ at eigenvalue $ i \omega_j $, and  all such eigenvalues are simple by ergodicity of the invariant measure $\mu$. As a result, $ \Hp$ admits an orthonormal basis $ \{ z_j \} $ satisfying
\begin{equation}   \label{eqKoopEig}
  V z_j = i \omega_j z_j, \quad U^t z_j = e^{i \omega_j t } z_j, \quad U^t f = \sum_j e^{i \omega_j t } \langle z_j , f \rangle_{\mu} z_j, \quad \forall f\in \Hp,
\end{equation}
where $\langle \cdot,\cdot \rangle_{\mu}$ is the inner product on $L^2(\mu)$. It follows from the above that the Koopman eigenfunctions form a distinguished orthonormal basis of $ \Hp $, whose elements evolve under the dynamics by multiplication by a periodic phase factor at a distinct frequency $ \omega_j $, even if the underlying dynamical flow is nonlinear and aperiodic. In contrast, observables  $ f \in \Hc $ do not exhibit an analogous quasiperiodic evolution, and are characterized instead by a weak-mixing property (decay of correlations), typical of chaotic dynamics, 
\begin{displaymath}
    \frac{1}{t} \int_0^t \lvert \langle g, U^s f \rangle_{\mu} \rvert\, ds \xrightarrow[t\to\infty]{} 0, \quad \forall g \in L^2(\mu). 
\end{displaymath}

\subsection{Pointwise and spectral approximation techniques} 

While the two classes of pattern extraction and prediction problems listed above are obviously related by the fact that they involve the same evolution operators, in some aspects they are fairly distinct, as for the latter it is sufficient to perform pointwise (or even weak) approximations of the operators, whereas the former are fundamentally of a spectral nature. In particular, observe that a convergent approximation technique for the prediction problem can be constructed by taking advantage of the fact that $U^t $ is a bounded (and therefore continuous) linear operator, without explicit consideration of its spectral properties. That is, given an arbitrary orthonormal basis $\{ \phi_0, \phi_1, \ldots \} $ of $L^2(\mu)$ with associated orthogonal projection operators $ \Pi_L : L^2(\mu) \to \spn\{ \phi_0, \ldots, \phi_{L-1} \} $, the finite-rank operator $U^t_L = \Pi_L U^t \Pi_L $ is fully characterized by the matrix elements $U^t_{ij} = \langle \phi_i, U^t \phi_j \rangle_{\mu}$ with $0 \leq i,j \leq L-1$, and by continuity of $U^t$, the sequence of operators $U^t_L$ converges pointwise to $U^t$. Thus, if one has access to data-driven approximations $U^t_{N,ij}$ of $U^t_{ij}$ determined from $N$ measurements of $F$ taken along an orbit of the dynamics, and these approximations converge as $N\to\infty$, then, as $ L\to\infty $ and $N \gg L$, the corresponding finite-rank operators $U^t_{N,L}$ converge  pointwise to $U^t$. 

This property was employed in \cite{BerryEtAl15} in a technique called diffusion forecasting, whereby the approximate matrix elements $U^t_{N,ij}$ are evaluated in a data-driven basis constructed from samples of $F$ using the diffusion maps algorithm  (a kernel algorithm for manifold learning) \cite{CoifmanLafon06}. By spectral convergence results for kernel integral operators \cite{VonLuxburgEtAl08} and ergodicity, as $N\to\infty$, the data-driven basis functions converge to an orthonormal basis of $L^2(\mu)$ in an appropriate sense, and thus the corresponding approximate Koopman operators $U^t_{N,L}$ converge pointwise to $U^t$ as described above. In \cite{BerryEtAl15}, it was demonstrated that diffusion forecasts of observables of the Lorenz 63 (L63) system \cite{Lorenz63} have skill approaching that of ensemble forecasts using the true model, despite the fact that the Koopman group in this case has a purely continuous spectrum (except from the trivial eigenfrequency at 0). Pointwise-convergent approximation techniques for Koopman operators were also studied in \cite{KlusEtAl16,KordaMezic18}, in the context of extended dynamic mode decomposition (EDMD) algorithms \cite{WilliamsEtAl15}. However, these methods require the availability of an orthonormal basis of $L^2(\mu)$ of sufficient regularity, which, apart from special cases, is difficult to have in practice (particularly when the support of $\mu$ is an unknown, measure-zero subset of the ambient state space $M$).

Of course, this is not to say that the spectral decomposition in~\eqref{eqSpecDecomp} is irrelevant in a prediction setting, for it reveals that an orthonormal basis of $L^2(\mu)$ that splits between the invariant subspaces $\Hp$ and $\Hc$ would yield a more efficient representation of $U^t$ than an arbitrary basis. This representation could be made even more efficient by choosing the basis of $\Hp$ to be a Koopman eigenfunction basis (e.g., \cite{Giannakis17}). Still, so long as a method for approximating a basis of $L^2(\mu)$ is available, arranging for compatibility of the basis with the spectral decomposition of $U^t$ is a matter of optimizing performance rather than ensuring convergence.

In contrast, as has been recognized since the earliest techniques in this area \cite{DellnitzJunge99,DellnitzEtAl00,MezicBanaszuk04,Mezic05}, in coherent pattern extraction problems the spectral properties of the evolution operators play a crucial role from the outset. In the case of measure-preserving ergodic dynamics studied here, the Koopman eigenfunctions in~\eqref{eqKoopEig} provide a natural notion of temporally coherent observables that capture intrinsic frequencies of the dynamics. Unlike the eigenfunctions of other operators commonly used in data analysis (e.g., the covariance operators employed in the proper orthogonal decomposition \cite{AubryEtAl91}), Koopman eigenfunctions have the property of being independent of the observation map $F$, thus leading to a definition of coherence that is independent of the observation modality used to probe the system. In applications in fluid dynamics \cite{RowleyEtAl09,GiannakisEtAl18}, climate dynamics \cite{SlawinskaGiannakis16}, and many other domains, it has been found that the patterns recovered by Koopman eigenfunction analysis have high physical interpretability and ability to recover dynamically significant timescales from multiscale input data.

\subsection{Review of existing methodologies} 

Despite the attractive theoretical properties of evolution operators, the design of data-driven spectral approximation techniques that can naturally handle both point and continuous spectra, with rigorous convergence guarantees, is challenging, and several open problems remain. As an illustration of these challenges, and to place our work in context, it is worthwhile noting that besides approximating the continuous spectrum (which is obviously challenging), rigorous approximation of the atomic spectral component $ \Ep $ is also non-trivial, since, apart from the case of circle rotations, it is concentrated on a dense, countable subset of the real line. In applications, the density of the atomic part of the spectrum and the possibility of the presence of a continuous spectral component  necessitate  the use of some form of regularization to ensure well-posedness of spectral approximation schemes. In the transfer operator literature, the use of regularization techniques such as  domain restriction to function spaces where the operators are quasicompact \cite{DellnitzEtAl00}, or compactification by smoothing by kernel integral operators \cite{Froyland2013}, has been prevalent, though these methods generally require more information than the single observable time series assumed to be available here. On the other hand, many of the popular techniques in the Koopman operator literature, including the dynamic mode decomposition (DMD) \cite{Schmid10,RowleyEtAl09} and EDMD \cite{WilliamsEtAl15} do not explicitly consider regularization, and instead implicitly regularize the operators by projection onto finite-dimensional subspaces (e.g., Krylov subspaces and subspaces spanned by general dictionaries of observables). \blue{Despite the practical simplicity of this approach, controlling its asymptotic behavior as the dimension of the approximation space increases is difficult; see, e.g., Figure~\ref{figZetaTauTorus} below.} 

To our knowledge, the first spectral convergence results for EDMD \cite{ArbabiMezic17} were obtained for a variant of the framework called Hankel-matrix DMD \cite{BruntonEtAl17}, which employs dictionaries constructed by application of delay-coordinate maps \cite{SauerEtAl91} to the observation function. However, these results are based on an assumption that the observation map lies in a finite-dimensional Koopman invariant subspace (which must be necessarily a subspace of $ \Hp $); an assumption unlikely to hold in practice. This assumption is relaxed in \cite{KordaMezic18}, who establish weak spectral convergence results implied by strongly convergent  approximations of the Koopman operator derived through EDMD. This approach makes use of an a priori known orthonormal basis of $L^2(\mu)$, the availability of which is not required in Hankel-matrix DMD.

A fairly distinct class of approaches to (E)DMD perform spectral estimation for Koopman operators using harmonic analysis techniques \cite{MezicBanaszuk04,Mezic05,KordaEtAl18,DasGiannakis_RKHS_2018}. Among these, \cite{MezicBanaszuk04,Mezic05} consider a spectral decomposition of the Koopman operator closely related to~\eqref{eqSpecDecomp}, though expressed in terms of spectral measures on $S^1 $ as appropriate for unitary operators, and utilize harmonic averaging (discrete Fourier transform) techniques to estimate eigenfrequencies and the projections of the data onto Koopman eigenspaces. While this approach can theoretically recover the correct eigenfrequencies corresponding to eigenfunctions with nonzero projections onto the observation map, its asymptotic behavior in the limit of large data exhibits a highly singular dependence on the frequency employed for harmonic averaging---this hinders the construction of practical algorithms that converge to the true eigenfrequencies by examining candidate eigenfrequencies in finite sets. The method also does not address the problem of approximating the continuous spectrum, or the computation of Koopman eigenfunctions on the whole state space (as opposed to eigenfunctions computed on orbits). The latter problem was addressed in \cite{DasGiannakis_RKHS_2018}, who employed the theory of reproducing kernel Hilbert spaces (RKHSs) \cite{FerreiraMenegatto2013,PaulsenRaghupathi2016} to identify conditions for a candidate frequency $ \omega \in \real $ to be a Koopman eigenfrequency based on the RKHS norm of the corresponding Fourier function $e^{i\omega t}$ sampled on an orbit. For the frequencies meeting these criteria, they constructed pointwise-defined Koopman eigenfunctions in RKHS using out-of-sample extension techniques \cite{CoifmanLafon06b}. While this method also suffers from a singular behavior in $ \omega $, it was found to significantly outperform conventional harmonic averaging techniques, particularly in mixed-spectrum systems with non-trivial atomic and continuous spectral components simultaneously present. However, the question of approximating the continuous spectrum remains moot. RKHS-based approaches for spectral analysis of Koopman operators have also been proposed in \cite{Kawahara16,KlusEtAl18}, though these methods rely on the strong assumption that the Koopman operator maps the RKHS into itself. The latter is known to be satisfied only in special cases, such as RKHSs with flow-invariant reproducing kernels \citep[][Corollary~9]{DasGiannakis_RKHS_2018}. In \cite{KordaEtAl18}, a promising approach for estimating both the atomic and continuous parts of the spectrum  was introduced, based on spectral moment estimation techniques. This approach consistently approximates the spectral measure of the Koopman operator on the cyclic subspace associated with a given scalar-valued observable, and is also capable of identifying its atomic, absolutely continuous, and singular continuous components. However, since it operates on cyclic subspaces associated with individual observables, it is potentially challenging to extend to applications involving a high-dimensional data space $Y$, including spatiotemporal systems where the dimension of $Y$ is formally infinite.  

In \cite{GiannakisEtAl15,Giannakis17,DasGiannakis_delay_2019} a different approach was taken, focusing on approximations of the eigenvalue problem for the skew-adjoint generator $V$, as opposed to the unitary Koopman operators $U^t$, in an orthonormal basis of an invariant subspace of $\Hp$ (of possibly infinite dimension) learned from observed data via kernel algorithms \cite{BelkinNiyogi03,CoifmanLafon06,VonLuxburgEtAl08,BerryHarlim16,BerrySauer16b} as in diffusion forecasting.  A key ingredient of these techniques is a family $K_1, K_2, \ldots $ of kernel integral operators on $L^2(\mu)$ constructed from delay-coordinate-mapped data with $Q$ delays, such that, in the infinite-delay limit, $K_Q$ converges in norm to a compact integral operator $K_\infty:L^2(\mu)\to L^2(\mu)$ commuting with $U^t$ for all $t\in \real$. Because commuting operators have common eigenspaces, and the eigenspaces of compact operators at nonzero corresponding eigenvalues are finite-dimensional, the eigenfunctions of $K_\infty$ (approximated by eigenfunctions of $K_Q$ at large $Q$) provide a highly efficient basis to perform Galerkin approximation of the Koopman eigenvalue problem. In \cite{GiannakisEtAl15,Giannakis17,DasGiannakis_delay_2019}, a well-posed variational eigenvalue problem was formulated by regularizing the raw generator $V$ by the addition of a small amount of diffusion, represented by a positive-semidefinite self-adjoint operator $\Delta : D(\Delta) \to L^2(\mu)$ on a suitable domain $D(\Delta) \subset D(V) $. This leads to an advection-diffusion operator
\begin{equation}
    \label{eqL}
    L = V - \theta \Delta, \quad \theta > 0,
\end{equation}
whose eigenvalues and eigenfunctions can be computed through provably convergent Galerkin schemes based on classical approximation theory for variational eigenvalue problems \cite{BabuskaOsborn91}. The diffusion operator in~\eqref{eqL} is constructed so as to compute with $V$, so that every eigenfunction of $L$ is a Koopman eigenfunction, with eigenfrequency equal to the imaginary part of the corresponding eigenvalue. Moreover, it was shown that the variational eigenvalue problem for $ L $ can be consistently approximated \blue{from time series data acquired via a generic observation map}.

Advection-diffusion operators as in~\eqref{eqL} can, in some cases,  also provide a notion of coherent observables in the continuous spectrum subspace $\Hc$, although from this standpoint the results are arguably not very satisfactory. In particular, it follows from results in \cite{FrankeEtAl10} that if the support $X\subseteq M$ of the invariant measure $\mu$ has manifold structure, and $\Delta $ is chosen to be a Laplacian or weighted Laplacian for a suitable Riemannian metric, then the spectrum of $L$ contains only isolated eigenvalues, irrespective of the presence of continuous spectrum \cite{Giannakis17}. However, if $ V $ has a non-empty continuous spectrum, then there exists no smooth Riemannian metric whose corresponding Laplacian commutes with $ V $, meaning that $L$ is necessarily non-normal. The spectra of non-normal operators can have several undesirable, or difficult to control, properties, including extreme sensitivity to perturbations and failure to have a complete basis of eigenvectors. The behavior of $L$ is even more difficult to characterize if $X$ is not a smooth manifold, and $V$ possesses continuous spectrum.  In \cite{GiannakisEtAl15,Giannakis17,DasGiannakis_delay_2019}, these difficulties are avoided by effectively restricting $V$ to an invariant subspace of $\Hp$ through a careful choice of data-driven basis,  but this approach provides no information about the ability of the method to identify coherent observables in $\Hc$. 

Operators analogous to $L$ in \eqref{eqL}, acting on suitable spaces of distributions, have also been shown to consistently approximate the spectrum of the generator of Anosov flows \cite{Liverani05,DyatlovZworski15}, allowing, in particular, to recover Pollicott-Ruelle resonances \cite{Pollicott86,Ruelle86} in such systems through zero viscosity ($\theta \to 0$) limits. However, these approaches make extensive use of the hyperbolic structure of Anosov flows, which is not exhibited by the more general class of ergodic flows studied here. Put together, these facts motivate a different regularization approach to~\eqref{eqL} that can seamlessly handle both the point and continuous spectra of $V$, while being amenable to data-driven approximation.   
 
\subsection{Contributions of this work} 

In this paper, we propose a data-driven framework for pattern extraction and prediction in measure-preserving, ergodic dynamical systems, which retains the advantageous aspects of \cite{BerryEtAl15,GiannakisEtAl15,Giannakis17,DasGiannakis_delay_2019} through the use of kernel integral operators to provide orthonormal bases of appropriate regularity, while being naturally adapted to dynamical systems with arbitrary (pure point, mixed, or continuous) spectral characteristics. The key element of our approach is to replace the diffusion regularization in~\eqref{eqL} by a \emph{compactification} of the skew-adjoint generator $V$ of such systems (which is unbounded, and has complicated spectral behavior), mapping it to a family of compact, skew-adjoint operators $\KVKst_\tau: \RKHS_\tau \to \RKHS_\tau$, $\tau > 0 $, each acting on an RKHS $\RKHS_\tau$ of functions on the state space manifold $M$. In fact, the operators $\KVKst_\tau$ are not only compact, they are trace-class integral operators with continuous kernels. Moreover, the spaces $\RKHS_\tau$ employed in this framework are dense in $L^2(\mu)$, and have Markovian reproducing kernels. We use the unitary operator group $ \{ e^{t\KVKst_\tau} \}_{t \in\real} $ generated by $\KVKst_\tau $ as an approximation of the Koopman group $U $, and establish spectral and pointwise convergence as $\tau\to 0$ in an appropriate sense. This RKHS approach has the following advantages.
\begin{enumerate}[(i)]
    \item The fact that $\KVKst_\tau$ is skew-adjoint avoids non-normality issues, and allows decomposition of these operators in terms of unique PVMs $ \mathcal{E}_\tau : \BR \to \mathcal{L}(\RKHS_\tau) $. The existence of \blue{ $\mathcal{E}_\tau$ } allows in turn the construction of a Borel functional calculus for $\KVKst_\tau$, meaning in particular that operator exponentiation, $ e^{t\KVKst_\tau}$, is well defined. Moreover, by compactness of $\KVKst_\tau$, the measures $\mathcal E_\tau$ are purely atomic,  have bounded support, and are thus characterized by a countable set of bounded, real-valued eigenfrequencies with a corresponding orthonormal eigenbasis of $\RKHS_\tau$. \blue{The skew-adjointness of $W_\tau $ and the generator $V$ also enables the use of spectral approximation techniques based on strong convergence in a core of $V$ \cite{Oliveira2008}, which are special to skew- or self-adjoint operators and would not be available in a direct approximation of the unitary Koopman group.}   
    \item For systems that do possess nontrivial Koopman eigenfunctions, there exists a subset of the eigenfunctions of $\KVKst_\tau$ converging to them as $\tau \to 0 $. These eigenfunctions can be identified a posteriori by monitoring the growth of a Dirichlet energy functional as a function of $\tau$. Crucially, however, the eigenfunctions of $\KVKst_\tau$ provide a basis for the whole of $L^2(\mu)$, including the continuous spectrum subspace $\Hc$, that evolves under the dynamics as an approximate Koopman eigenfunction basis. 
    \item The evaluation of $e^{t\KVKst_\tau}$ in the eigenbasis of $\KVKst_\tau$ leads to a stable and efficient scheme for forecasting observables, which can be initialized with pointwise initial data in $M$. This improves upon diffusion forecasting \cite{BerryEtAl15}, as well as comparable prediction techniques operating directly on $L^2(\mu)$, which produce ``weak'' forecasts (i.e., expectation values of observables with respect to probability densities in $L^2(\mu)$). \blue{In addition, being based on an approximation of the generator, the evolution of observables under $ e^{t W_\tau}$ is of a fundamentally generative nature, in contrast with direct approximations of the action of the Koopman group on fixed target observables \cite{AlexanderGiannakis20} which would typically be of an interpolatory/discriminative nature.} 
    \item Our framework is well-suited for data-driven approximation using techniques from statistics and machine learning \cite{CuckerSmale01,VonLuxburgEtAl08,CoifmanLafon06b}. In particular, the theory of interpolation and out-of-sample extension in RKHS allows for consistent and stable approximation of  quantities of interest (e.g., the eigenfunctions of $\KVKst_\tau$ and the action of $e^{t\KVKst_\tau}$ on a prediction observable), based on data acquired on a finite trajectory in the state space $M$. 
\end{enumerate}

In our main results, Theorems \ref{thm:Markov}, \ref{thm:Main} and Corollaries~\ref{corr:APS}, \ref{thm:predic},  we prove the spectral convergence of $\KVKst_\tau$ to $V$ in an appropriate sense by defining auxiliary compact operators acting on $L^2(\mu)$. In Theorem \ref{thm:data_predic}, we give a data-driven analog of our main results, indicating how to construct finite-rank operators from finite datasets without prior knowledge of the underlying system and/or state space, and how spectral convergence still holds in an appropriate sense. 

\subsection{Plan of the paper} In Section \ref{sect:Assump}, we make our assumptions on the underlying system precise, and state our main results. This is followed by  results on compactification of operators in RKHS, Theorems \ref{thm:A}--\ref{thm:F}, in Section \ref{sect:Schemes}, which will be useful for the proofs of the main results. Before proving our main results, we also review some concepts from ergodic theory and functional analysis in Section \ref{sect:theory}. Then, in Sections~\ref{sect:proof:ABCD} and~\ref{sect:proof:EF}, we prove Theorems \ref{thm:A}--\ref{thm:D} and \ref{thm:E}, \ref{thm:F}, respectively, while Section~\ref{sect:proof:main} contains the proof of our main results. In Section \ref{sect:numerics}, we describe a data-driven method to approximate the compactified generator $\KVKst_\tau$, and establish its convergence (Theorem \ref{thm:data_predic}). In Section~\ref{sect:examples}, we present illustrative numerical examples of our framework applied to dynamical systems with both purely atomic and continuous Koopman spectra, namely a quasiperiodic rotation on a 2-torus, and the R\"ossler and L63 systems. We state our primary conclusions in Section~\ref{sect:examples}. The paper also includes an appendix on variable-bandwidth Gaussian kernels \cite{BerryHarlim16} (\ref{appVB}). Pseudocode is included in \ref{sec:algo}.    

\begin{figure}
    \centering
    \resizebox{\linewidth}{!}
    {\input{schematic_file.tex}}
    \caption{\blue{Outline of the main results and relationships between operators employed in the paper. Our focus is on approximating, in a spectral sense, the group of unitary Koopman operators $U^t$ on the $L^2(\mu)$ space of observables of an ergodic dynamical system with an invariant measure $\mu$, and establishing associated data-driven schemes. We approach this problem by approximating the Koopman generator $V$ and its associated spectral measure given in \eqref{eqn:V_Spec}. In particular, we construct a family of compact, skew-adjoint operators $\tilde{V}_\tau$ which have \emph{strong resolvent convergence} to $V$ as $\tau \to 0$,  resulting in a spectral approximation of $V$ in the sense of convergence of the corresponding spectral measures. Simultaneously, the exponentiation of $e^{t \tilde V_\tau}$ converges strongly to the Koopman operator $e^{tV}$. These results provide means of (i) extraction of coherent observables associated with the approximate point spectrum of $V$ (Corollary~\ref{corr:APS}); and (ii) prediction of observables under the dynamics (Corollary~\ref{thm:predic}). A key element of the approach is a family of compact, skew-adjoint operators $W_\tau$ on RKHSs $\RKHS_\tau$ of appropriate regularity (Theorem~\ref{thm:Markov}). These operators are constructed in Theorem~\ref{thm:B}, and are related to $\tilde{V}_\tau$ by a canonical unitary transformation $\mathcal{U}_\tau$ between $L^2(\mu)$ and $\mathcal H_\tau$. The main results on the spectral convergence of $\tilde V_\tau$ to $V$ and strong convergence of the associated spectral measures are stated in Theorem~\ref{thm:Main}.  To connect $W_\tau$ to $V$, we rely on two regularizations, $A_\tau$ (Theorem~\ref{thm:A}) and $B_\tau$ (Theorem~\ref{thm:C}), of $V$. Theorems~\ref{thm:D} and~\ref{thm:E} establish commutation and spectral relationships between $\tilde{V}_\tau$, $W_\tau$, $A_\tau$, and $B_\tau$, followed by spectral convergence results for these operators in Theorem~\ref{thm:F}. Theorem~\ref{thm:data_predic} and Corollary~\ref{corPredic} describe data-driven analogs of the main approximation results utilizing time series data. These approaches are illustrated with numerical examples in Section~\ref{sect:numerics}.}}
	\label{fig:outline}
\end{figure}

%-_-_-_-_-_-_-_-_-_-_-_-_-_-_-_-_-_-_-_-_-_-_-_-_-_-_-_-_-_-_-_-_-_-_-_-_-_-_-_-_-_-_-_-_-_-_-_-_-_-_-_-_-_-_-_-_-_-_-_-_-_-_-_-_-_-_-_-_-_-_-_-_-_-_-_-
\section{Main results}\label{sect:Assump}

All of our main results will use the following standing assumptions and notations.
\begin{Assumption}\label{assump:A1}
    $\Phi^t :\mathcal{M} \to \mathcal{M}$, $t\in\real$, is a continuous-time, continuous flow on a metric space $\mathcal{M}$. There exists a forward-invariant, $\dimM$-dimensional, $C^r$, compact, connected manifold $M \subseteq \mathcal{M}$, such that the restricted flow map $\Phi^t\rvert_{M}$ is also $C^r$. $X\subseteq M$ is a compact invariant set, supporting an ergodic, invariant Borel probability measure $\mu$. 
\end{Assumption}

This assumption is met by many dynamical systems encountered in applications, including ergodic flows on compact manifolds with regular invariant measures (in which case $\mathcal{M} = M = X $), certain dissipative ordinary differential equations on noncompact manifolds (e.g., the L63 system \cite{Lorenz63}, where $\mathcal{M}=\real^3$, $M$ is an appropriate absorbing ball \cite{LawEtAl15}, and $X$ a fractal attractor \cite{Tucker99}), and certain dissipative partial equations with inertial manifolds \cite{ConstantinEtAl1989} (where $\mathcal{M}$ is an infinite-dimensional function space).  

In what follows, we seek to compactify the generator $V$, whose action is similar to that of a differentiation operator along the trajectories of the flow. Intuitively, one way of achieving this is to compose $V$ with appropriate smoothing operators. To that end, we will employ kernel integral operators associated with RKHSs. 

\paragraph{Kernels and their associated integral operators} In the context of interest here, a kernel will be a continuous function $k:M \times M\to\cmplx$, which can be thought of as a measure of similarity or correlation between pairs of points in $M$. Associated with every kernel $k $ and every finite, compactly supported Borel measure $ \nu $ (e.g., the invariant measure $\mu$) is an integral operator $K :L^2(\nu) \to C^0(M)$, acting on  $f \in L^2(\nu) $ as
\begin{equation} \label{eqn:KOp}
Kf := \int_M k(\cdot,y)f(y) \, d\nu(y).
\end{equation}
If, in addition, $k$ lies in $ C^r(M\times M)$, then $K$ imparts this smoothness to $Kf$, i.e., $Kf\in C^r(M)$. Note that the compactness of $ \supp(\nu) $ is important for this conclusion to hold. The kernel $ k $ is said to be Hermitian if $k(x,y) = k^*(y,x)$ for all $x,y\in M$. It is called positive-definite if for every sequence of distinct points $x_1,\ldots,x_n\in M$ the $n \times n $ kernel matrix $ \bm K = [ k(x_i,x_j ) ] $ is non-negative, and \emph{strictly} positive-definite if $ \bm K$ is \blue{strictly positive-definite}. Clearly, every real, Hermitian kernel is symmetric, i.e., $ k(x,y)=k(y,x) $ for all $x,y\in M$.  

Aside from inducing an operator mapping into $C^r(M) $, a kernel $k $ also induces an operator $ G = \iota K $ on $L^2(\nu)$, where $ \iota : C^0(M) \to L^2(\nu) $ is the canonical $L^2$ inclusion map on continuous functions. The operator $ G $ is  Hilbert-Schmidt, and thus compact and of finite trace. In particular, its Hilbert-Schmidt norm and trace are given by
\begin{equation}\label{eqn:HilSchm}
 \lVert G \rVert_{\text{HS}} := \sqrt{ \tr( G^* G ) } =  \lVert k \rVert_{L^2(\nu\times\nu)}, \quad \tr G = \int_M k(x,x) \, d\nu(x) ,
\end{equation}
respectively. Moreover, if $k$ is Hermitian, $G$ is self-adjoint, and there exists an orthonormal basis of $L^2(\mu)$ consisting of its eigenfunctions. Let $ X_\nu$ denote the support of $ \nu$. A kernel $k$ will be called $L^2(\nu)$-positive and $L^2(\nu)$-strictly-positive if $ G \geq 0 $ and $ G > 0 $, respectively; in those cases, $G$ is also of trace class. Note that if $k $ is (strictly) positive-definite on $ X_\nu \times X_\nu$, then it is $L^2(\nu)$- (strictly-) positive. Moreover, $k$ will be called a $L^2(\nu)$-Markov kernel if the associated integral operator $G : L^2(\nu) \to L^2(\nu) $ is Markov, i.e., (i) $G f \geq 0$ if $ f \geq 0$; (ii) $ \int_M Gf \, d\nu = \int_M f \, d\nu$, for all $f \in L^2(\nu)$; and (iii) $G f = f $ if $f$ is constant. The Markov kernel $k$ will be said to be ergodic if $Gf =f $ iff $f$ is constant. A sufficient condition for $k$ to be Markov is that $ k \geq 0 $ on $X_\nu \times X_\nu $, and $ \int_M k( x, \cdot ) \, d\nu = 1 $ for $ \nu $-a.e.\ $x \in M$. If $ k > 0 $ on $ X_\nu \times X_\nu  $, then $ k $ is ergodic.

\paragraph{Reproducing kernel Hilbert spaces} An RKHS on $M$ is a Hilbert space $\RKHS$ of complex-valued functions on $M$ with the special property that for every $x\in M$, the point-evaluation map $\delta_x:\RKHS\to\cmplx$, $\delta_x f = f( x )$, is a bounded, and thus continuous, linear functional. By the Riesz representation theorem, every RKHS has a unique reproducing kernel, i.e., a kernel $k : M \times M \to \cmplx$ such that for every $ x \in M $ the kernel section $k(x,\cdot)$ lies in $\RKHS$, and for every $f\in \RKHS$, 
\[f( x ) = \delta_x f = \langle k(x,\cdot), f \rangle_{\RKHS},\]
where $ \langle \cdot, \cdot,\rangle_\RKHS$ is the inner product of $\RKHS$, assumed conjugate-linear in the first argument. It then follows that $k$ is Hermitian. Conversely, according to the Moore-Aronszajn theorem \cite{Aronszajn1950}, given a Hermitian, positive-definite kernel $k : M \times M \to \cmplx$, there exists a unique RKHS $\RKHS$ for which $k$ is the reproducing kernel. Moreover, the range of $K $ from~\eqref{eqn:KOp} lies in $\RKHS $, so we can view $K$ as an operator $K:L^2(\nu) \to \RKHS $ between Hilbert spaces. With this definition, $K$ is compact, and the adjoint operator $ K^*: \RKHS \to L^2(\nu)$ maps $ f \in \RKHS$ into its $L^2(\nu)$ equivalence class, i.e., $ K^* = \iota\rvert_\RKHS$ and $ G = K^*K$. For any compact subset $S\subseteq M$, one can similarly define $\RKHS(S)$ to be the RKHS induced on $S$ by the kernel $k|_{S\times S}$. In fact, upon restriction to the support $X_\nu$, the range of $K$ is a dense subspace of $\RKHS( X_\nu )$. This implies that every function in $\RKHS(X_\nu)$ has a unique extension to a function in $\RKHS$, lying in the closed subspace $\mathcal{K}:= \overline{\ran K} \subseteq \RKHS $.

\paragraph{Nystr\"om extension} Let  $\RKHS$ be an RKHS on $M$ with reproducing kernel $k$. Then, the Nystr\"om extension operator $\Nyst : D(\Nyst) \to \RKHS$ acts on a subspace $D(\Nyst)$ of $L^2(\nu)$, mapping each element $f$ in its domain to a function $\Nyst f \in \RKHS$, such that $\Nyst f $ lies in the same $L^2(\nu)$ equivalence class as $f$. In other words, $ \Nyst f(x) = f(x) $ for $ \nu $-a.e.\ $ x \in M$, and $K^*\Nyst$ is the identity on $ D(\Nyst) $. It can also be shown that $D(\Nyst) = \ran K^*$, $\ran \Nyst = \mathcal{K}$, and $ \Nyst K^*$ is the identity on $\mathcal{K}$. Moreover, if $k $ is $L^2(\nu)$-strictly-positive, then $D(\Nyst)$ is a dense subspace of $L^2(\nu)$. In fact, $D(\Nyst)$ can be endowed with the structure of a Hilbert space, equipped with the inner product $\langle f, g \rangle_{\Nyst} = \langle \Nyst f, \Nyst g \rangle_{\RKHS}$. If $k$ is $L^2(\nu)$-strictly-positive and Markov ergodic, this space behaves in many ways analogously to a Sobolev space on a compact Riemannian manifold. In particular, equipped with this inner product, $D(\Nyst)$ embeds compactly into $L^2(\nu)$, and $ \lVert f \rVert_{\Nyst} \geq \lVert f \rVert_{L^2(\nu)}$ with equality iff $f$ is constant. Moreover, the $ \lVert \cdot \rVert_{\Nyst}$ norm induces a \emph{Dirichlet energy functional} $\Dirich : D(\Nyst) \to \mathbb{R}$, 
    \begin{equation}\label{eqDirichlet}
        \Dirich(f) =  \frac{\lVert f \rVert^2_{\Nyst}}{ \lVert f \rVert^2_{L^2(\nu)}} - 1, \quad \forall f \in D(\mathcal{N}) \setminus \{ 0 \}, \quad \text{and} \quad \mathcal{D}(0) = 0,
    \end{equation}
where $ \Dirich(f)$ is non-negative and vanishes iff $f$ is constant by $L^2(\nu)$-Markovianity and ergodicity of $k$. Intuitively, $ \Dirich$ can be interpreted as a measure of ``roughness'' of functions in $D(\Nyst)$, which vanishes for constant functions, and is large for functions that project strongly to the eigenfunctions of $G$ with small corresponding eigenvalues. We will give a precise constructive definition of $ \Nyst $, and discuss its properties, in Section~\ref{sect:theory}. 

The following assumption specifies our nominal requirements on kernels pertaining to regularity and existence of an associated RKHS. 

\begin{Assumption}\label{assump:A2}
$p:M\times M\to\real$ is a $C^r$, symmetric, positive-definite kernel, and $\nu$ a Borel probability measure with compact support $ X_\nu \subseteq M$. Moreover, $p$ is $L^2(\nu)$-strictly-positive and Markov ergodic.
\end{Assumption}

We will later describe how kernels satisfying Assumption \ref{assump:A2} can easily be constructed from symmetric, positive-definite, positive-valued $C^r$ kernels using the bistochastic kernel normalization technique proposed in \cite{CoifmanHirn2013}. It should be noted that many of our results will require $r = 1$ differentiability class in Assumptions~\ref{assump:A1} and~\ref{assump:A2}, but in some cases that requirement can be relaxed to $ r = 0 $. 

\paragraph{One-parameter kernel families} Let $P : L^2(\nu) \to \RKHS $ be the integral operator associated with a kernel $ p $ satisfying Assumption~\ref{assump:A2}, taking values in the corresponding RKHS $\RKHS$. The associated operator $ G = P^*P $ on $L^2(\nu) $ has positive eigenvalues, which can be ordered as $1 = \lambda_0 > \lambda_1 \geq \ldots$. Given a real, orthonormal basis $ \{ \phi_0, \phi_1, \ldots \} $ of $ L^2(\nu) $ consisting of corresponding eigenfunctions, the set  $ \{ \psi_0, \psi_1, \ldots \} $ with $\psi_j = \lambda_j^{-1/2} P \phi_j$ is an orthonormal basis of  $ \overline{ \ran P } \subseteq \RKHS$, and the restrictions of these functions to $X_\nu$ form an orthonormal basis of $ \RKHS(X_\nu)$. Defining
\begin{equation}\label{eqn:def:schemeIII}
\lambda_{\tau,j} := \exp \left( \tau (1- \lambda_j^{-1}) \right), \quad \psi_{\tau,j} := \sqrt{ \lambda_{\tau,j} / \lambda_j } \; \psi_j, \quad p_\tau(x,y) := \sum_{j=0}^{\infty} \psi_{\tau,j}(x) \psi_{\tau,j} (y),
\end{equation}
where $ \tau > 0 $, and $x,y $ are arbitrary points in $M$, the following theorem establishes the existence of a one-parameter family of RKHSs, indexed by $ \tau $, and an associated Markov semigroup on $L^2(\nu)$. 

\begin{theorem}[Markov kernels] \label{thm:Markov}
    Let Assumption \ref{assump:A2} hold. Then, for every $ \tau > 0 $, the series expansion for $p_\tau(x,y)$ in \eqref{eqn:def:schemeIII} converges in $C^r(M\times M) $ norm to a $C^r$, symmetric function. Moreover, the following hold:
\begin{enumerate}[(i)]
    \item For every $ \tau > 0 $, $p_\tau $ is a positive-definite kernel on $M$. In addition, it is $L^2(\nu)$-strictly-positive and Markov ergodic.
    \item For every $ \tau >0 $, the RKHS $\RKHS_\tau$ associated with $ p_\tau $ lies dense in $L^2(\nu)$, and for every $0<\tau_1<\tau_2$, the inclusions  $ \RKHS_{\tau_2} \subseteq \RKHS_{\tau_1} \subseteq \RKHS $ hold. Moreover,  $\{ \psi_{\tau,0}, \psi_{\tau,1}, \ldots \}$ is an orthonormal basis of  $\RKHS_\tau$. 
    \item Define $ G_0 := \Id_{L^2(\nu)} $ and $ G_\tau = P^*_\tau P_\tau $, where $ P_\tau : L^2(\nu) \to \RKHS_\tau $ is the integral operator associated with $p_\tau$. Then, the family $ \{ G_\tau \}_{\tau \geq 0 } $ forms a strongly continuous, self-adjoint Markov semigroup. 
\end{enumerate}
\end{theorem}

\begin{rk*}
    Theorem~\ref{thm:Markov} is independent of the dynamical system in Assumption~\ref{assump:A1}. It is a general RKHS result, allowing one to employ basis functions for the RKHS $ \RKHS(X_\nu)$, restricted on the support of $ \nu$, to construct a family of RKHSs $\RKHS_\tau$ on the entire compact manifold $M$. In particular, $\ran P_\tau$ is a dense subspace of $\RKHS_\tau$, while $ \ran P$ is not necessarily dense in $\RKHS$. 
\end{rk*}

The semigroup structure of the family $\{ G_\tau \}_{\tau\geq0}$ in Theorem~\ref{thm:Markov}(iii) implies, in particular, that for every $\tau_1, \tau_2 \geq 0$, $G_{\tau_1+\tau_2} = G_{\tau_1} G_{\tau_2}$. Moreover, strong continuity is equivalent to pointwise convergence of $ G_\tau $ to the identity operator as $\tau \to 0^+$. These two properties, as well as the Markov ergodic property, will all be important in our compactification schemes for the Koopman generator, presented in Theorem~\ref{thm:Main} and Section~\ref{sect:Schemes} below. The measure $\nu$ will now be set to the invariant measure $\mu$. In what follows, $\Nyst_\tau: D(\Nyst_\tau) \to \RKHS_\tau$ will be the Nystr\"om operator associated with $\RKHS_\tau$. We also let $ H_\infty = \bigcap_{\tau>0} D(\Nyst_\tau)$ be the dense subspace of $L^2(\mu) $ whose elements have $\RKHS_\tau$ representatives for every $ \tau > 0 $. Note that $H_\infty$ is dense since it contains all finite linear combinations of the $ \phi_j$. Similarly, setting $ \RKHS_\infty = \bigcap_{\tau>0} \RKHS_\tau$, it follows that $\RKHS_\infty(X)$ is a dense subspace of $\RKHS(X)$. In addition, we will be making use of the polar decomposition of $ P_\tau$. The latter can be shown to take the form
\begin{equation}\label{eqn:def:polar}
P_\tau = \mathcal{U}_\tau G_\tau^{1/2}, 
\end{equation}
where $\mathcal{U}_\tau : L^2(\mu) \to \RKHS_\tau$ is the unitary operator such that $\mathcal{U}_\tau \phi_j = \psi_{j,\tau} $ for all pairs $(\phi_j, \psi_{j,\tau})$ from~\eqref{eqn:def:schemeIII}. Given a Borel-measurable function $ Z : i\real \to \cmplx$ and a densely-defined skew-adjoint operator $ T $, $ Z(T) $ will denote the operator-valued function obtained through the Borel functional calculus as in Section \ref{sec:intro}. For every set $\Omega\subset\cmplx$, $\partial \Omega$ will denote its boundary.
\begin{theorem}[Main theorem] \label{thm:Main}
Under Assumptions~\ref{assump:A1}, \ref{assump:A2} with $r=1$, and the definitions in~\eqref{eqn:def:schemeIII}, the following hold for every $ \tau > 0$:
\begin{enumerate}[(i)]
    \item The operator $\KVKst_\tau :=  P_\tau  V P^*_\tau : \RKHS_\tau \to \RKHS_\tau$ is a well-defined, skew-adjoint, real integral operator of trace class.
    \item The operator $G_\tau V : D(V)\to L^2(\mu)$ extends to a trace class integral operator $\GV_\tau:L^2(\mu)\to L^2(\mu)$. Moreover, the restriction of $B_\tau$ to the dense subspace $ D(\Nyst_\tau) \subseteq D(V) $ coincides with the operator $P_\tau^* \KVKst_\tau \Nyst_\tau$. 
    \item The operators $\GV_\tau$ and $\KVKst_\tau$ have the same spectra, including multiplicities of eigenvalues. Moreover, there exists a unique, purely atomic PVM $\MesW_\tau : \Borel(\real) \to \mathcal{L}(\RKHS_\tau)$, such that $ W_\tau = \int_{\real} i \omega \, d\MesW_\tau(\omega)$.
    \setcounter{tempEnumi}{\value{enumi}}
\end{enumerate}
In addition, as $\tau \to 0^+$:
\begin{enumerate}[(i)]
    \setcounter{enumi}{\value{tempEnumi}}
\item For every bounded, Borel-measurable set $\Omega\subset\real$ such that $E(\partial \Omega) = 0$,  $ P^*_\tau \MesW_\tau(\Omega) \Nyst_\tau$ and $\mathcal{U}_\tau^*\MesW_\tau(\Omega)\mathcal{U}_\tau$ converge to $E(\Omega)$, in the strong operator topologies of $ H_\infty $ and $L^2(\mu) $, respectively. 
\item For every bounded continuous function $Z:i\real\to\cmplx$,  $P_\tau^* Z(\KVKst_\tau) \Nyst_\tau$ and $ \mathcal{U}^*_\tau Z(W_\tau) \mathcal{U}_\tau$ converge to $Z(V)$, in the strong operator topologies of $H_\infty$ and $L^2(\mu)$, respectively.
\item For every holomorphic function $Z : D(Z)\to\cmplx$, with $i \real \subset D(Z) \subseteq \mathbb{C}$ and $Z \rvert_{i \real}$ bounded, $Z(B_\tau)$ converges strongly to $Z(V)$ on $L^2(\mu)$. 
\item For every element $i\omega$ of the spectrum of the generator $V$, there exists a continuous curve $ \tau \to \omega_\tau$ such that $ i\omega_\tau $ is an eigenvalue of $\GV_\tau$ and $\KVKst_\tau$, and $\lim_{\tau\to 0^+} \omega_\tau = \omega$.
\end{enumerate}
\end{theorem}

The skew-adjoint operator $W_\tau$ from Theorem~\ref{thm:Main} can be viewed as a compact approximation to the generator $V$.  This approximation has a number of advantages for both coherent extraction and prediction. First, although $V$ is unbounded and could exhibit complex spectral behavior (see Section~\ref{sec:intro}), $W_\tau$ has a complete orthonormal basis of eigenfunctions, which are $C^1$ functions lying in $ \RKHS_\tau$. This suggests that the eigenfunctions of $W_\tau$ are good candidates for coherent observables of high regularity, which are well defined for systems with general spectral characteristics. Moreover, the discrete spectra of compact, skew-adjoint operators can be used to construct and approximate to any degree of accuracy the Borel functional calculi of these operators, and in particular perform forecasting through exponentiation of $W_\tau$. The eigenvalues and eigenfunctions of the smoothing operators $P_\tau$ employed in the construction of $W_\tau$ can also be easily derived from those of $P$ with little computational overhead. In Corollaries~\ref{corr:APS} and~\ref{thm:predic} below, we make precise the utility of $W_\tau$ for the purposes of coherent pattern extraction and forecasting, respectively. See Figure~\ref{fig:SpecTau} for an illustration of the dependence of the spectrum of $W_\tau $ on $ \tau $ for dynamical systems with point and continuous Koopman spectra. 

\paragraph{Approximate point spectrum} Given $t\in\real$ and $\epsilon > 0$, a complex number $\gamma$ is said to lie in the $\epsilon$-approximate point spectrum  of  $U^t$ if there exists a nonzero $ f \in L^2(\mu)$ such that 
\begin{equation} \label{eqAPS_Ut}
\lVert U^t f - \gamma f \rVert_{L^2(\mu)} < \epsilon \lVert f \rVert_{L^2(\mu)}. 
\end{equation}
Such observables $f$ (which include Koopman eigenfunctions as special cases), satisfying~\eqref{eqAPS_Ut} for small $\epsilon$ and $t$ lying in a given time interval, exhibit a form of dynamical coherence, as they evolve approximately as Koopman eigenfunctions over that time interval. We will refer to $(\gamma, f) $ satisfying~\eqref{eqAPS_Ut} as an $\epsilon$-approximate eigenpair of $U^t$. A discussion on how the $\epsilon$-approximate point spectrum varies with $\epsilon$, and its relation to the spectrum, in the context of a general, closed, unbounded operator, can be found in Section~\ref{sect:theory}. The following corollary of Theorem~\ref{thm:Main} establishes that the eigenvalues of $W_\tau$ corresponding to eigenfunctions that satisfy certain Dirichlet energy criteria, can be used to identify points in the $\epsilon$-approximate point spectrum of the Koopman operator at any $\epsilon>0$.   In what follows,  $ \Dirich: D(\Nyst) \to \real$ will denote the Dirichlet energy from~\eqref{eqDirichlet}, induced on $L^2(\mu)$ by the kernel $p$ in Assumption~\ref{assump:A2}.  We also introduce the function $R : \real_+ \times \real_+ \to \real $, defined as 
\[ R(\epsilon,\tau) :=  \sup \{ T >0 : \lVert (U^t - e^{t \GV_\tau} ) P^* \rVert < \epsilon, \; \forall t\in [-T,T] \}.  \]
Here, the norm of $(U^t - e^{t \GV_\tau} ) P^*$ is taken as an operator from $\RKHS$ into $L^2(\mu)$. We will later show in Proposition \ref{prop:AppEigen} that for every $\epsilon>0$, $R(\epsilon,\tau) $ diverges as $\tau \to 0^+$.
    
\begin{cor}[Coherent observables] \label{corr:APS} 
    Let $(i\omega_\tau,\zeta_\tau )$ be an eigenpair of $W_\tau$. Then, $ ( e^{i\omega_\tau t}, \tilde z_\tau )$, with $ \tilde z_\tau = P^*_\tau \zeta_\tau $, is an $\epsilon$-approximate eigenpair of $U^t$ for all $ t \in ( - T( \epsilon,\tau ), T( \epsilon, \tau ) ) $, where
    \begin{displaymath}
         T(\epsilon,\tau ) = R(\epsilon,\tau ) / \sqrt{ \mathcal{D}(\tilde z_\tau) + 1 }.
    \end{displaymath}
In addition, the following hold:
\begin{enumerate}[(i)]
	\item If $ \lim_{\tau\to0^+}\omega_\tau =: \omega $ exists, and $T(\epsilon,\tau)$ diverges as $\tau \to 0^+$ for every $\epsilon > 0$, then $ i \omega $ is an element of the spectrum of $V$.
        \item If $ \lim_{\tau\to0^+}\omega_\tau =: \omega $ exists, and $\Dirich(\tilde z_\tau)$ is bounded as $ \tau \to 0^+$, then $i \omega$ is an eigenvalue of $V $. Moreover, the sequence $\tilde z_\tau$ converges to the eigenspace of $V$ corresponding to $ i \omega $.
\end{enumerate}
\end{cor}

\begin{rk*} An important consideration in spectral approximation techniques is to identify and/or control the occurrence of \emph{spectral pollution} \cite{Bogli18}, i.e., eigenvalues $i\omega_\tau$ of the approximating operators $W_\tau$ converging to points which do not lie in the spectrum of $V$. Corollary~\ref{corr:APS} establishes that the regularity of the corresponding eigenfunctions $\zeta_\tau$, as measured through the Dirichlet energy functional associated with the RKHS $\RKHS$, provides a useful a posteriori criterion for identifying spectral pollution.
\end{rk*}

Turning now to forecasting, let $ \{ i \omega_{\tau,0}, i \omega_{\tau,1}, \ldots \} $ be the set of eigenvalues of $ W_\tau$ Note that since $W_\tau$ is a compact, skew-adjoint real operator, the $ i\omega_{j,\tau}$ occur in complex-conjugate pairs, and 0 is the only accumulation point of the sequence $ \omega_{\tau,0}, \omega_{\tau,1}, \ldots $. Let also $ \{ \EigenW_{\tau,0}, \EigenW_{\tau,1}, \ldots \} $ be an orthonormal basis of $ \RKHS_\tau$ consisting of corresponding eigenfunctions. The following is a corollary of Theorem~\ref{thm:Main}, which shows that the evolution of an observable in $L^2(\mu)$ under $U^t$ can be evaluated to any degree of accuracy by evolution of an approximating observable in $ \RKHS_\infty$ under $e^{t W_\tau}$. 

\begin{cor}[Prediction]\label{thm:predic} 
    For every $ \tau > 0 $, $W_\tau $ generates a norm-continuous group of unitary operators $e^{tW_\tau} : \RKHS_\tau \to \RKHS_\tau$, $t \in \real$. Moreover, for any observable $f\in L^2(\mu)$, error bound $\epsilon>0$, and compact set $ \mathcal{T} \subset \mathbb{R} $, there exists $ \blue{ f_\epsilon } \in \RKHS_\infty$ (independent of $\mathcal{T}$) and $ \tau_0>0$,  such that for every $\tau \in (0,\tau_0)$ and $ t \in \mathcal{T} $,
\[ \left\| U^t f -  P_\tau^* e^{tW_\tau} f_\epsilon \right\|_{L^2(\mu)} < \epsilon, \quad e^{tW_\tau} f_\epsilon = \sum_{j=0}^{\infty} e^{t i \omega_{\tau,j}} \langle \EigenW_{\tau,j}, f_\epsilon \rangle_{\RKHS_\tau} \EigenW_{\tau,j}. \]
\end{cor}

\begin{rk*} The function $e^{tW_\tau} f_\epsilon$ lies in $\RKHS_\tau$, and is therefore a continuous function which we employ to predict the evolution of the observable $f$ under $U^t$. Corollary \ref{thm:predic} suggests that to obtain this function, we first regularize $f$ by approximating it by a function $f_\epsilon \in \RKHS_\infty$, and then invoke the functional calculus for the compact operator $\KVKst_\tau$ to evolve $f_{\epsilon}$ as an approximation of $U^tf$. Note that analogous error bounds to that in Corollary~\ref{thm:predic} can be obtained for operator-valued functions $Z(V)$  of the generator other than the exponential functions, $ Z(V) =  e^{tV} = U^t$.  A constructive procedure for obtaining the forecast function in a data-driven setting will be described in Section \ref{sect:numerics}.
\end{rk*}

%-_-_-_-_-_-_-_-_-_-_-_-_-_-_-_-_-_-_-_-_-_-_-_-_-_-_-_-_-_-_-_-_-_-_-_-_-_-_-_-_-_-_-_-_-_-_-_-_-_-_-_-_-_-_-_-_-_-_-_-_-_-_-_-_-_-_-_-_-_-_-_-_-_-_-_-
\section{Compactification schemes for the generator}\label{sect:Schemes}

In this section, we lay out various schemes for obtaining compact operators by composing the generator $V$ with operators derived from kernels. These schemes are of independent interest, as they are applicable, with appropriate modifications, to more general classes of unbounded, skew- of self-adjoint operators obtained by extension of differentiation operators. In some cases, the following weaker analog of Assumption~\ref{assump:A2} will be sufficient.
\begin{Assumption} \label{assump:K}
$k:M\times M\to \real$ is a $C^1$, symmetric positive-definite kernel.
\end{Assumption}
Given the RKHS $\RKHS \subset C^1(M)$ associated with $k $ from Assumption~\ref{assump:K}, and the corresponding integral operators $K : L^2(\mu) \to \RKHS$, $G=K^*K:L^2(\mu) \to L^2(\mu) $, and closed subspace $\mathcal{K} = \overline{\ran K} \subseteq \RKHS$,  we begin by formally introducing the operators $ A: L^2(\mu) \to L^2(\mu)$ and $W : \RKHS \to \RKHS$, where
\begin{equation} \label{eqn:similarity_trnsfrm}
\VG := V G = V K^* K, \quad \KVKst := K V K^*.
\end{equation}
Note that it is not necessarily the case that these operators are well defined, for the ranges of $G$ and $K^*$ may lie outside of the domain of $V$. Nevertheless, as the following two theorems establish,  $\VG$ and $\KVKst$ are well-defined, and in fact compact, operators.

\begin{theorem}[Pre-smoothing]\label{thm:A}
Let Assumptions \ref{assump:A1} and~\ref{assump:K} hold, and define $k':M\times M \to \real$ as the $C^0$ kernel with \blue{$k'(x,y) := \lim_{t\to 0}(k(\Phi^t(x),y) - k(x,y) )/t$}. Then:
\begin{enumerate}[(i)]
\item The range of $G$ lies in the domain of $V$. 
\item The operator $\VG$ from \eqref{eqn:similarity_trnsfrm} is a well-defined, Hilbert-Schmidt integral operator on $L^2(\mu)$ with kernel $k' $, and thus bounded in operator norm by 
\begin{displaymath}
\lVert \VG \rVert \leq \lVert \VG \rVert_\text{HS} = \| k' \|_{L^2(\mu\times\mu)} \leq \| k' \|_{C^0(X\times X)}.
\end{displaymath}
\item $\VG$ is equal to the negative adjoint, $ -(GV)^* $, of the densely defined operator $GV:D(V) \to L^2(\mu) $.
\end{enumerate}
\end{theorem}

\begin{rk*} As stated in Section~\ref{sec:intro}, $V$ is an unbounded operator, whose domain is a strict subspace of $L^2(\mu)$. Theorem \ref{thm:A} thus shows that if we regularize this operator by first applying the smoothing operator $G$, then not only is \blue{ $\VG$ } bounded, it is also Hilbert-Schmidt, and thus compact. In essence, this property follows from the $C^1$ regularity of the kernel.
\end{rk*}

Arguably, the regularization scheme leading to $\VG$, which involves first smoothing by application of $G$, followed by application of $V$, is among the simplest and most intuitive ways of regularizing $V$. However, the resulting operator $\VG$ will generally not be skew-symmetric; in fact, apart from special cases, $\VG$ will be non-normal. Theorem \ref{thm:B} below provides an alternative regularization approach for $V$, leading to a Hilbert-Schmidt operator on $\RKHS$ which is additionally skew-adjoint. Working with this operator also takes advantage of the RKHS structure, allowing pointwise function evaluation by bounded linear functionals.

\begin{theorem}[Compactification in RKHS]\label{thm:B}
Let Assumptions \ref{assump:A1} and \ref{assump:K} hold, and define $\tilde k': M \times M \to \real $ as the $C^0$ kernel with $\tilde k'(x,y)=-k'(y,x)$. Then:
\begin{enumerate}[(i)]
\item The range of $K^*$ lies in the domain of $V$, and $VK^*:\RKHS\to L^2(\mu)$ is a bounded operator.
\item The operator $\KVKst$ from \eqref{eqn:similarity_trnsfrm} is a well-defined, Hilbert-Schmidt, skew-adjoint, real operator on $\RKHS$, with $\ran W \subseteq \mathcal{K}$	, satisfying 
\begin{displaymath}
W f = \int_M \tilde k'(\cdot,y) f(y) \, d\mu(y).
\end{displaymath}
\end{enumerate}
\end{theorem}

\begin{rk*} Because $\KVKst$ is skew-adjoint, real, and compact, it has the following properties, which we will later use.
\begin{enumerate}[(i)]
\item Its nonzero eigenvalues are purely imaginary, occur in complex-conjugate pairs, and accumulate only at zero. Moreover, there exists an orthonormal basis of $\RKHS$ consisting of corresponding eigenfunctions. 
\item It generates a norm-continuous, one-parameter group of unitary operators $e^{tW} : \RKHS \to \RKHS$, $ t \in \real $.
\end{enumerate}
\end{rk*}

In the next theorem, we connect the operators $\VG$ and $\KVKst$ through the adjoint of $\VG$. 
\begin{theorem}[Post-smoothing]\label{thm:C}
Let Assumptions \ref{assump:A1} and \ref{assump:K} hold. Then, the adjoint of $-\VG$ from \eqref{eqn:similarity_trnsfrm} is a Hilbert-Schmidt integral operator $\GV : L^2(\mu) \to L^2(\mu)$ with kernel $\tilde k'$. In addition:
\begin{enumerate}[(i)]
\item The densely-defined operator $ GV: D(V) \to L^2(\mu)$ is bounded, and $B$ is equal to its closure, $\overline{GV}:=(GV)^{**}$. Moreover, $ B$ is a closed extension of $KW\Nyst : D(\Nyst) \to L^2(\mu) $, and if the kernel $k$ is $L^2(\mu)$-strictly-positive, i.e., $D(\Nyst)$ is a dense subspace of $L^2(\mu)$, that extension is unique.
\item $\GV$ generates a norm-continuous, 1-parameter group of bounded operators $e^{t\GV}: L^2(\mu) \to L^2(\mu) $, $t\in\real$, satisfying
\[ K^* e^{t\KVKst} = e^{t\GV} K^*, \quad K^* e^{t\KVKst} \Nyst = e^{t\GV} |_{D(\Nyst)}, \quad \forall t\in\real. \]
\end{enumerate}
\end{theorem}

\begin{rk*} Because $V$ is an unbounded operator, defined on a dense subset $D(V) \subset L^2(\mu)$, the domain of $G V$ is also restricted to $D(V)$. It is therefore a non-intuitive result that a regularization of $V$ \emph{after} an application of $G$ could still result in a bounded operator that can be extended to the entire space $L^2(\mu)$.
\end{rk*}

Theorem~\ref{thm:C}(i) shows that, on the subspace $D(\Nyst) \subset L^2(\mu)$, $B$ acts by first performing Nystr\"om extension, then acting by $\KVKst$, then mapping back to $L^2(\mu)$ by inclusion via $K^*$. In other words, $B$ is a natural analog of $\KVKst$ acting on $L^2(\mu)$, though note that, unlike $ W$, $\GV$ is generally not skew-adjoint. To summarize, on the basis of Theorems~\ref{thm:A}--\ref{thm:C}, we have obtained the following sequence of operator extensions:
\begin{displaymath}
K W \Nyst \subseteq G V \subset B = \overline{GV} = (GV)^{**}.
\end{displaymath}

As our final compactification of $V$, we will construct a skew-adjoint operator $\symV$ on $L^2(\mu) $ by conjugation by a compact operator. In particular, since $G$ is positive-semidefinite, it has a square root $G^{1/2} : L^2(\mu) \to L^2(\mu) $, which is the unique positive-semidefinite operator satisfying $G^{1/2} G^{1/2} = G$. Note that by compactness of $G$, $G^{1/2}$ is compact, and its action on functions can be conveniently evaluated in an eigenbasis of $G$. Moreover, it can be verified that $ \ran G^{1/2} = \ran K^* $. In fact, the operators $ K $ and $G^{1/2} $ are related to $K$ via the polar decomposition, $ K = \mathcal{U} G^{1/2} $, where $ \mathcal{U}:L^2(\mu) \to \RKHS $ is a (uniquely defined) partial isometry with $\ran \mathcal{U} = \mathcal{K}$, analogous to $\mathcal{U}_\tau$ in \eqref{eqn:def:polar}. Note that $\mathcal{K}$ is an invariant subspace of $W$. Moreover, if the kernel $k$ is $L^2(\mu)$-strictly-positive and $\mathcal{K}=\RKHS$ (i.e., $K$ has dense range), then  $ \mathcal{U} $ becomes unitary. Using these definitions, we will show in Theorem~\ref{thm:D} below that the operator $G^{1/2} V G^{1/2}$, defined on the subspace $\{f\in L^2(\mu) : G^{1/2}f\in D(V)\}$, actually extends to a well-defined compact operator.

\begin{theorem}[Skew-adjoint compactification]\label{thm:D}
    Let Assumptions \ref{assump:A1} and \ref{assump:K} hold with $r=1$. Then, $G^{1/2} V G^{1/2}$ is a densely defined, bounded operator with a unique skew-adjoint extension to a Hilbert-Schmidt, real operator $\symV : L^2(\mu)\to L^2(\mu)$. Moreover, $\symV$ is related to the operator $\KVKst$ from Theorem~\ref{thm:B} via conjugation by the partial isometry $ \mathcal{U} $, i.e., $\symV = \mathcal{U}^* \KVKst \mathcal{U}$. In particular, if the kernel $k$ is $L^2(\mu)$-strictly-positive, then $ \tilde V $ and $W\rvert_\mathcal{K}$ are unitarily equivalent.
\end{theorem}

This completes the statement of our compactification schemes for $V$. Since these schemes are all carried out using the same kernel $k$, one might expect that the spectral properties of the compact operators $\VG$, $\GV$, $ \symV $, and $ W $, exhibit non-trivial relationships. These relationships will be made precise in Theorems~\ref{thm:E} and~\ref{thm:F} below. Hereafter, $\sigma(T)$ and $\sigma_p(T)$ will denote the spectrum and point spectrum (set of eigenvalues) of a linear operator $T$, respectively.

\begin{theorem}[Spectra of the compactified generators]\label{thm:E}
Let Assumptions \ref{assump:A1} and \ref{assump:K} hold with $r=1$, and assume further that the kernel $k$ is $L^2(\mu)$-strictly-positive. Let also $ \{ \tilde z_0, \tilde z_1, \ldots \} $ be an orthonormal basis of $L^2(\mu)$, consisting of eigenfunctions $\tilde z_j $ of $\symV $ corresponding to purely imaginary eigenvalues $ i \omega_j $. Then: 
\begin{enumerate}[(i)]
    \item $\VG$ and $\GV$  have the same eigenvalues as $ \symV $, including multiplicities. Moreover, $\sigma_p(W) =\sigma_p(\tilde V) $ (including multiplicities) if $K$ has dense range, and $\sigma_p(W)=\sigma_p(\tilde V)\cup\{0\}$ otherwise.
\setcounter{tempEnumi}{\value{enumi}}\end{enumerate}
In addition, if the kernel $k$ is $L^2(\mu)$-Markov ergodic: 
\begin{enumerate}[(i)]\setcounter{enumi}{\value{tempEnumi}}
    \item 0 is a simple eigenvalue of each of the operators $\VG$, $\GV$, $\symV$, and $W\rvert_\mathcal{K}$, corresponding to constant eigenfunctions. 
\item Every $\tilde{z}_j$ lies in the domain of $G^{-1/2}$. Moreover, the set $\{ z'_0, z'_1, \ldots \}   \}$ with $ \EigenGV'_j = G^{-1/2}\tilde{z}_j $ consists of eigenfunctions of $\VG$, corresponding to the eigenvalues $ \{ i\omega_0, i \omega_1, \ldots \} $, and forms an unconditional Schauder basis of $L^2(\mu)$.
\item The set $ \{ z_0, z_1, \ldots \} $ with $ z_j = G^{1/2} \tilde z_j $ is an unconditional Schauder basis of $L^2(\mu)$, consisting of eigenfunctions of $B$ corresponding to the same eigenvalues, $\{ i\omega_0, i\omega_1, \ldots \} $. Moreover, it is the unique dual sequence to the $\{\EigenGV_j\}$, satisfying $ \langle \EigenGV'_j, \EigenGV_l \rangle_{\mu} = \delta_{jl} $ .
\item The set $ \{ \zeta_0, \zeta_1, \ldots \} $ with $ \zeta_j = K z'_j $ is an orthonormal basis of $\mathcal{K}$ consisting of eigenfunctions $\zeta_j$ of $ W$ corresponding to the eigenvalues $ i \omega_j $. 
\item The operators $\VG$, $\GV$, $\symV $, and $ \KVKst$,  admit the representations 
\begin{gather*}
A = \sum_{j=0}^\infty i\omega_j \langle z_j, \cdot \rangle_{L^2(\mu)} z'_j, \quad B = \sum_{j=0}^\infty i\omega_j \langle z'_j, \cdot \rangle_{L^2(\mu)} z_j, \quad \tilde V = \sum_{j=0}^\infty i \omega_j \langle \tilde z_j, \cdot \rangle_{L^2(\mu)} \tilde z_j,\\
W = \sum_{j=0}^\infty i\omega_j\langle \zeta_j, \cdot \rangle_{\RKHS} \zeta_j, 
\end{gather*}
where the infinite sums for $A$ and $B$ converge strongly, and those for $\tilde V$,  and $W$ converge in Hilbert-Schmidt norm.
\end{enumerate}
\end{theorem}

\begin{rk*}
    The Markovianity assumption on the kernel was important to conclude that $A$, $B$, $\tilde V$, and $ W\rvert_{\mathcal{K}}$ have finite-dimensional nullspaces (which may not be the case for a general compact operator), allowing us to establish a one-to-one correspondence of the spectra of these operators, including eigenvalue multiplicities. 
\end{rk*}

An immediate consequence of Theorem~\ref{thm:E}, in conjunction with Theorems~\ref{thm:C} and~\ref{thm:D}, is that $\tilde V $ and $ W $ are decomposable in terms of unique PVMs $\tilde E : \Borel(\real) \to \mathcal{L}(L^2(\mu)) $ and $\mathcal{E} : \Borel(\real) \to \mathcal{L}(\RKHS)$, such that $ \tilde V = \int_{\real} i\omega \, d\tilde E(\omega) $, $  W = \int_{\real} i \omega \, d \mathcal{E}(\omega)$, and
\begin{equation}
    \label{eqSpecMeasVW}
    \MesVComp(\Omega) = \sum_{j:\omega_j\in \Omega} \langle \tilde z_j, \cdot \rangle_{L^2(\mu)} \tilde z_j, \quad \mathcal{E}(\Omega) = \sum_{j:\omega_j\in \Omega} \langle \zeta_j, \cdot \rangle_{\RKHS} \zeta_j  + 1_\Omega(0) \proj_{\mathcal{K}^\perp}, 
\end{equation}
where $ 1_\Omega$ is the indicator function on $ \Omega$, and $ \proj_{\mathcal{K}^\perp} : \RKHS \to \RKHS $ the orthogonal projection onto $ \mathcal{K}^\perp$. Moreover, $\MesVComp$ and $\mathcal{E}$ are related by conjugation by the partial isometry $\mathcal{U} : L^2(\mu) \to \RKHS $  from Theorem~\ref{thm:D},  
\begin{equation}
\label{eqSpecUnitary}
\tilde E(\Omega) = \mathcal{U}^*\MesW(\Omega)\mathcal{U}, \quad \forall \Omega \in \mathcal{B}(\real),
\end{equation}
and if $k$ is $L^2(\mu)$-strictly positive, $\tilde E(\Omega)$ and $\MesW(\Omega)\rvert_{\mathcal{K}}$ are unitarily equivalent. The compactness of $\tilde V$ and $ W$, which is reflected in the fact that $\tilde E $ and $\mathcal{E}$ are purely atomic PVMs, allows for simple expressions for the Borel functional calculi of these operators. In particular, for every Borel-measurable function $Z : i \real \to \cmplx$, we have 
\begin{align*}
    Z( \tilde V ) &= \int_{\real} Z(i\omega) \, d\tilde E(\omega) = \sum_{j=0}^\infty Z(i\omega_j) \langle \tilde z_j, \cdot \rangle_{L^2(\mu)} \tilde z_j, \\
    Z( W ) &= \int_{\real} Z(i\omega) \, d\mathcal{E}(\omega) = \sum_{j=0}^\infty Z(i\omega_j) \langle \zeta_j, \cdot \rangle_{\RKHS} \zeta_j + Z(0) \proj_{\mathcal{K}^\perp}, 
\end{align*}
with all limits taken in the strong operator topology. Note that if $K$ has dense range (as in Theorem~\ref{thm:Main}), $ \mathcal{K}^\perp $ reduces to the zero subspace, and $ \proj_{\mathcal{K}^\perp}$ vanishes in the above expressions.

In the case of $A$ and $B$, the fact that these are, in general, non-normal operators precludes the construction of associated Borel functional calculi. Nevertheless, the compactness of these operators allows one to construct their holomorphic functional calculi in a straightforward manner. Specifically, given any holomorphic function $Z: D(Z) \to \cmplx$ on an open set $D(Z) \subseteq \cmplx$ containing $\sigma(A)=\sigma(B)$, we define
\begin{displaymath}
Z(A) = \oint_\gamma Z(z)(z-A)^{-1} \, dz, \quad Z(B) = \oint_\gamma Z(z)(z-B)^{-1} \, dz,
\end{displaymath}
where $\gamma$ is a Cauchy contour in $D(Z)$ containing $\sigma(A)$ in its interior. Now, because $ \tilde V G^{1/2} = G^{1/2} V G = G^{1/2}A $, we have $A = G^{-1/2} \tilde V G^{1/2}$, and it follows from Taylor series that for any such holomorphic function $Z$, 
\begin{equation}
\label{eqZAB}
Z(A) = G^{-1/2} Z(\tilde V) G^{1/2}, \quad Z(B) = Z(-A)^* \supseteq G^{1/2} Z(\tilde V) G^{-1/2}.
\end{equation}

The results in Theorems~\ref{thm:A}--\ref{thm:E} are for compactifications based on general kernels satisfying Assumptions~\ref{assump:A1} and~\ref{assump:K} and their associated integral operators. Next, we establish spectral convergence results for one-parameter families of kernels that include the kernels $ p_\tau $ associated with the Markov semigroups in our main result, Theorem~\ref{thm:Main}. Specifically, we assume:

\begin{Assumption}\label{assump:A4}
$\{ k_{\tau}:M\times M\to\real \} $ with $ \tau > 0 $ is a one-parameter family of $C^1$, symmetric, $L^2(\mu)$-strictly-positive kernels, such that, as $ \tau \to 0^+ $, the sequence of the corresponding compact operators $ G_\tau = K_\tau^* K_\tau $ on $L^2(\mu) $ converges strongly to the identity, and the sequence of the skew-adjoint compactified generators $ \symV_\tau \supseteq G_\tau^{1/2}V G_\tau^{1/2} $ converges strongly to $V$ on the subspace $D(V^2) \subset D(V)$. 
\end{Assumption}
Let $\RKHS_\tau$ be the RKHS on $M$ with reproducing kernel $k_\tau$; $\Nyst_\tau : D(\Nyst_\tau) \to \RKHS_\tau $ be the corresponding Nystr\"om extension operator; and $H_\infty$ the $L^2(\mu)$ subspace equal to $\cap_{\tau>0} D(\Nyst_\tau)$. Define the partial isometries $\mathcal{U}_\tau : L^2(\mu) \to \RKHS_\tau$ through the polar decomposition $K_\tau = \mathcal{U}_\tau G_\tau^{1/2}$, as in Theorem~\ref{thm:D}. Note that, in general, $H_\infty$ could be the zero subspace, but contains at least constant functions if the $k_\tau$ are $L^2(\mu)$-Markov kernels. As stated in Section~\ref{sect:Assump}, if $H_\infty$ is the space associated with the kernels $p_\tau$ from~\eqref{eqn:def:schemeIII}, whose corresponding integral operators form a Markov semigroup and thus have common eigenspaces, then it is even dense in $L^2(\mu)$. With these definitions, we establish the following notion of spectral convergence for approximations of the generator $V$ by compact operators.
\begin{theorem}[Spectral convergence]\label{thm:F}
    Suppose that Assumptions \ref{assump:A1} and \ref{assump:A4} hold with $r=1$, and let $A_\tau, B_\tau, \tilde V_\tau : L^2(\mu) \to L^2(\mu)$ and $W_\tau : \RKHS_\tau \to \RKHS_\tau$ with $ \tau > 0 $, be the Hilbert-Schmidt operators from Theorems~\ref{thm:A}--\ref{thm:D}, applied with the kernels $k_\tau $ from Assumption~\ref{assump:A4}. Let also $ \tilde E_\tau$ and $\mathcal{E}_\tau $ be the PVMs associated with $ \tilde V_\tau $ and $ W_\tau $, respectively, constructed as in~\eqref{eqSpecMeasVW}. Then, as $\tau\to0^+$, the following hold:
\begin{enumerate}[(i)]
\item The operator $\GV_\tau$ converges strongly to $V$ on $D(V)$.
\item For every bounded continuous function $Z:i\real\to\cmplx$, $Z( \tilde V_\tau) $ and $ \mathcal{U}_\tau^* Z(W_\tau) \mathcal{U}_\tau$ converge strongly to $ Z(V)$ on $L^2(\mu)$.
\item For every holomorphic function $Z : D(Z)\to\cmplx$, with $i \real \subset D(Z) \subseteq \cmplx$ and $Z \rvert_{i \real}$ bounded, $Z(A_\tau)$ and $Z(B_\tau)$ converge strongly to $Z(V)$ on $L^2(\mu)$. Moreover, $ K_\tau^* Z(\KVKst_\tau) \Nyst_\tau $ converges strongly to $Z(V)$ on $H_\infty$.
\item For every bounded Borel-measurable set $\Omega\subset\real$ such that $\tilde E(\partial \Omega) =0$, $\tilde{E}_\tau(\Omega)$ and $\mathcal{U}_\tau^* \mathcal{E}_\tau(\Omega)\mathcal{U}_\tau$ converge strongly to $E(\Omega)$ on $L^2(\mu)$.
\item For every element $i\omega$ of the spectrum of $V$, there exists a sequence of eigenvalues $i\omega_\tau $ of $ A_\tau$, $B_\tau$, $\tilde V_\tau$, and $W_\tau$ converging to $ i\omega$. 
\end{enumerate}
\end{theorem}

Theorem \ref{thm:F} makes several of the statements of our main result, Theorem \ref{thm:Main}. In Section~\ref{sect:proof:ABCD}, we will prove the latter by invoking Theorems~\ref{thm:A}--\ref{thm:F} for the family of Markov kernels $p_\tau$. There, the semigroup structure of $p_\tau$ will allow us to extend the convergence result for $K^*_\tau Z(W_\tau) \Nyst_\tau$ from holomorphic functions to bounded continuous functions $Z$, and further deduce that $A_\tau$, $B_\tau$, $\tilde V_\tau$, and $W_\tau$ are of trace class.

%-_-_-_-_-_-_-_-_-_-_-_-_-_-_-_-_-_-_-_-_-_-_-_-_-_-_-_-_-_-_-_-_-_-_-_-_-_-_-_-_-_-_-_-_-_-_-_-_-_-_-_-_-_-_-_-_-_-_-_-_-_-_-_-_-_-_-_-_-_-_-_-_-_-_-_-
\section{Results from functional analysis and analysis on manifolds}\label{sect:theory}

In this section, we review some basic concepts from RKHS theory, spectral approximation of operators, and analysis on manifolds that will be useful in our proofs of the theorems stated in Sections~\ref{sect:Assump} and~\ref{sect:Schemes}. 

\subsection{\label{secReview_RKHS}Results from RKHS theory}

\paragraph{Nystr\"om extension} We begin by describing the Nystr\"om extension in RKHS. In what follows, $\RKHS$ will be an RKHS on $M$ with reproducing kernel $k$, $\nu$ an arbitrary finite Borel measure with compact support $X_\nu\subseteq M$, and $K : L^2(\nu) \to \RKHS $ the corresponding integral operator defined via~\eqref{eqn:KOp}. The Nystr\"om extension operator $\Nyst : D(\Nyst) \to \RKHS $, with $D(\Nyst)\subset L^2(\nu) $, extends elements of its domain, which are equivalence classes of functions defined up to sets of $ \nu $ measure zero, to functions in $\RKHS$, which are defined at every point in $M$ and can be pointwise evaluated by continuous linear functionals. Specifically, introducing the functions 
\begin{equation}\label{eqn:def:psi}
\psi_j = \lambda_j^{-1/2} K \phi_j, \quad j \in J, 
\end{equation}
where $ \{ \phi_0, \phi_1, \ldots \} $ is an orthonormal set in $ L^2(\nu) $ consisting of eigenfunctions of $ G = K^*K $, corresponding to strictly positive eigenvalues $\lambda_0 \geq \lambda_1 \geq \cdots $, and $ J = \{ j \in \num_0 : \lambda_j >0 \} $, we define 
\begin{equation}\label{eqn:def:nyst}
D(\Nyst) = \left\{ \sum_{j\in J} a_j \phi_j : \sum_{j\in J} |a_j|^2 / \lambda_j < \infty \right\}, \quad \Nyst \left( \sum_{j\in J} a_j \phi_j \right) := \sum_{j\in J} a_j \lambda_j^{-1/2} \psi_j . 
\end{equation}
It follows directly from these definitions that $ \{ \psi_j \}_{j\in J} $ is an orthonormal set in $\RKHS $ satisfying $ K^* \psi_j = \lambda_j^{1/2} \phi_j $, and $\Nyst$ is a closed-range, closed operator with $ D(\Nyst) = \ran K^*$ and $\ran\Nyst= \mathcal{K} := \overline{\ran K} = \overline{ \spn\{ \psi_j \}_{j \in J}} $. Moreover, $ K^* \Nyst $ and $\Nyst K^*$ reduce to the identity operators on $D(\Nyst)$ and $ \ran \Nyst$, respectively. In fact, upon restriction to $X_\nu$, $\ran \Nyst $ coincides with the RKHS $\RKHS(X_\nu)$, and $ \{ \blue{\psi_j\rvert_X } \}_{j\in J} $ forms an orthonormal basis of the latter space. If, in addition, the kernel $k$ is $L^2(\nu)$-strictly-positive, as we frequently require in this paper, then $D(\Nyst)$ is a dense subspace of $L^2(\nu)$, and $K^*$ coincides with the pseudoinverse of $\Nyst$. The latter is defined as the unique bounded operator $\Nyst^\dag : \RKHS \to L^2(\mu)$ satisfying (i) $ \ker \Nyst^\dag = \ran \Nyst^\perp $; (ii) $\overline{\ran \Nyst^\dag} = \ker \Nyst^\perp$; and (iii) $\Nyst \Nyst^\dag f = f$, for all $ f \in \ran \Nyst$. Note that we have described the Nystr\"om extension for the $L^2$ space associated with an arbitrary compactly supported Borel measure $ \nu $ since later on we will be interested in applying this procedure not only for the invariant measure $\mu$ of the system, but also for discrete sampling measures encountered in data-driven approximation schemes. 

\paragraph{Polar decomposition} A number of the results stated in Sections~\ref{sect:Assump} and~\ref{sect:Schemes} make use of the polar decomposition of kernel integral operators associated with RKHSs. We now review this construction. First, recall that the polar decomposition of a bounded linear map $T:H_1 \to H_2 $ between two Hilbert spaces $H_1 $ and $H_2$ is the unique factorization $ T = \mathcal{U} \lvert T \rvert$, where $ \lvert T \rvert = ( T^*T )^{1/2} $ is a non-negative, self-adjoint operator on $H_1$, and $\mathcal{U} : H_1 \to H_2$ is a partial isometry with $ \ker \mathcal{U}^\perp = \overline{\ran \lvert T \rvert } $. The spaces $ \ker \mathcal{U}^\perp $ and $ \ran \mathcal{U} $ are known as the initial and final spaces of the partial isometry $ \mathcal{U}$. In the case of the integral operator $K: L^2(\nu) \to \RKHS $, we have $ K = \mathcal{U} \lvert K \rvert $, where $ \lvert K \rvert = G^{1/2} $ by definition of $G = K^* K$. Moreover, it follows from the relationships $ K \phi_j = \lambda_j^{1/2} \psi_j $ and $ G^{1/2} \phi_j = \lambda_j^{1/2} \phi_j $, which hold for every $j \in J$, that $ \mathcal{ U } \phi_j = \psi_j $ for $j \in J $. Thus, the initial and final spaces of $\mathcal{U} $ are given by $ \ker\mathcal{U}^\perp = \overline{ \ran K^* } = \overline{D(\mathcal{N})} $ and $ \ran \mathcal{U} = \overline{ \ran K } = \mathcal{K} $, respectively. In addition, since $ K^* \psi_j = \lambda_j^{1/2} \phi_j = G^{1/2} \phi_j $, we can conclude that $ \ran G^{1/2} = \ran K^* $, and
\begin{equation}\label{eqNystGInv}
D(G^{-1/2}) = D(\Nyst), \quad \Nyst = \mathcal{U} G^{-1/2}, \quad K^*\mathcal{U} = G^{1/2}.
\end{equation}

\paragraph{Mercer representation} A classical result in the theory of RKHSs with continuous kernels is the Mercer theorem \cite{Mercer1909}, allowing one to represent the kernel through eigenfunctions. In the following lemma, we will state this result together with a useful integral formula for computing the trace of integral operators associated with continuous kernels.

\begin{lemma} \label{lemMercer} 
Let $\RKHS$ be an RKHS on $M$ associated with a continuous reproducing kernel $k$, and $ \nu $ a finite Borel measure with compact support $S \subseteq M$. Assume, further, the notations in~\eqref{eqn:def:psi}. Then, the following hold:
\begin{enumerate}[(i)]
\item (Mercer theorem) For every $x,y \in S$, $ k(x,y) = \sum_{j\in J} \psi_j^*(x) \psi_j(y) $, where the sum converges absolutely and uniformly with respect to $(x,y) \in S \times S$. 
\item The trace of the integral operator $G=K^*K$ is equal to $\int_M k(x,x) \, d\nu(x)$.
\end{enumerate}
\end{lemma}

\begin{proof} We will only prove Claim~(ii). For that, we use Claim~(i) to compute explicitly
\[\begin{gathered}
\int_M k(x,x) \, d\nu(x) = \int_S k(x,x) \, d\nu(x) = \int_S \sum_{j\in J} \psi^*_j(x) \psi_j(x) \, d\nu(x) = \sum_{j\in J} \int_S \lvert \psi_j(x) \rvert^2 d\nu(x) \\
= \sum_{j \in J} \int_S \lvert K^* \psi_j \rvert^2 \, d\nu = \sum_{j \in J } \int_S \lambda_j \rvert \phi_j \rvert^2 \, d\nu = \sum_{j \in J } \lambda_j \int_S \rvert \phi_j \rvert^2 \, d\nu = \sum_{j \in J } \lambda_j = \tr G.
\end{gathered}\]
The last equality on the first line follows from the absolute convergence of $ \sum_{j\in J} \lvert \psi_j(x) \rvert^2 $ to $k(x,x) $. The first equality in the second line follows from the fact that $K^*$ is the $L^2(\nu)$-inclusion operator on $\RKHS$.
\end{proof}

\paragraph{Bistochastic kernel normalization} Our main result, Theorem~\ref{thm:Main}, as well as a number of the auxiliary results in Theorem~\ref{thm:E}, require that the reproducing kernel under consideration be Markovian. However, the notion of Markovianity depends on a choice of measure (e.g., in the case of Theorems~\ref{thm:Main} and~\ref{thm:E}, the invariant measure $\mu$), which is usually either unknown, or integrals with respect to it cannot be evaluated in closed form. As a result, \blue{a common approach to building Markov kernels is to} start from a positive-valued unnormalized kernel, which can be evaluated in closed form, and then perform a normalization procedure to render it Markovian. Such kernel normalizations are widely used in manifold learning \cite{CoifmanLafon06,BerryHarlim16,BerrySauer16b}, spectral clustering \cite{VonLuxburgEtAl08}, and other applications. However, many of these approaches produce non-symmetric kernels which are not suitable for defining RKHSs. Here, we construct symmetric Markov kernels with associated RKHSs using the bistochastic normalization procedure introduced in \cite{CoifmanHirn2013}, which yields symmetric, positive-definite Markov kernels with corresponding RKHSs. The starting point for this construction is a kernel $k$ on $M $ satisfying Assumption \ref{assump:K}, and in addition, being strictly positive-valued everywhere, i.e., $ k > 0 $. Given a Borel probability measure $ \nu $ with compact support $X_\nu \subseteq M$, the kernel $k $ induces the functions $ d : M \to \real $ and $ q : M \to \real $ such that 
\begin{displaymath}
d(x) = \int_M k(x,y) \, d\nu(y), \quad q(x) = \int_M \frac{k(x,y)}{d(y)}\, d\nu(y).
\end{displaymath}
By strict positivity and $C^r$ regularity of $k$ and compactness of $X_\nu$, the functions $ d $, $ q $, $ 1/ d $, and $ 1/q $ are strictly positive and $C^r $. As a result, $ p : M \times M \to \real $, with 
\begin{equation}
\label{eqn:def:bistochastic}
p(x,y) = \int_M \frac{k(x,z) k(z,y) }{ d(x) q(z) d(y) } \, d\nu(z) 
\end{equation}
is also a $C^r $, positive-definite kernel with $ p > 0$. It then follows by construction that $ p $ is symmetric and satisfies $ \int_M p( x, \cdot ) \, d\nu = 1 $ for all $ x \in M $. That is, $ p$ is a positive-definite, symmetric, and $L^2(\nu)$-Markov ergodic kernel. In fact, if the kernel $k $ is strictly positive-definite on $X_\nu \times X_\nu$, then $p$ is also strictly positive-definite on that set, and thus is $L^2(\nu)$-strictly-positive. To verify this, note that 
\begin{displaymath}
    p(x,y) = \int_M \frac{\tilde k(x,z) \tilde k(z,y)}{ \tilde d(x) \tilde d(y)} \, d\nu(z), \quad \tilde k(x,y) = \frac{k(x,y)}{ q^{1/2}(x)q^{1/2}(y)}, \quad \tilde d(x) = \frac{d(x)}{ q^{1/2}(x)}, 
\end{displaymath}
and because $ x \mapsto \tilde d(x) $ is a strictly positive continuous function, it suffices to show that the kernel $ \tilde k_2( x, y ) = \int_M \tilde k(x,z) \tilde k(z,x) \, d\nu(z) $ is strictly positive-definite on $L^2(\nu)$. Now note that $ \tilde k $ is a strictly positive-definite kernel on $X_\nu \times X_\nu$ by strict positive-definiteness of $k$ and strict positivity of the continuous function $ x \mapsto \tilde q(x) $. Thus, in order to verify that $ \tilde k_2 $, and thus $ p$, is strictly positive-definite on $X_\nu\times X_\nu$, it suffices to show:

\begin{lemma}
Let $ \nu $ be a finite Borel measure with compact support $ X_\nu \subseteq M $, and $ k : X_\nu \times X_\nu \to \real $ a symmetric, strictly positive-definite kernel. Then, the kernel $ k_2 : X_\nu \times X_\nu \to \real $, with $ k_2(x,y) = \int_M k( x, z) k( z, y ) \, d\nu(z) $ is strictly positive-definite.
\end{lemma}

\begin{proof} We must show that for any collection of distinct points $x_0, \ldots x_{m-1} \in X_\nu $ the $ m \times m $ kernel matrix $ \bm G_2 := [ k_2(x_i,x_j) ] $ is positive definite. Defining $ \nu_m = \sum_{j=0}^{m-1} \delta_{x_j}/m$, this is equivalent to showing that the operator $ G_2 : L^2(\nu_m) \to L^2(\nu_m)$ with matrix representation $\bm G_2$ in the standard orthonormal basis of $L^2(\nu_m)$ is positive. To that end, observe that $ G_2 = ( K^*K_m)^* K^* K_m $, where $ K_m : L^2(\nu_m) \to \RKHS(X_\nu) $ and $ K : L^2(\nu) \to \RKHS(X_\nu) $ are the integral operators associated with $k $ and the measures $ \nu_m$ and $ \nu $, respectively, mapping into the RKHS $\RKHS(X_\nu)$ associated with $k$. Because $ K_m $ is an injective operator by strict positive-definiteness of $k $, and $K^*$ is injective by definition, $ K^* K_m $ is injective, and for every nonzero $f \in L^2(\nu_m)$, $ \langle f, G_2 f \rangle_{L^2(\nu_m)} = \langle K^*K_m f, K^* K_m f \rangle_{\nu} > 0 $. This shows that $ G_2 $ is positive, and thus $ k_2$ is a strictly positive-definite kernel, proving the lemma.
\end{proof}

In summary, we have established that if the kernel $k$ satisfies Assumption~\ref{assump:K}, and is also positive-valued and strictly positive-definite on the support of $ \nu$, then the bistochastic normalization procedure in~\eqref{eqn:def:bistochastic} yields a $C^r$, strictly positive definite, and thus $L^2(\nu)$-strictly-positive, Markov ergodic kernel. In particular, if it happens that $k $ is strictly positive-definite on $M\times M $, the kernel from~\eqref{eqn:def:bistochastic} is $L^2(\nu)$-strictly-positive and Markov ergodic for every compactly supported Borel probability measure $\nu$. This approach therefore provides a convenient way of constructing Markov kernels meeting the conditions of Theorem~\ref{thm:Markov}. In Section~\ref{sect:numerics}, we will employ bistochastic normalization of strictly positive-definite, positive-valued kernels to construct data-driven approximations to the Markov kernels in Theorem~\ref{thm:Markov} that converge in the limit of large data.

\subsection{Spectral approximation of operators}

\paragraph{Strong resolvent convergence} In order to prove the various spectral convergence claims made in Sections~\ref{sect:Assump} and~\ref{sect:Schemes}, we need appropriate notions of convergence of operators approximating the generator $V$ that imply spectral convergence. Clearly, because $V$ is unbounded, it is not possible to employ convergence in operator norm for that purpose. In fact, for the approximations studied here, even strong convergence on the domain of $V$ may not necessarily hold. For example, in an approximation of $V$ by $ V T_\tau $, where $ T_\tau $, $ \tau \geq 0 $, is a family of smoothing operators on $L^2(\mu) $ with $ \ran T_\tau \subseteq D(V) $, the convergence of $ T_\tau f $ to $f$ as $ \tau \to 0^+ $ does not necessarily imply that \blue{$VT_\tau f$} converges to $Vf$, as $V$ is unbounded. In the setting of unbounded, skew-adjoint operators, a weaker form of convergence, which is nevertheless sufficient to establish our spectral convergence claims, is \emph{strong resolvent convergence} \cite{Oliveira2008}.

To wit, let $T : D(T) \to H $ be a skew-adjoint operator on a Hilbert space $H$, and consider a sequence of operators $T_\tau : D(T_\tau) \to H $ indexed by a parameter $ \tau > 0 $. The sequence $T_\tau $ is said to converge to $ T $ as $ \tau \to 0^+ $ in strong resolvent sense if for every complex number $\rho$ in the resolvent set of $ T $, not lying on the imaginary line, the resolvents $ ( \rho - T_\tau )^{-1} $ converge to $ ( \rho - T )^{-1} $ strongly. 
\blue{ If $ T_\tau $ is bounded, for every quadratic polynomial $\mathfrak{p}$, $\mathfrak{p}( i T_\tau)$ is also bounded. Following \cite{BeckusBellissard16}, we say that the sequence $T_\tau$ is \emph{$ p2$-continuous} if every $ T_\tau $ is bounded, and the function $ \tau \mapsto \lVert \mathfrak{p}(iT_\tau) \rVert $ is continuous for every such $\mathfrak{p}$. }
%Following \cite{BeckusBellissard16}, we will say that the sequence $T_\tau$ is \emph{$ p2$-continuous} if every $ T_\tau $ is bounded, and the function $ \tau \mapsto \lVert\mathbb{P}_2(iT_\tau) \rVert $ is continuous for all quadratic polynomials $ \mathbb{P}_2 $ with real coefficients. 
Henceforth, when convenient, we will use the notation $ \xrightarrow{s}$ and $\xrightarrow{sr}$ to indicate strong convergence and strong resolvent convergence, respectively.

As we will see in Lemma~\ref{lem:core_SRC} below, $T_\tau \xrightarrow{s} T$ implies $T_\tau \xrightarrow{sr} T$. Further, if $T$ is bounded and the sequence $T_\tau $ is uniformly bounded in operator norm, then $ T_\tau \xrightarrow{sr} T$ implies $T_\tau \xrightarrow{s} T$ \citep[][Proposition~10.1.13]{Oliveira2008}. These facts indicate that strong resolvent convergence can be viewed as a generalization of strong convergence of bounded operators. For our purposes, the usefulness of strong resolvent convergence is that it implies the following convergence results for spectra and Borel functional calculi of skew-adjoint operators.

\begin{prop}\label{prop:SRC}
Suppose that $ T_\tau : D(T_\tau) \to H $ is a sequence of skew-adjoint operators converging in strong resolvent sense as $ \tau \to 0^+ $ to a skew-adjoint operator $ T : D(T) \to H $. Let also $ \Theta_\tau : \mathcal{B}(\real) \to \mathcal{L}(H) $ and $ \Theta : \mathcal{B}(\real) \to \mathcal{L}(H) $ be the PVMs associated with $ T_\tau$ and $T $, respectively. Then: 
\begin{enumerate}[(i)]
\item For every bounded, continuous function $Z: i \real \to \cmplx$, $Z(T_\tau)$ converges strongly to $Z(T)$.
\item Let $J\subset J'\subset i\real$ be two bounded intervals. Then, for every $f\in L^2(\mu)$, $ \limsup_{\tau\to 0^+} \|1_{J} (T_\tau) f \|_{L^2(\mu)} \leq \| 1_{J'}(T)f\|_{L^2(\mu)} $.
\item For every bounded, Borel-measurable set $\Omega \subset \real$ such that $\Theta(\partial \Omega) = 0 $, $ \Theta_\tau( \Omega ) $ converges strongly to $ \Theta( \Omega ) $.
\item For every bounded, Borel-measurable function $Z: i\real \to \cmplx$ of bounded support, $Z(T_\tau) $ converges strongly to $Z(T) $, provided that $ \Theta(S)= 0 $, where $S\subset\real $ is a closed set such that $iS $ contains the discontinuities of $Z$.
\item If $T $ is bounded, (ii) holds for every Borel-measurable set $ \Omega \subseteq \real $, and (iii) for every bounded Borel-measurable function $ Z : i \real \to \cmplx $.
\item If the operators $ T_\tau $ are compact, then for every element $ \theta \in i \real $ of the spectrum of $T$, there exists a one-parameter family $\theta_\tau \in i \real $ of eigenvalues of $T_\tau$ such that $ \lim_{\tau \to 0^+} \theta_\tau = \theta $. Moreover, if the sequence $ T_\tau$ is $p2$-continuous, the curve $ \tau \to \theta_\tau $ is continuous. 
\end{enumerate}
\end{prop}
\begin{proof} Claim~(i) is actually an equivalent characterization of strong resolvent convergence \cite[][Proposition~10.1.9]{Oliveira2008}. Claim~(v) is classical result from spectral approximation theory for normal, bounded operators, e.g., \citep[][Chapter 8, Theorem 2]{DunSchII1988}. In Claim~(vi), the existence of the family $\theta_\tau$ follows from \citep[][Corollary~10.2.2]{Oliveira2008}, in conjunction with compactness of $T_\tau$. The continuity of $\tau \mapsto \theta_\tau$ follows from \citep[][Theorem~1]{BeckusBellissard16}.

    It now remains to prove Claims~(ii)--(iv). Starting from Claim~(ii), note that a property of the Borel functional calculus for a skew-adjoint operator $T:D(T) \to H $ (more commonly stated for self-adjoint operators, e.g., \cite{kowalskispectral}) is that for any Borel-measurable function $Z : i\real \to \real $ lying in $L^\infty(i\real)$, $ Z(T) $ is a bounded self-adjoint operator. Moreover, this functional calculus preserves positivity, in the sense that if $Z$ is non-negative, then $Z(T)$ is positive-semidefinite, and as a result $Z(T) \leq Z'(T) $ whenever $Z\leq Z'$. With these properties, let $Z:i\real\to \real$ be a piecewise-linear continuous function equal to $1$ on $J$, and with support contained in $J'$. Let also $1_\Omega$ be the indicator function of any set $ \Omega$. Then, the inequalities $1^2_J \leq Z^2 \leq 1^2_{J'}$ hold everywhere in $i \real$, so for each $\tau > 0$, $1^2_{J}(T_\tau) \leq Z^2(T_\tau) \leq 1^2_{J'}(T)$. In addition, since $Z$ is continuous and bounded by Claim~(i), $Z(T_\tau)$ converges strongly to $Z(T)$. The proof of Claim~(ii) can now be completed using the following inequality:
\[\begin{split}
\limsup_{\tau\to 0^+} \| 1_{J}(T_\tau) f \|_{L^2(\mu)}^2 & = \limsup_{\tau\to 0^+} \langle 1_{J}(T_\tau) f , 1_{J}(T_\tau) f \rangle_{\mu} = \limsup_{\tau\to 0^+} \langle 1_{J}^2(T_\tau) f , f \rangle_{\mu} \\
& \leq \limsup_{\tau\to 0^+} \langle Z^2(T_\tau) f , f \rangle_{\mu} = \limsup_{\tau\to 0^+} \langle Z(T_\tau) f , Z(T_\tau) f \rangle_{\mu}\\
&= \limsup_{\tau\to 0^+} \| Z(T_\tau) f \|_{L^2(\mu)}^2 = \| Z(T) f \|_{L^2(\mu)}^2\\
& = \langle Z(T) f , Z(T) f \rangle_{\mu} = \langle Z^2(T) f , f \rangle_{\mu} \\
& \leq \langle 1_{J'}^2(T) f , f \rangle_{\mu} = \| 1_{J'}(T) f \|_{L^2(\mu)}^2 .
\end{split}\]

Next, we will prove Claim~(iii) in the case that $\Omega$ is an interval $[a,b] \subset \real$ with $\Theta(\partial\Omega) = \Theta(\{ a, b \}) = 0 $ (i.e., neither of $ia$ and $ib$ is an eigenvalue of $T$). Given any $w>0$, let $f_w : i \real \to i \real $ be a continuous function such that $f_w(i\omega)$ equals  $i\omega$ for $\omega\in i[a,b]$, equals $0$ outside $i[a-w,b+w]$, and is linear on the intervals $i[a-w,a]$ and $i[b,b+w]$. By Claim~(i), $\lim_{\tau\to 0^+} f_w(T_\tau) = f_w(T)$. Moreover, the operators $f_w(T_\tau), f_w(T)$ are bounded and skew-adjoint, and therefore, by Claim~(v), for every bounded, measurable $g:i\real\to\real$,
\begin{equation}
\label{eqGFW}
\lim_{\tau\to 0^+} (g\circ f_w)(T_\tau) = \lim_{\tau\to 0^+} g(f_w(T_\tau)) = g(f_w(T)) = (g\circ f_w)(T).
\end{equation}
Setting $g=1_{i\Omega}$, then leads to
\[ g\circ f_w = 1_{i\Omega} + 1_{J_w}, \quad J_w := [b,b+w]\cap f_w^{-1}(\Omega). \]
Thus, substituting for $ g \circ f_w $ in~\eqref{eqGFW} using the latter identity, and rearranging, we obtain
\begin{equation}\label{eqn:strnglim1}
\lim_{\tau\to 0^+} \left[ \Theta_\tau(\Omega) + \Theta_\tau(J_w) - \Theta(J_w) \right] = \Theta(\Omega), \quad \forall w>0. 
\end{equation}
Note that here we have used the fact that for any Borel set $S\subset\real$, $\Theta(S) = 1_{iS}(T)$, and a similar fact for $T_\tau$. The operator $\Theta(J_w)$ is the spectral projection onto the subspace $H_w = \ran \Theta( J_w) \subseteq H$. Since $\Theta(\partial \Omega ) = 0$ and $\cap_{w>0} J_w = \{b\}$, we have $\cap_{w>0} H_w =\{0\}$. As a result, as $w\to 0^+$, the $H_w^\bot$ form an increasing sequence of subspaces with $\cup_{w>0} H_w^\bot = H $. Thus, to prove that $\Theta_\tau(\Omega)$ converges strongly to $\Theta(\Omega)$, it is enough to verify the same claim on $H_{w_0}^\bot$ for every fixed $w_0>0$. To that end, let $w_0>0$ be fixed, and consider an arbitrary $f\in H_{w_0}^\bot$. By construction, $\Theta(J_w) f $ vanishes for every $0 < w < w_0$. Moreover, by Claim~(ii), 
\[ \quad \limsup_{\tau\to 0^+} \| \Theta_\tau(J_w)f \|_{L^2(\mu)} = \limsup_{\tau\to 0^+} \| 1_{iJ_w}(T_\tau)f \|_{L^2(\mu)} \leq \| 1_{iJ_{w'}}(T_\tau)f \|_{L^2(\mu)} = 0, \quad \forall w' \in ( w, w_0 ), \]
from which it follows that $\lim_{\tau\to 0^+} \Theta_\tau(J_w)f =0$. Thus, substituting the identities $\Theta(J_w) f = 0$ and $\lim_{\tau\to 0^+} \Theta_\tau(J_w)f =0$ into \eqref{eqn:strnglim1} yields 
\[ \Theta(\Omega)f = \lim_{\tau\to 0^+} \left[ \Theta_\tau(\Omega) + \Theta_\tau(J_w) - \Theta(J_w) \right]f = \lim_{\tau\to 0^+} \Theta_\tau(\Omega) f + \lim_{\tau\to 0^+} \Theta_\tau(J_w) f + \Theta(J_w) f = \lim_{\tau\to 0^+} \Theta_\tau(\Omega) f, \]
proving that Claim~(iii) is true for $\Omega$ equal to an interval.

We now extend this result to the case that $\Omega$ is an arbitrary bounded Borel subset of $\real$ with $ \Theta( \partial \Omega ) = 0 $. For that, it is sufficient to fix an arbitrary $b>0$, and prove the result for the elements of the set $\Sigma' = \{ \Omega \in \mathcal{B}([ -b, b ]) : \Theta(\partial \Omega) = 0 \} $, where $\mathcal{B}([-b, b])$ is the Borel $\sigma$-algebra on $[-b,b]$. Let then $\Sigma$ be the collection of subsets $\Omega \subseteq \mathcal{B}([-b,b])$, such that $\Theta_\tau(\Omega)$ converges strongly to $\Theta(\Omega)$. It can be shown that $\Sigma$ is a $\sigma$-algebra. Moreover, $\Sigma$ contains all intervals having zero $\Theta$ measure on their boundary, and thus must also contain the $\sigma$-algebra generated by such intervals. But this latter $\sigma$-algebra contains $\Sigma' $, and therefore $ \Theta_\tau(\Omega) \xrightarrow{s} \Theta(\Omega)$ for all $\Omega \in \Sigma'$, proving Claim~(iii).

Finally, we prove  Claim~(iv). Let $Z$ be as claimed, with support contained in a bounded open interval $I\subset i\real$. Then, the set $I\setminus S$ is a countable union of bounded open intervals $I_1,I_2,\ldots$.  Note that $H$ is the direct sum of the mutually orthogonal spaces $\ran \Theta(I_1),\ran \Theta(I_2),\ldots$, $\ran \Theta(S)$, and $\ran \Theta(I^c)$. Among these, $\ran\Theta(I^c)$ is contained in the kernel of $Z(T)$. Moreover, in a manner similar to the proof of Claim~(iii), it can be shown that $Z|_S(T_\tau)$ converges pointwise to $0$. Thus, for every $f\in\ran \Theta(S)$, $Z(T_\tau)f = Z|_S(T_\tau)f$ converges to $ Tf=0$. In light of these facts, Claim~(iv) can be simplified to the case that $Z$ is a continuous function supported on an interval $(ia,ib)\subset i\real$ with $\Theta(\{a,b\})=0$. In this case, constructing a
 function $f_w$ as in Claim~(iii), and using the same line of reasoning, it can be shown that $Z(T_\tau)\to Z(T)$ as $\tau\to 0^+$. This proves Claim (iv) and the Proposition.
\end{proof} 

Proposition~\ref{prop:SRC} lays the foundation for many of the spectral convergence results in Theorem~\ref{thm:F}, and thus Theorem~\ref{thm:Main}. It also highlights, through Claim~(iii), the convergence properties for the functional calculus and spectrum lost from the fact that $V$ is unbounded. Yet, despite the usefulness of the results stated in Proposition~\ref{prop:SRC}, the basic assumption made, namely that $T_\tau $ converges to $T$ in strong resolvent sense, is oftentimes difficult to explicitly verify. Fortunately, in the case of skew-adjoint operators of interest here, there exist sufficient conditions for strong resolvent convergence, which are easier to verify. Before stating these conditions, we recall that a \emph{core} for a closed operator $T : D(T) \to H $ on a Hilbert space $H$ is any subspace $C\subseteq D(T)$ such that $T$ is the closure of the restricted operator $T\rvert_C$. In other words, $C$ is a core if the closure of the graph of $T\rvert_C$, as a subset of $H\times H$, is the graph of $T$. Note that $T$ may not have a unique core. We also introduce the notion of convergence in the \emph{strong dynamical sense} \cite{Oliveira2008}. Specifically, a sequence $ T_\tau : D( T_\tau) \to H $, $ \tau > 0 $, of skew-adjoint operators is said to converge to $ T : D(T) \to H $ as $ \tau \to 0 $ in the strong dynamical sense if $ e^{tT_\tau} $ converges strongly to $e^{tT}$ for every $ t \in \real $. Note that in the case of the operators $ \tilde V_\tau $ from \blue{Assumption~\ref{assump:A4}} approximating the generator $V$, strong dynamical convergence means that the unitary operators $ e^{tV_\tau} $ converge strongly to the Koopman operator $U^t = e^{tV} $ for every time $t \in \real$.

\begin{lemma}\label{lem:core_SRC}
Let $T_\tau: D(T_\tau)\to H $ and $T :D(T) \to H $ be the skew-adjoint operators from \blue{Proposition}~\ref{prop:SRC}. Then, the following hold: 
\begin{enumerate}[(i)]
\item The domain $D(T^2)$ of the operator $T^2$ is a core for $T$.
\item If $T_\tau$ converges pointwise to $T$ on a core of $T$, then it also converges in strong resolvent sense.
\item Strong resolvent convergence of $T_\tau$ to $T$ is equivalent to strong dynamical convergence. 
\end{enumerate}
\end{lemma}
\begin{proof} Claim~(i) follows from \citep[][Theorem~5]{StoSza2003}. Claims~ (ii) \blue{and (iii)} follow from Propositions 10.1.18 and 10.1.8, respectively, of \cite{Oliveira2008}. There, the statements are for self-adjoint operators, but they apply to skew-adjoint operators as well. 
\end{proof}

\begin{rk*} Lemma~\ref{lem:core_SRC}(ii) indicates that a sufficient condition for strong resolvent convergence of a sequence skew-adjoint operators is pointwise convergence in a smaller domain (a core) than the full domain of the limit operator; that is, strong resolvent convergence is weaker than strong convergence for this class of operators. In Proposition~\ref{prop:Semigroup} ahead, we will see that the operator family $\tilde V_\tau $ employed in Theorem~\ref{thm:Main} actually converges pointwise to $V$ on the whole of $D(V)$.
\end{rk*}

\paragraph{Approximate point spectrum and pseudospectrum} \blue{Generalizing the definition in \eqref{eqAPS_Ut}, we say} that a complex number $\gamma$ lies in the $\epsilon$-approximate point spectrum $ \sigma_{ap,\epsilon}(T) $ of a closed operator $T : D(T) \to H$ on a Hilbert space $H$ for $\epsilon > 0$, if there exists $ f \in H$, with $ \lVert f \rVert_{H} = 1$, such that \cite{ChaitinHarrabi98,Chatelin11}
\begin{equation} \label{eqAPS}
\lVert T f - \gamma f \rVert_H < \epsilon. 
\end{equation}
As $\epsilon$ decreases towards 0, $\sigma_{ap,\epsilon}(T) $ forms an increasing family of open subsets of the complex plane, such that $ \cup_{\epsilon > 0} \sigma_{ap,\epsilon}(T) = \cmplx$. Moreover, if $T$ is a normal operator, $\sigma_{ap,\epsilon}(T)$ is the union of all open $\epsilon$-balls in the complex plane with centers lying in its spectrum, $\sigma(T)$. If, in addition, $T$ is bounded, $ \cap_{\epsilon > 0} \sigma_{ap,\epsilon}(T) = \sigma(T)$. The $\epsilon$-approximate point spectrum is also a subset of the $\epsilon$-pseudospectrum \blue{$\sigma_{\epsilon}(T)$} of $T$, defined as the set of complex numbers $\gamma$ such that $ \lVert ( T - \gamma )^{-1} \rVert > 1 / \epsilon $, with the convention that $ \lVert ( T - \gamma )^{-1} \rVert = \infty $ if $ \gamma \in \sigma(T)$ \cite{TrefethenEmbree05}. Specifically, $ \sigma_{\epsilon}(T) = \sigma_{ap,\epsilon}(T) \cup \sigma(T)$, and if $T$ is normal and bounded, $ \sigma_\epsilon(T) = \sigma_{ap,\epsilon}(T)$. For our purposes, a distinguished property of each element $\gamma \in \sigma_{ap,\epsilon}(T)$ is that there exists an associated unit-norm vector $f \in H$ which behaves approximately as an eigenfunction of $T$, in the sense of~\eqref{eqAPS}.

\subsection{Results from analysis on manifolds} 

We will state a number of standard results from analysis on manifolds that will be used in the proofs presented in Sections~\ref{sect:proof:ABCD} and~\ref{sect:proof:main}. In what follows, we consider that $M$ is a $C^r $ compact manifold, equipped with an arbitrary $C^{r-1} $ Riemannian metric (e.g., a metric induced from the ambient space $\mathcal{M}$, or the embedding $F:M\to Y$ into the data space $Y$ from Section~\ref{sect:numerics}), and an associated covariant derivative operator $\nabla$. We let $ C^0(M; TM) $ denote the vector space of continuous vector fields on $M$ (continuous sections of the tangent bundle $TM$), and $C^q(M; T^{*n} M )$ with $ 0 \leq q \leq r $ the vector space of tensor fields $\alpha $ of type $ (0,n )$ having continuous covariant derivatives $ \nabla^j \alpha \in C^{q-j}(M;T^{*n+j} M) $ up to order $j=r$. The Riemannian metric induces norms on these spaces defined by $ \lVert \Xi \rVert_{C^0(M;TM)} = \max_{x\in M} \lVert \Xi \rVert_x $, $ \lVert \alpha \rVert_{C^0(M;T^{*n} M)} = \max_{x\in M} \lVert \alpha \rVert_x $, and $ \lVert \alpha \rVert_{C^q(M;T^{*n}M)} = \sum_{j=0}^q \lVert \nabla^j \alpha \rVert_{C^0(M;T^{*(q+j)}M)}$, where $ \lVert \cdot \rVert_x $ denotes pointwise Riemannian norms on tensors. The case $C^q(M;T^{*n}M) $ with $n=0$ corresponds to the $C^q(M) $ spaces of functions. All of the $ C^0(M;TM) $ and $C^q(M; T^{*n}M)$ spaces become Banach spaces with the norms defined above, and by compactness of $M$, the topology of these spaces is independent of the choice of Riemannian metric. Hereafter, we will use $\iota^{(q)} $ to denote the canonical inclusion map of $C^q(M)$ into $L^2(\mu)$, and abbreviate $\iota^{(0)} = \iota$ as in Section~\ref{sect:Assump}. We will also use $\iota_{\RKHS}$ to denote the inclusion map of an RKHS $\RKHS$ with a $C^q$ reproducing kernel into $C^q(M)$. It follows from \citep[][Propositions~6.1 and 6.2]{FerreiraMenegatto2013} that the latter map is bounded. 

The following result expresses how vector fields can be viewed as bounded operators on functions. 
\begin{lemma}
\label{lemmaVec}
Let $M$ be a compact, $C^1$ manifold, equipped with a $C^0$ Riemannian metric. Then, as an operator from $C^1(M) $ to $C^0(M) $, every vector field $\Xi \in C^0(M;TM) $ is bounded, with operator norm $ \lVert \Xi \rVert $ bounded above by $ \lVert \Xi \rVert_{C^0(M;TM)}$. 
\end{lemma}
\begin{proof}
Denoting the gradient operator associated with the Riemannian metric on $M$ by $\grad$, the claim follows by an application of the Cauchy-Schwartz inequality for the Riemannian inner product, viz.
\begin{align*}
\lVert \Xi f \rVert_{C^0(M)} &= \lVert \Xi \cdot \grad f \rVert_{C^0(M)} \leq \lVert \Xi \rVert_{C^0(M;TM)} \lVert \grad f \rVert_{C^0(M;TM)} = \lVert \Xi \rVert_{C^0(M;TM)} \lVert \nabla f \rVert_{C^0(M;T^*M)} \\ & \leq \lVert \Xi \rVert_{C^0(M;TM)} \lVert f \rVert_{C^1(M)}. \qedhere
\end{align*}
\end{proof}

In particular, under Assumption~\ref{assump:A1}, the dynamical flow $ \Phi^t $ on $M$ is generated by a vector field $ \vec V \in C^0(M;TM) $, for which Lemma~\ref{lemmaVec} applies. This vector field is related to the generator $V$ by a conjugacy with  $\iota$ and $\iotaC$, namely, $ \iota \vec V = V \iotaC $. 

The following is a well known result from analysis \cite{Schuh1919}.

\begin{lemma}[$C^1$ convergence theorem]\label{lem:Ta0_C1}
Let $M$ be a compact, connected, $C^1$ manifold equipped with a $C^0$ Riemannian metric. Let also $f_j:M\to\real$ be a sequence of tensor fields in $C^1(M;T^{*n}M)$, such that the sequence $\{\| \nabla f_j\|_{C^0(M;T^{*(n+1)}M)}\}_{j\in\num} $ is summable. Then, if there exists $x\in M$ such that the series $F_x := \sum_{j\in\num} f_j(x)$ converges in Riemannian norm, the series $\sum_{j\in\num} f_j$ converges uniformly to a tensor field $F\in C^1(M;T^{*(n+1)}M)$ such that $F(x) = F_x$.
\end{lemma}

This lemma leads to the following $C^r$ convergence result for functions, which will be useful for establishing the smoothness of kernels constructed as infinite sums of $C^r$ eigenfunctions. 

\begin{lemma}\label{lem:Ta0_Cm}
Let $M$ be a compact, connected, $C^r$ manifold with $ r \geq 1$, equipped with a $C^{r-1} $ Riemannian metric. Suppose that $f_j:M\to\real$ is a sequence of real-valued $C^r(M)$ functions such that (i) the sequence $\{\| f_j\|_{C^r(M)}\}_{j\in\num}$ is summable; and (ii) there exists $x \in M $ such that the series $F_x = \sum_{j=0}^{\infty} f_j(x)$ converges. Then, the series $ \sum_{j=0}^{\infty} f_j$ converges absolutely and in $ C^r(M)$ norm to a $C^r$ function $F$, such that $F(x) = F_x$.
\end{lemma}
\begin{proof} We will prove this lemma by induction over $ q \in \{ 1, \ldots, r \} $, invoking Lemma~\ref{lem:Ta0_C1} as needed. First, note that summability of $ \{ \lVert f_j \rVert_{C^r(M)} \}_{j\in\num} $ implies summability of $ \{ \lVert \nabla^q f_j \rVert_{C^0(M;T^{*q}M)} \}_{j \in \num} \} $ for all $ q \in \{ 1, \ldots, r \} $. Because of this, and the fact that $ \sum_{j\in\num} f_j(x) $ converges, it follows from Lemma~\ref{lem:Ta0_C1} that $\sum_{j\in\num} f_j$ converges in $C^1$ norm to some $C^1$ function $F$. This establishes the base case for the induction ($q=1$). Now suppose that it has been shown that $\sum_{j\in\num} f_j $ converges to $F$ in $C^q(M)$ norm for $ 1 < q < r $. In that case, $\sum_{j\in\num} \nabla^q f_j(x) $ converges, and by summability of $\{ \lVert \nabla^{q+1} f_j \rVert_{C^0(M;T^{*(q+1)}M)} \}_{j\in\num}$, it follows from Lemma~\ref{lem:Ta0_C1} that $ \nabla^q F = \sum_{j\in\num} \nabla^q f_j$ converges in $C^1(M;T^{*q}M) $ norm. Thus, $\nabla^{q+1} F = \sum_{j\in\num} \nabla^{q+1} f_j$ converges in $C^0(M;T^{*(q+1)}M)$ norm, which in turn implies that $ \sum_{j\in\num} f_j $ converges to $F$ in $C^{q+1}(M)$ norm, and the lemma is proved by induction.
\end{proof}

%-_-_-_-_-_-_-_-_-_-_-_-_-_-_-_-_-_-_-_-_-_-_-_-_-_-_-_-_-_-_-_-_-_-_-_-_-_-_-_-_-_-_-_-_-_-_-_-_-_-_-_-_-_-_-_-_-_-_-_-_-_-_-_-_-_-_-_-_-_-_-_-_-_-_-_-
\section{Proof of Theorems \ref{thm:A}--\ref{thm:D}}\label{sect:proof:ABCD}

\paragraph{Proof of Theorem \ref{thm:A}} By Assumption \ref{assump:K}, $\RKHS $ is a subspace of $C^1(M)$, and therefore for every $f\in\RKHS$, $ K^* f = \iotaC f $. Claim (i) then follows from the facts that $ \ran \iotaC \subset D(V) $, and $K$ is bounded.  To prove Claim (ii), let $K': L^2(\mu) \to C^0(M) $ be the kernel integral operator associated with the continuous kernel $k'$, and $ \iota $ the $ C^0(M) \to L^2(\mu) $ inclusion map. Because $\iota K' $ is a Hilbert-Schmidt integral operator on $L^2(\mu) $, with operator norm bounded above by its Hilbert-Schmidt norm, 
$ \lVert \iota K' \rVert \leq  \lVert k' \rVert_{L^2(\mu \times \mu)} \leq \lVert k' \rVert_{C^0(X\times X)} $, 
the claim will follow if it can be shown that $  \iota K' = VG $. To that end,  note that for every $ f \in L^2(\mu) $ and  $ x \in M $ we have $ K' f( x ) = \langle k'(x,\cdot ), f \rangle_\mu $. \blue{Thus, using the  $C^0(M) $ limit $ k'( x,\cdot) = \lim_{t\to 0} g_t $, where $ g_t = ( k( \Phi^t(x), \cdot) -k(x,\cdot) ) / t $, and continuity of inner products, we obtain}
\begin{displaymath}
    K'f(x) = \langle k'(x,\cdot),f \rangle_\mu = \langle \lim_{t\to 0} g_t, f \rangle_\mu = \lim_{t\to 0} \langle g_t, f \rangle_\mu = \lim_{t\to0} \frac{1}{t}[\langle k( \Phi^t(x), \cdot), f \rangle_\mu - \langle k( x, \cdot ), f \rangle_\mu]=  \vec V Kf(x).
\end{displaymath}
As a result, because $ \ran K \subset C^1(M) $, for any $ f \in L^2(\mu) $ it follows that
\begin{displaymath}
    \iota K' f = \iota \vec V K f = V\iotaC K f = V K^* K f = V G f,
\end{displaymath}
proving Claim~(ii). Finally, to prove Claim~(iii), we have by definition of the adjoint,
\[ D((GV)^*) := \{ f\in L^2(\mu): \; \text{$\exists h\in L^2(\mu)$ such that $\forall g\in D(V)$, $\langle f, G V g\rangle_{\mu} = \langle h, g \rangle_\mu$} \}, \quad (GV)^*f:= h,\]
\blue{where $h$ is unique by the Riesz representation theorem and density of $D(V)$ in $L^2(\mu)$}. We will now use this definition to show that $(GV)^* = -VG= -\VG$. Indeed, for every $f\in D(A) = L^2(\mu)$ and every $g\in D(V)$, setting $h = -\VG f$, we obtain 
\[ \langle h, g \rangle_\mu = \langle -\VG f, g \rangle_\mu = -\langle V G f, g \rangle_\mu = \langle G f, V g \rangle_\mu = \langle f, G V g \rangle_\mu. \]
This satisfies the definition of $(GV)^*$, proving the claim and the theorem. \qed

\paragraph{Proof of Theorem \ref{thm:B}} We begin with the proof of Claim~(i). The inclusion $ \ran K^* \subset D(V) $ holds because $\RKHS$ is a subspace of $C^1$. To prove that $VK^*$ is bounded, \blue{note that by} Lemma~\ref{lemmaVec}, and the fact that the inclusion map $\iota_{\RKHS} : \RKHS \to C^1(M)$ is bounded, 
\begin{displaymath}
    \lVert VK^*f \rVert_{L^2(\mu)} = \Vert \iota \vec V f \rVert_{L^2(\mu)} \leq \lVert \vec V f \rVert_{C^0(M)} \leq \lVert \vec V \rVert \lVert f \rVert_{C^1(M)} \leq \lVert \iota_{\RKHS} \rVert \lVert \vec V \rVert \lVert f \rVert_{\RKHS},
\end{displaymath}
proving that $ VK^* $ is bounded and completing the proof of Claim~(i). Turning to Claim~(ii), that $W$ is compact follows from the fact that it is a composition of a compact operator, $K$, by a bounded operator, $ VK^*$. Moreover, $W$ is skew-symmetric by skew-adjointness of $V$, and thus skew-adjoint because it is bounded. $W$ is also real because $K$ and $V$ are real operators. It thus remains to verify the integral formula for $ W f $ stated in the theorem. For that, it follows from the Leibniz rule for vector fields and the fact that $ k$ lies in $C^1(M \times M)$ that for every $ f \in C^1(M) $ and $ x \in X$, 
\begin{displaymath}
    k(x,\cdot) \vec V f = \vec V( k(x,\cdot) f ) - ( \vec V k(x,\cdot) ) f = \vec V(k(x,\cdot)f)+ \tilde k'(x,\cdot) f.
\end{displaymath}
Using this result, and the fact that $ \int_M \vec V( k(x,\cdot) f ) \, d\mu = \langle 1_M, V( k(x,\cdot ), f \rangle_{\mu} $ vanishes by skew-adjointness of $V$, we obtain
\begin{displaymath}
    KVK^*f(x) = K V \iotaC f(x) = K \vec V f( x ) = \int_M k(x,\cdot) \vec V f \, d\mu = \int_M \tilde k'( x, \cdot ) f \, d\mu. \qed 
\end{displaymath}

\paragraph{Proof of Theorem \ref{thm:C}} That $B=-A^* $ is a Hilbert-Schmidt integral operator with kernel $\tilde k'$ follows from standard properties of integral operators. Next, to prove Claim~(i), note that $GV$ is bounded as it has a bounded adjoint, $(GV)^* = -A$, by Theorem~\ref{thm:A}, and therefore has a unique closed extension $ \overline{GV} : L^2(\mu) \to L^2(\mu) $ equal to $(GV)^{**}$. In order to verify that $ \overline{GV} = B $, it suffices to show that $ GVf = Bf $ for all $ f $ in any dense subspace of $D(V)$; in particular, we can choose the subspace $\iotaC C^1(M)$. For any observable $ \iotaC f $ in this subspace, we have $Bf=\iota \tilde K'f $ and $ GVf = \iotaC K \vec V f $, where $\tilde K' : L^2(\mu) \to C^0(M)$ is the integral operator with kernel $\tilde K'$, defined analogously to the operator $K'$ in the proof of Theorem~\ref{thm:A}. Employing the Leibniz rule as in the proof of Theorem~\ref{thm:B}, it is straightforward to verify that $Bf $ is indeed equal to $GVf $, proving that $\GV$ is the unique closed extension of $GV$. Next, to show that $\GV$ is also an extension of $K^* W\Nyst$, it suffices to show that $GV \supseteq K^*W\Nyst$. For that, note that $K^*W\Nyst$ is a well defined operator by Theorem \ref{thm:B}, and thus, substituting the definition for $\KVKst$ in \eqref{eqn:similarity_trnsfrm}, and using the fact that $K^*\Nyst$ is the identity on $D(\Nyst)$, we obtain
    \[K^* \KVKst \Nyst = G V K^* \Nyst = G V\rvert_{D(\Nyst)}.\]
    This shows that $ K^*K W \Nyst \subseteq GV \subset B $, confirming that $\GV$ is a closed extension of $K^*W\Nyst$. If $k$ is strictly positive, then $D(\Nyst)$ is dense, and $\GV$ is the unique closed extension of $ K^*W\Nyst$. This completes the proof of Claim~(i). 
    
    Next, to prove Claim~(ii), note that because $\GV$ is bounded, the Taylor series   $e^{t\GV} = \sum_{n=0}^\infty (t\GV)^n / n!$ converges in operator norm for every $t \in \real$, and the set $\{ e^{t\GV} \}_{t\in\real}$ clearly forms a group under composition of operators. This group is norm-continuous by boundedness of $\GV$. Similarly, we have $ e^{tW} = \sum_{n=0}^{\infty} (t \KVKst)^n/n! $ in operator norm,  and observing that for every $n\in\num$, $K^* \KVKst^n = \GV^n$, we arrive at the claimed identity,
\[ K^* e^{t\KVKst} = \sum_{n=0}^{\infty} \frac{1}{n!} t^n K^* \KVKst^n = \sum_{n=0}^{\infty} \frac{1}{n!} t^n  \GV^n = e^{t\GV}. \] 
The identity $K^* e^{tW} \Nyst = e^{t\GV}\rvert_{D(\Nyst)}$ then follows from the fact that $K^*\Nyst$ is the identity on $D(\Nyst)$. 

\paragraph{Proof of Theorem \ref{thm:D}} Let $\{ \phi_j \}_{j=0}^\infty$ be an orthonormal basis of $L^2(\mu)$ consisting of eigenfunctions $\phi_j$ of $G$ corresponding to eigenvalues $ \lambda_j $ ordered in decreasing order. Let also $ \{ \psi_j \}_{j=0}^\infty$ be an orthonormal basis of $\RKHS$, whose first $J$ elements are given by~\eqref{eqn:def:psi} (with some abuse of notation as $J$ may be infinite). Recall from Section~\ref{secReview_RKHS} that $\mathcal{U} \phi_j = \psi_j$.  To prove the theorem, it suffices to show that $ G^{1/2} V G^{1/2} $ is well-defined on a dense subspace of $L^2(\mu)$, and on that subspace, $ G^{1/2}  V G^{1/2} $ and $ \mathcal{U}^* W \mathcal{U} $ are equal. To verify that $G^{1/2} V G^{1/2} $ is densely defined, note first that $G^{1/2}\phi_j $ trivially vanishes for $ j \notin J$, and therefore $ G^{1/2}V G^{1/2} \phi_j $ is well-defined and vanishes too. Moreover if $ j \in J$, $G^{1/2} \phi_j = K^* \psi_j $, and $ G^{1/2} V G^{1/2} \phi_j $ is again well defined since $ \ran K^* \subset D(V) $. As a result, the domain of $ G^{1/2} V G^{1/2} $ contains all linear combinations of $ \phi_j $ with $ j \notin J$, and all finite combinations with $ j \in J $, and is therefore a dense subspace of $L^2(\mu) $. Next, to show that $ \mathcal{U}^* W \mathcal{U} $ and $ G^{1/2} V G^{1/2} $ are equal on this subspace, it suffices to show that they have the same matrix elements in the $ \{ \phi_j \} $ basis of $L^2(\mu)$, i.e., that $ \langle \phi_i, G^{1/2} V G^{1/2} \phi_j \rangle_{\mu} $ is equal to $\langle \phi_i, \mathcal{U}^* W \mathcal{U} \phi_j \rangle_{\RKHS}$ for all $i,j \in \num_0$. Indeed, because $ \ker \mathcal{U} = \ker G^{1/2} = ( \ran G^{1/2} )^\perp$, both $ \mathcal{U} \phi_j $ and $ G^{1/2} \phi_j$ vanish when $ j \notin J$. We therefore deduce that if either of $ i $ and  $j $ does not lie in $J$, the matrix elements $\langle \phi_i, \mathcal{U}^*W\mathcal{U} \phi_j \rangle_{\mu}$ and $\langle \phi_j, G^{1/2}WG^{1/2} \phi_j \rangle_{\mu}$ both vanish. On the other hand, if $i,j \in J$, we have 
\begin{align*}
        \langle \phi_i, \mathcal{U}^* W \mathcal{U} \phi_j \rangle_{\mu} &= \langle \psi_i, W \psi_j \rangle_{\RKHS} = \langle K^* \psi_i, V K^* \psi_j \rangle_{\mu} = \langle \lambda_i^{-1/2} K^*K \phi_i, \lambda_j^{-1/2} K^* K \phi_j \rangle_{\mu} \\
        &= \langle G^{1/2} \phi_i, V G^{1/2} \phi_j \rangle_{\mu} =  \langle \phi_i, G^{1/2} V G^{1/2} \phi_j \rangle_\mu.  
\end{align*}
We have thus shown that $ \mathcal{U}^* W \mathcal{U} $ and $ G^{1/2} V G^{1/2} $ are equal on a dense subspace of $L^2(\mu)$, and because the former operator is bounded and defined on the whole of $L^2(\mu)$, it follows that $\tilde V = \mathcal{U}^*W \mathcal{U}$ is the unique closed extension of $G^{1/2}VG^{1/2}$. That $\tilde V $ is skew-adjoint, Hilbert-Schmidt, and real follows immediately. \qed 

%-_-_-_-_-_-_-_-_-_-_-_-_-_-_-_-_-_-_-_-_-_-_-_-_-_-_-_-_-_-_-_-_-_-_-_-_-_-_-_-_-_-_-_-_-_-_-_-_-_-_-_-_-_-_-_-_-_-_-_-_-_-_-_-_-_-_-_-_-_-_-_-_-_-_-_-
\section{Proof of Theorems \ref{thm:E} and \ref{thm:F}}\label{sect:proof:EF} 

We will need the following lemma, describing how to convert between eigenfunctions of  $A$, $\GV$, $\tilde V$, and $\tilde W$. The proof will be omitted since it follows directly from the definitions of these operators.
\begin{lemma}\label{lem:fsdt}
Let Assumptions \ref{assump:A1} and \ref{assump:K} hold with $r=1$. Then,
\begin{enumerate}[(i)]
    \item If $\EigenW\in\mathcal{K}$ is an eigenfunction of $\tilde W$ at eigenvalue $i\omega$, then $K^* \EigenW$ is an eigenfunction of $\GV$ at eigenvalue $i\omega$. 
\item $\EigenGV' $ is an eigenfunction of $\VG$ at eigenvalue $i\omega$ iff $K \EigenGV'$ is an eigenfunction of $\tilde W$ at eigenvalue $i\omega$.
\item If $\EigenGV' \in L^2(\mu)$ is an eigenfunction of $\VG$ at eigenvalue $i\omega$, then $G^{1/2}\EigenGV'$ is an eigenfunction of $\symV$ at eigenvalue $i\omega$. 
\item If $\EigApp$ is an eigenfunction of $\symV$ at eigenvalue $i\omega$, then $G^{1/2}\EigApp$ is an eigenfunction of $\GV$ at eigenvalue $i\omega$.
\end{enumerate}

\end{lemma}

\paragraph{Proof of Theorem \ref{thm:E}} Starting from Claim~(i), let $\tilde W$ be the restriction of $W$ onto the closed subspace $\mathcal{K} \subseteq \mathcal{H}$. Since $\mathcal{K} $ is invariant under $W$, and $ \ker W \supseteq \mathcal{K}^\perp $ by definition, we have $ \sigma_p(W) = \sigma_p(\tilde W) $ if $ \mathcal{K}^\perp = \{ 0 \} $ (i.e., $K$ has dense range) and $ \sigma_p( W) = \sigma_p(\tilde W) \cup \{ 0 \} $ otherwise. Thus, to prove the claim, it is enough to show that $ \sigma_p(A) = \sigma_p(B)= \sigma_p(\tilde V) = \sigma_p(\tilde W) $, including eigenvalue multiplicities. To that end, note first that   $ \tilde W $ and $ \tilde V$ are unitarily equivalent by Theorem~\ref{thm:D} and strict $L^2(\mu)$-positivity of $k$, and thus $ \sigma_p(\tilde W) = \sigma_p(\tilde V) $, including multiplicities. Moreover, by Lemma~\ref{lem:fsdt}, $ \sigma_p(A) \subseteq \sigma_p( \tilde W ) \subset i \real $, and because $A$ is a real operator, it follows that $\sigma_p(A)$ is symmetric about the origin of the imaginary line $i \real$, so that
\begin{displaymath}
\sigma_p(A)= -\sigma_p(A)= -\sigma_p(A)^* = - \sigma_p(A^*) = - \sigma_p( -B ) = \sigma_p(B).
\end{displaymath}
Thus, the equality of $ \sigma_p(A) $, $\sigma_p(B)$, $\sigma_p(\tilde V) $, and $\sigma_p(\tilde W)$ will follow if it can be shown that $ \sigma_p(A) = \sigma_p(\tilde V) $. Indeed, it follows from Lemmas~\ref{lem:fsdt}(iii) and~\ref{lem:fsdt}(iv) that $ \sigma_p(A) \subseteq \sigma_p(\tilde V) $ and $ \sigma_p( \tilde V) \subseteq \sigma_p( B) $, respectively. These relationships, together with the fact that $ \sigma_p(A) = \sigma_p(B) $, imply that $ \sigma_p(A) = \sigma_p(\tilde V)$, and thus $\sigma(A)=\sigma(B)=\sigma(\tilde V) = \sigma(\tilde W)$, as claimed. The equality of the multiplicities of the eigenvalues of $A$, $B$, and $\tilde V$ follows from the facts that $K$ and $G^{1/2}$ are injective operators.	 This completes the proof of Claim~(i).

To prove Claim~(ii), note that because $k$ is $L^2(\mu)$-Markov ergodic, $G f = f $ implies that $f $ is $\mu$-a.e.\ constant. In addition, by ergodicity of the flow $ \Phi^t$, $ Vf = 0 $ implies again that $ f $ is $ \mu$-a.e.\ constant. It then follows that 
\[ \VG f = 0 \implies V (Gf) = 0 \implies Gf = \mbox{$\mu$-a.e.\ constant} \implies f = \mbox{$\mu$-a.e.\ constant}.\]
This shows that 0 is a simple eigenvalue of $A$ with constant corresponding eigenfunctions. Therefore, since $ \sigma_p(A)=\sigma_p(B)=\sigma_p(\tilde V) = \sigma_p(\tilde W) $, 0 is also a simple eigenvalue of $ B $, $ \tilde V $, and $\tilde W$, and the constancy of the corresponding eigenfunctions follows directly from the definition of these operators.

Next, to prove Claims~(iii) and~(iv), fix a nonzero eigenvalue $i\omega_j$ of $\VG$. By compactness of this operator, the corresponding eigenspace is finite-dimensional, and thus the injective operator $G^{1/2}$ maps every basis of this eigenspace to a linearly independent set. By Lemma~\ref{lem:fsdt}(iii) and Claim~(i), this set is actually a basis of the eigenspace of $\tilde V$ at eigenvalue $ i \omega_j $. As result, every eigenfunction of $\symV$ at nonzero corresponding eigenvalue lies in the range of $G^{1/2}$. Moreover, it follows from Claim~(ii) that every eigenfunction of $\tilde V $ at eigenvalue 0 is constant, and thus also lies in the range of $G^{1/2}$. We therefore conclude that every eigenfunction of $\tilde V$ lies in the range of $G^{1/2}$, and thus in the domain of $G^{-1/2}$, as claimed. The fact that $\tilde v_j \in \ran G^{1/2} $ for all $ j \in \num_0$ also implies that $z'_j$ is an eigenfunction of $A $ at eigenvalue $i\omega_j$, since 
\begin{displaymath}
\tilde V \EigApp_j = \tilde V G^{1/2} G^{-1/2} \EigApp_j = G^{1/2} V G G^{-1/2} \EigApp_j = G^{1/2} V G G^{-1/2} \EigApp_j = G^{1/2}A \EigenGV'_j. 
\end{displaymath}
In addition, we can deduce directly from Lemma~\ref{lem:fsdt}(iv) that each of the $z_j $ are eigenfunctions of $B$ at eigenvalue $i\omega_j$, as stated in Claim~(iv). 

To complete the proof of Claims~(iii) and (iv), it remains to show that $\{ z'_0, z'_1, \ldots \} $ and $ \{ z_0, z_1, \ldots \}$ form unconditional Schauder bases of $L^2(\mu)$. For that, note first that $ \EigenGV_j $ is a dual sequence to the $\EigenGV'_j$, i.e., 
\[\langle \EigenGV'_j, \EigenGV_l \rangle_\mu = \langle G^{-1/2} \EigApp_j, G^{1/2} \EigApp_l \rangle_\mu = \langle \EigApp_j, \EigApp_l \rangle_\mu = \delta_{jl}.\]
As a result, since every Schauder basis has a unique dual sequence, which is also a Schauder basis \cite{Young1981}, Claims~(iii) and~(iv) will be proved if it can be shown that $ \{ z_0, z_1, \ldots \} $ is an unconditional Schauder basis. To verify that this is indeed the case, fix $\{ \phi_0, \phi_1, \ldots \} $ from~\eqref{eqn:def:psi} as an orthonormal basis of $L^2(\mu)$ (corresponding to the eigenvalues $\lambda_0, \lambda_1, \ldots$), and $ \{ e_0, e_1, \ldots \} $ as the standard orthonormal basis of $\ell^2$, and define the unbounded operator $Z':D(Z') \subset \ell^2 \to \ell^2$, the bounded operator $L: \ell^2 \to \ell^2$, the unitary operator $U : \ell^2 \to \ell^2$, and the diagonal operator $\Lambda: \ell^2 \to \ell^2$ such that 
\[ \langle e_i, Z'e_j \rangle_{\ell^2}= \langle \phi_i, \EigenGV'_j \rangle_\mu, \quad \langle e_i, L e_j \rangle_{\ell^2} := \langle \EigenGV_i, \phi_j \rangle_\mu, \quad \langle e_i, U e_j \rangle_{\ell^2} = \langle \phi_i, \tilde z_j \rangle_{\mu}, \quad \langle e_i, \Lambda e_j \rangle_{\ell^2} = \lambda_i \delta_{ij}. \] 
Here, $D(Z')$ is defined as the dense subspace of $\ell^2$ whose elements $\sum_{j=0}^\infty c_j e_j$ satisfy $ \sum_{i,j=0}^{\infty} \lvert \langle \phi_i, z'_j \rangle_{\mu} c_j \rvert^2 < \infty $. Note that $Z'^{*}$ an $L$ are the matrix representations of the mappings $\phi_j \mapsto \EigenGV'_j$, and $\phi_j \mapsto \EigenGV_j$, respectively, in the orthonormal basis $\{\phi_0, \phi_1, \ldots \}$. With these definitions, the $\num \times \num$ matrices with elements $ \langle e_i, Z' e_j \rangle_{\ell^2}$ and $ \langle e_i, L e_j \rangle$, which represent $Z'$ and $L$, respectively, have $\ell^2$ summable columns and rows respectively. 

Next, note that $L$ is a left inverse of $Z'$, as can be verified by computing 
\begin{align}
\nonumber \langle e_i, LZ' e_j \rangle_\mu &= \sum_{j=0}^{\infty} \langle \EigenGV_i, \phi_j \rangle_\mu \langle \phi_j, \EigenGV'_l \rangle_\mu = \sum_{j=0}^{\infty} \langle \EigenGV_i, \phi_j \rangle_\mu \langle \phi_j, \phi_j \rangle_\mu \langle \phi_j, \EigenGV'_l \rangle_\mu = \sum_{j=0}^{\infty}\sum_{k=0}^{\infty} \langle \EigenGV_i, \phi_j \rangle_\mu \langle \phi_j, \phi_k \rangle_\mu \langle \phi_j, \EigenGV'_l \rangle_\mu \\
\label{eqLeftZInv}& =\left\langle \sum_{j=0}^{\infty} \langle \phi_j, \EigenGV_i \rangle_\mu \phi_j, \sum_{k=0}^{\infty} \langle \phi_j, \EigenGV'_l \rangle_\mu \phi_k \right \rangle_\mu = \langle \EigenGV_i , \EigenGV'_l \rangle_\mu = \delta_{il}.
\end{align}
Similarly, one can verify the identities $L = U^* \Lambda^{1/2}$ and $Z' = \Lambda^{-1/2} U$. Using these results, and defining $\Pi_l: \ell^2 \to \ell^2$ as the canonical orthogonal projection onto $ \spn \{ e_0, \ldots, e_{l-1} \} $, we obtain 
\begin{equation}\label{eqn:vsr}
Z'\Pi_lL = \Lambda^{-1/2} U \Pi_l U^* \Lambda^{1/2} = \Lambda^{-1/2} \Pi_l \Lambda^{1/2} = \Pi_l, \quad Z'\Pi_lL \xrightarrow[l\to\infty]{s} \Id.
\end{equation}
By \citep[][Lemma 2.1]{TJC_I_2012}, \eqref{eqLeftZInv} and \eqref{eqn:vsr} imply that the columns of the matrix representation of $Z'$, i.e., the eigenfunctions $\EigenGV'_j$, form a Schauder basis of $L^2(\mu)$. The unconditionality of this basis follows from the fact that if the $\EigenGV_j$ are permuted, \eqref{eqn:vsr} still holds, but with the rows and columns of the matrix representations of $U$, $Z'$, $L$, and $\Lambda$ correspondingly permuted. This completes the proof of Claims~(iii) and~(iv).

In Claim~(v), the fact that the $\EigenW_j$ are eigenfunctions of $\tilde W$ follows from Lemma \ref{lem:fsdt}~(ii). We also have 
\begin{displaymath}
    \zeta_j = K G^{-1/2} \tilde z_j = \mathcal{U} \tilde z_j,
\end{displaymath}
and because $ \mathcal{U} $ acts as a unitary operator from $ L^2(\mu) $ to $\mathcal{K}$, the fact that $ \{ \tilde z_0, \tilde z_1, \ldots \} $ is an orthonormal basis of $ L^2(\mu) $ implies that $ \{ \zeta_0, \zeta_1, \ldots \} $ is an orthonormal basis of $\mathcal{K}$, proving the claim.

Finally, in Claim~(vi), note first that all of the summations are well defined and independent of ordering due to the unconditionality of all the bases involved. The results for $\tilde V$ and $\tilde W$ follow from standard properties of Hilbert-Schmidt, skew-adjoint operators. Here, we will only verify the representation of $\GV$, as the case for $\VG$, is analogous. By Claim~(iv), every $f\in L^2(\mu)$ has a unique expansion $f = \sum_{j=0}^{\infty} a_j \EigenGV_j$, with the summation holding in $L^2(\mu)$ sense. Then, since $\GV\EigenGV_j = i\omega_j \EigenGV_j$ and $\GV$ is bounded,
\[ \GV f = \GV \sum_{j=0}^{\infty} a_j \EigenGV_j = \sum_{j=0}^{\infty} a_j \GV\EigenGV_j = \sum_{j=0}^{\infty} a_j i\omega_j \EigenGV_j .\]
The fact that $Bf = \sum_{j=0}^{\infty} \langle\EigenGV'_j, f \rangle_{\mu} i\omega_j \EigenGV_j $ then follows from the identity below for the coefficients $a_j$:
\[ \langle\EigenGV'_j, f \rangle_{\mu} = \left\langle\EigenGV'_j, \sum_{k=0}^\infty a_k \EigenGV_k \right\rangle_{\mu} = \sum_{k=0}^\infty a_k \langle\EigenGV'_j, \EigenGV_k \rangle_{\mu} = \sum_{k=0}^\infty a_k \delta_{jk} = a_j.\]
This completes the proof of the claim and Theorem \ref{thm:E}. \qed

\paragraph{Proof of Theorem \ref{thm:F}} It follows from the strong convergence $G_\tau\xrightarrow{s} \Id$ in Assumption~\ref{assump:A4} that 
\[ \lim_{\tau\to 0^+} \| (\GV_\tau-V)f \|_{L^2(\mu)} = \lim_{\tau\to 0^+} \| (G_\tau V-V)f \|_{L^2(\mu)} = \lim_{\tau\to 0^+} \| (G_\tau-\Id)Vf \|_{L^2(\mu)} =0, \quad \forall f\in D(V^2), \]
proving Claim~(i). Turning to Claim~(ii), it follows from Lemma~\ref{lem:core_SRC}(i) that $D(V^2) $ is a core for $V$, and thus by Assumption~\ref{assump:A4} and Lemma~\ref{lem:core_SRC}(ii) that, as $\tau \to 0^+$, $\tilde V_\tau$ converges to $V$ in strong resolvent sense. The strong convergence of $Z(\tilde V_\tau) $ to $Z(V)$ then follows by Proposition~\ref{prop:SRC}(i). The result for $\mathcal{U}_\tau^* Z(W_\tau) \mathcal{U}_\tau$ follows from the fact that this operator is equal to $ Z(\tilde V_\tau) $, by~\eqref{eqSpecUnitary}.

To prove Claim~(iii) note first that, by standard properties of the Borel functional calculus, $Z(\tilde V_\tau)$ is a uniformly bounded family of operators with $ \lVert Z(\tilde V_\tau) \rVert \leq \lVert Z \rVert_{C^0(i\real)}$. As a result, it follows from a uniform boundedness principle that $ Z(\tilde V_\tau ) G_\tau^{1/2} \xrightarrow{s} Z(V) $, as $ \tau \to 0^+$. Similarly, $G_\tau^{1/2}$ is uniformly bounded, so the strong $ \tau \to 0^+$ limit of $ G_\tau^{1/2}Z(A_\tau)$ is equal to the strong $ \tau \to 0^+$ limit of $Z(A_\tau)$. However, $G_\tau^{1/2}Z(A_\tau) = Z( \tilde V_\tau ) G^{1/2}_\tau $ by~\eqref{eqZAB}, and we conclude that $ Z(A_\tau) \xrightarrow{s} Z(V)$, as claimed. That $Z(B_\tau) \xrightarrow{s} Z(V)$ then follows immediately from the fact that $ B_\tau = - A_\tau^*$. The latter result leads in turn to the strong convergence $K^*_\tau Z(W_\tau) \mathcal{N}_\tau \xrightarrow{s} Z(V) $ on $H_\infty $, since, by Theorem~\ref{thm:C}(ii) and complex analyticity of $Z$, $K^*_\tau Z(W_\tau) \mathcal{N}_\tau $ and $Z(B_\tau) $ are equal operators on $H_\infty$.

Next, the strong convergence of $\tilde E_\tau(\Omega)$ to $E(\Omega)$ in Claim~(iv) follows from Proposition \ref{prop:SRC}(iii). Equation~\eqref{eqSpecUnitary} then leads to the result for $ \mathcal{U}_\tau^* \mathcal{E}_\tau(\Omega) \mathcal{U}_\tau$. Finally, Claim~(v) follows from Proposition \ref{prop:SRC}(vi). \qed

%-_-_-_-_-_-_-_-_-_-_-_-_-_-_-_-_-_-_-_-_-_-_-_-_-_-_-_-_-_-_-_-_-_-_-_-_-_-_-_-_-_-_-_-_-_-_-_-_-_-_-_-_-_-_-_-_-_-_-_-_-_-_-_-_-_-_-_-_-_-_-_-_-_-_-_-
\section{Proof of Theorems \ref{thm:Markov}, \ref{thm:Main} and Corollaries \ref{corr:APS}, \ref{thm:predic}}\label{sect:proof:main}

\paragraph{Proof of Theorem \ref{thm:Markov}} First, note that by Lemma~\ref{lemMercer}(ii), the sequence $ \{ \lambda_j \}_{j=0}^\infty $ is summable. Moreover, since $ \lambda_j \leq 1 $, $ \{ \lambda_j^q \}_{j=0}^\infty $ is summable for every $ q \geq 1 $. Now define $ r_{\tau,j} = (\lambda_{\tau,j} / \lambda_j)^{1/2} $. Due to the exponential decay of the $\lambda_{\tau,j}$ in~\eqref{eqn:def:schemeIII}, the sequences $\{ r_{\tau,j}^q \}_{j=0}^\infty $ and $\{ r_{\tau,j}^q / \lambda_j \}_{j=0}^{\infty} $ are summable for every $ q \geq 1 $ and $ \tau > 0 $. Observe now that $ \psi_{\tau,j} $ and $ p_\tau $ can be expressed as 
\begin{equation}\label{eqn:PsiTauJ}
\psi_{\tau,j} = r_{\tau,j} \psi_j = r_{\tau,j} \lambda^{-1/2}_{\tau,j} P \phi_j, \quad p_\tau(x,y) = \sum_{j=0}^{\infty} r_{\tau,j}^2 \psi_j(x) \psi_j(y). 
\end{equation}
It therefore follows from the summability of $\{ r^2_{\tau,j} \}_{j=0}^\infty$ that for every $\tau >0$, the series for $p_\tau(x,y)$ also converges absolutely and uniformly on $X_\nu\times X_\nu$, and condition~(ii) of Lemma \ref{lem:Ta0_Cm} is satisfied. Next, observe that for every $j\in\num_0$ and $\alpha\in \{1,\ldots,r\}$ ,
\[\psi_{\tau,j} = r_{\tau,j} \lambda_j^{-1/2} \int_M p(\cdot,y) \phi_j(y)\, d\nu(y), \quad \nabla^\alpha \psi_{\tau,j} = r_{\tau,j} \lambda_j^{-1/2} \int_M \nabla^\alpha p(\cdot,y) \phi_j(y)\, d\nu(y),\]
and thus $\| \psi_{\tau,j} \|_{C^r(M)} \leq r_{\tau,j} \lambda_j^{-1/2} \| p\|_{C^r}$. Let now $ \nabla_1 f $ and $ \nabla_2 f $ denote the covariant derivatives of $ f \in C^1(M\times M) $ with respect to the first and second variables, respectively. Defining $f_j(x,y) = \psi_{\tau,j}(x) \psi_{\tau,j}(y)$, and noting that $ f_j $ is a $ C^r(M\times M) $ function by $ C^r $ regularity of $ p$ (and thus $ \psi_j$), we have
\[\| f_j \|_{C^r(M)} = \sum_{\substack{\alpha,\beta \in \{0, \ldots, r \},\\\alpha +\beta = m}} \left\| \left(\nabla^\alpha_1 \nabla^\beta_2\right) f_j(x,y) \right\|_{C^0(M; T^{*(\alpha+\beta)}M)} \leq C r_{\tau,j}^2/ \lambda_{j}, \]
where, $C$ is a constant equal to a multiple of $\| p\|^2_{C^r(M \times M )}$. This bound implies that $\left\{ \| f_j \|_{C^r} \right\}_{j=0}^{\infty} \in\ell^1$, and condition~(i) of Lemma \ref{lem:Ta0_Cm} is satisfied. We therefore conclude that Lemma~\ref{lem:Ta0_Cm} applies, and as a result, for every $x, y \in M $, $ \sum_{j=0}^{\infty} f_j(x,y) = \sum_{j=0}^\infty \psi_{\tau,j}(x) \psi_{\tau,j}(y)$ converges in $ C^r(M\times M) $ norm to a $C^r(M\times M) $ function, $ p_\tau$, as claimed. 

Next, we begin our proof of Claim~(i) by showing that $ p_\tau $ is the reproducing kernel for an RKHS. Fixing $ \tau > 0$, we start from the pre-Hilbert space $ H_0 = \spn\{ \psi_{\tau,j} \} $, equipped with the inner product 
\begin{displaymath}
\left\langle \sum_{i=0}^{m-1} a_i \psi_{\tau,j}, \sum_{j=0}^{n-1} b_j \psi_{j,\tau} \right\rangle_{H_0} = \sum_{i=0}^{m-1} \sum_{j=0}^{n-1} a_i^* \delta_{ij} b_j.
\end{displaymath}
By~\eqref{eqn:PsiTauJ}, for every $ f = \sum_{j=0}^{n-1} c_j \psi_{\tau,j} \in H_0 $, we have 
\begin{displaymath}
\lVert f \rVert_{\RKHS}^2 = \left \lVert \sum_{j=0}^{n-1} c_j r_{\tau,j} \psi_{j} \right \rVert^2_{\RKHS} = \sum_{j=0}^{n-1} \lvert r_{\tau,j} \rvert^2 \lvert c_j \rvert^2 \leq C \sum_{j=0}^{n-1} \lvert c_j \rvert^2 = C \lVert f \rVert^2_{H_0},
\end{displaymath}
where $ C = \max_{j \in \num_0} \lvert r_{\tau,j} \rvert^2$. This implies that every Cauchy sequence in $H_0$ is a Cauchy sequence in $ \RKHS$, and as a result the Hilbert space completion of $H_0$, denoted $H$, can be identified with a subspace of $\RKHS$. In particular, $H$ is a Hilbert space of functions on $M$ with an orthonormal basis $ \{ \psi_{\tau,j} \}_{j=0}^\infty$. We will next show that $H$ is an RKHS with reproducing kernel $p_\tau$ by showing that, for every $ x \in M$, the kernel sections $ p_\tau(x,\cdot) $ lie in $H$, and function evaluation at $ x $ is a bounded linear functional on $H$ equal to an inner product with these sections. Indeed, since $p$ is the reproducing kernel for $\RKHS$, for every $x\in M$, the section $p(x,\cdot)$ lies in $\RKHS$, and thus, by the Mercer representation for $p$, $\sum_{j=0}^{\infty} |\psi_{j}(x)|^2<\infty$. It therefore follows that
\begin{displaymath} \sum_{j=0}^{\infty} |\psi_{\tau,j}(x)|^2 = \sum_{j=0}^{\infty} r_{\tau,j}^2 |\psi_{j}(x)|^2<\infty,
\end{displaymath}
and because $ \{ \psi_{\tau,j} \}_{j=0}^\infty$ is an orthonormal basis of $H$, $ p_\tau(x, \cdot ) = \sum_{j=0}^\infty \psi_{\tau,j}(x) \psi_{\tau,j} $ lies in $H$. Moreover, for every $ x \in M$ and $f\in H$,
\[ f(x) = \sum_{j=0}^{\infty} \langle \psi_{\tau,j},f \rangle_H \; \psi_{\tau,j}(x) = \left\langle \sum_{j=0}^{\infty} \psi_{\tau,j}(x) \psi_{\tau,j}, f \right\rangle_H = \left\langle p_\tau(x,\cdot), f \right\rangle_H, \]
which shows that pointwise evaluation on $H$ is given by inner products with the kernel sections $ p_\tau $. We therefore conclude that $H$ is an RKHS, denoted $\RKHS_\tau$, with $ p_\tau $ as its $C^r$ reproducing kernel. As a result, $ p_\tau $ is positive-definite, and it induces integral operators $P_\tau : L^2(\nu) \to \RKHS_\tau $ and $ G_\tau = P^*_\tau P_\tau $. It also follows from the Mercer representation for $ p_\tau $ that $ G_\tau $ is a strictly positive, compact operator with the same eigenfunctions $ \phi_j $ as $ G $, corresponding to the eigenvalues $ 0 < \lambda_{\tau,j} \leq 1 $, where $ \lambda_{\tau,0}= 1$ is simple. \blue{What remains to prove Claim~(i) is to show that $ G_\tau$ is $L^2(\nu)$-Markov ergodic. We will verify this assertion following the proof of Claim~(iii).}

Turning to Claim~(ii), note that for every $j\in\num_0$, the function $ \psi_j/\lambda_j^{1/2}$ equals $ \psi_{\tau,j}/\lambda_{\tau,j}$, and thus lies in $\RKHS_\tau$. Moreover, this function lies in the same $L^2(\nu)$ equivalence class as $\phi_j$, and because the $\phi_j$ form an orthonormal basis of $L^2(\nu)$, it follows that $\RKHS_\tau$ is dense in $L^2(\nu)$. To verify the claimed inclusion relationships between $\RKHS$ and $\RKHS_\tau$, we use~\eqref{eqn:PsiTauJ} to characterize these spaces as 
\[\RKHS = \left\{ \sum_{j=0}^{\infty} a_j \psi_{j}: \; \sum_{j=0}^{\infty} |a_j|^2 < \infty \right\}, \quad \RKHS_\tau = \left\{ \sum_{j=0}^{\infty} a_j r_{\tau,j} \psi_{j} :\; \sum_{j=0}^{\infty} |a_j|^2 < \infty \right\} \]
Now note that since $\lambda_{\tau,j} = \exp [ -\tau ( 1/\lambda_j -1 ) ]$, for every $ \tau_2 > 0 $ and $ \tau_1 \in ( 0, \tau_2 ) $, we have
\begin{equation}\label{eqn:msue}
    \lambda_{\tau_2,j} / \lambda_{\tau_1,j} = \exp [ (\tau_1- \tau_2)( 1/\lambda_j -1 ) ] < 1 , 
\end{equation}
which shows that $\RKHS_{\tau_2} \subseteq \RKHS_{\tau_1}$. That $\RKHS_{\tau_1} \subseteq \RKHS$ follows from the fact that the $r_{\tau,j}$ are bounded. This completes the proof of Claim~(ii).

Next, turning to Claim~(iii), we have already established in Claim~(i) that \blue{for every $ \tau >0 $, $G_\tau $ is an $L^2(\nu)$-strictly positive, compact,  contraction on $L^2(\nu)$ with a simple eigenvalue  $\lambda_{\tau,0} = 1$.} The semigroup property follows directly from the facts the $\phi_j$ form an orthonormal eigenbasis for all $G_\tau$, $\tau \geq 0$, with eigenvalues $\lambda_{\tau,j}$, and for each $j\in\num_0$ and $ \tau_1,\tau_2 \geq 0$, $ \lambda_{\tau_1+\tau_2,j}= \lambda_{\tau_1,j} \lambda_{\tau_2,j}$. To establish strong continuity of this semigroup, it is enough to show that for every $ f \in L^2(\nu) $ and $ \epsilon > 0 $, 
\begin{equation}\label{eqLimGTau}
\lim_{\tau\to 0^+} \|(G_\tau - \Id)f\|_{L^2(\nu)} < 2\epsilon . 
\end{equation}
Indeed, expanding $f=\sum_{j=0}^{\infty} a_j \phi_j$, the partial sum $f_L = \sum_{j=0}^{L-1} a_j \phi_j$ with $L$ large-enough satisfies $\|f-f_L\|_{L^2(\nu)} < \epsilon$. Then, because
\[ (G_\tau - \Id)f = (G_\tau - \Id) f_L + (G_\tau - \Id)(f-f_L) = \sum_{j=0}^{L} a_j \left(\lambda_{\tau,j}-1\right)\phi_j + (G_\tau - \Id)(f-f_L), \]
and $\|G_\tau\| = \lambda_{\tau,0} = 1$, the last term in the above equation can be bounded as $\|(G_\tau - \Id)(f-f_L) \|_{L^2(\nu)} < 2\epsilon$. Now note that for each $j$, $\lambda_{\tau,j}-1 = \exp \left( \tau (1- \lambda_j^{-1}) \right)-1$ converges to $0$ as $\tau\to 0^+$, so that~\eqref{eqLimGTau} is satisfied. This proves Claim~(iii).  

\blue{We will now show that $G_\tau $ is $L^2(\nu)$-Markov, completing the proof of Claim~(i) and the theorem. By Hille-Yosida theory for strongly continuous, contraction semigroups of positive, compact operators \cite{ReedSimon_V2_2003methods}, there exists a positive, self-adjoint operator $ \mathcal L : D(\mathcal L) \to L^2(\nu) $ with  compact resolvent such that, for all $ \tau \geq  0 $, $ G_\tau = e^{-\tau \mathcal L}$. $\mathcal L$ is a diagonal operator with eigenbasis $\phi_j$ and corresponding eigenvalues 
\begin{displaymath}
    - \left.\frac{d\ }{d\tau} \lambda_{\tau,j}\right\rvert_{\tau=0} =  \frac{1}{\lambda_{j}} - 1. 
\end{displaymath}
In particular, since $\lambda_0 = 1$ is simple, $ \mathcal L $ has a simple eigenvalue $0$ corresponding to the constant eigenfunction $\phi_0 \equiv 1_M$. It then follows from results on Markov semigroups (e.g., \cite[Chapter~14, Theorem~2]{DellAntonio16}) that the semigroup generated by $-\mathcal L$ is actually $L^2(\nu)$-Markov ergodic. That is, for every $\tau > 0 $, $G_\tau = e^{-\tau \mathcal L}$  is a Markov operator with transition probability density $p_\tau(x,\cdot)$  relative to $\nu$. This completes the proof of Claim~(i) and Theorem~\ref{thm:Markov}}. 
\qed

Before proceeding with the proof of Theorem~\ref{thm:Main}, we will state a useful proposition, which is a consequence of the semigroup structure of the operator family $ \{ G_\tau \}_{\tau \geq 0 } $. In what follows, $\Pi_{L} : L^2(\mu) \to L^2(\mu)$ will denote the orthogonal projection onto the subspace spanned by $\{ \phi_0, \ldots , \phi_{L-1} \}$.

\begin{prop}
\label{prop:Semigroup}Under the assumptions of Theorem~\ref{thm:Main}:
\begin{enumerate}[(i)]
\item As $ \tau \to 0^+$, $ G_{\tau}^{-1/2} $ converges pointwise to the identity on $ H_\infty $. 
\item For every $ \tau > 0 $, the compactified generator $ \tilde V_\tau : L^2(\mu) \to L^2(\mu) $ from Assumption~\ref{assump:A4} is equal to $G_{\tau}^{1/2} V G_{\tau}^{1/2}$.
\item For every $\tau >0$, $A_\tau$, $B_\tau$, $\tilde V_\tau$, and $W_\tau$ are trace class operators.
\item The operator families $ \{ A_\tau \}_{\tau>0} $, $ \{ B_\tau \}_{\tau>0} $, $ \{ \tilde V_\tau \}_{\tau > 0 } $, and $ \{ W_\tau \}_{\tau>0} $ are $ p2$-continuous.
\item As $\tau\to0^+$, $A_\tau$, $B_\tau$, $\tilde V_\tau $, and $ \mathcal{U}^*_\tau W_\tau \mathcal{U}_\tau $ converge pointwise to $V$ on $D(V)$.
\end{enumerate}
\end{prop}

\begin{proof} By \eqref{eqn:msue}, for every $j$, $\lambda_{\tau,j}$ increases strictly monotonically as $\tau \to 0^+$, which means that $\lambda_{\tau,j}^{-1/2}$ decreases strictly monotonically. Now, since $D(G_\tau^{-1/2}) = D(\mathcal{N_\tau}) $ (see Section~\ref{secReview_RKHS}) and $G_\tau^{-1/2} : \phi_j \mapsto \lambda_{\tau,j}^{-1/2} \phi_j$, for every $ f \in H_\infty $ and $ 0 < \tau' < \tau $, we have $ \|G_{\tau'}^{-1/2} f \|_{L^2(\mu)} \leq \| G_{\tau}^{-1/2} f \|_{L^2(\mu)} $, and thus
\[ H_\infty \subseteq D(G_\tau^{-1/2} ) \subseteq D(G_{\tau'}^{-1/2} ). \]
Therefore, fixing $\epsilon>0$ and $\tau_0>0$, it is enough to show that
\begin{equation}
\label{eqGInvLim}
\lim_{\tau\to 0^+ }\| G_\tau^{-1/2} f - f \|_{L^2(\mu)} = \lim_{\tau\to 0^+} \| (G_\tau^{-1/2} - \Id) f \|_{L^2(\mu)} < 2\epsilon, \quad \forall f\in D(G_{\tau_0}^{-1/2}). 
\end{equation}
To that end, we begin by using the triangle inequality to write down the bound 
\begin{equation}\label{eqn:vsn}
\| (G_\tau^{-1/2} - \Id) f \|_{L^2(\mu)} \leq \| (G_\tau^{-1/2} - \Id) \Pi_{L} f \|_{L^2(\mu)} + \| (G_\tau^{-1/2} - \Id) (\Id-\Pi_{L}) f \|_{L^2(\mu)} .
\end{equation}
Now, since $G_{\tau}$ and $G_\tau^{-1/2}$ are diagonal operators, they commute with $\Pi_{L}$ and $\Id-\Pi_{L}$. As a result, for every $ \tau \in (0, \tau_0 ) $, by \eqref{eqn:msue},
\[ \| (G_\tau^{-1/2} - \Id) (\Id-\Pi_{L}) f \|_{L^2(\mu)} \leq \| (G_{\tau_0}^{-1/2} - \Id) (\Id-\Pi_{L}) f \|_{L^2(\mu)} =  \| (\Id-\Pi_{L}) (G_{\tau_0}^{-1/2} - \Id) f \|_{L^2(\mu)} , \]
and the last term vanishes as $L\to \infty$. Therefore, for $L$ large-enough, the second term on the right-hand side of \eqref{eqn:vsn} is less than $\epsilon$ for every $ \tau \in (0, \tau_0) $. Similarly, for any fixed $L$, for $\tau$ small-enough, the first term is also less than $\epsilon$, proving \eqref{eqGInvLim} and Claim (i).

To prove Claim~(ii), note that for every $ \tau' > 0 $, 
\begin{displaymath}
\tilde V_{2\tau'} \supseteq G_{2\tau'}^{1/2} V G_{2\tau'}^{1/2} = G_{\tau'} V G_{\tau'},
\end{displaymath}
where the last equality follows from the semigroup structure of $ \{ G_\tau \}_{\tau \geq 0}$. However, the range of $G_\tau$ lies in the domain of $V$, so we conclude that $ \tilde V_{2\tau'} = G_{\tau'} V G_{\tau'} = G_{\tau'} A_{\tau'} $. Setting $ \tau' = \tau/2 $ and noting that $ G_{\tau/2} = G_\tau^{1/2} $ leads to the claim.

Next, to prove Claim~(iii), observe that $ \tilde V_{\tau} = G_{\tau/2} B_{\tau/2} $, which shows that $ \tilde V_{\tau} $ is trace class since $ G_{\tau/2} $ is trace class and $ B_{\tau/2} $ is bounded. Similarly, we have $B_\tau = B_{\tau/2} G_{\tau/2}$, which shows that $ B_\tau$ is trace class. That $W_\tau $ and $ A_\tau $ are trace class then follows from the fact that the former is unitarily equivalent to $\tilde V_\tau$ and the latter equal to the negative adjoint of $B_\tau$.

Turning to Claim~(iv), we will only prove p2-continuity for $ \{ \tilde V_\tau \}_{\tau>0} $. The result for $ \{ W_\tau \}_{\tau>0} $ follows immediately by unitary equivalence of $\tilde V_\tau$ and $W_\tau$; the results for $ \{ A_\tau \}_{\tau>0} $ and $ \{ B_\tau \}_{\tau>0} $, can be verified analogously to the proof for $ \{ \tilde V_\tau \}_{\tau>0} $ below.

First, by Claim~(ii), it is sufficient to establish p2-continuity for the family of operators $\{G_\tau \VG_\tau \}_{\tau > 0}$. That is, fixing a quadratic polynomial $Q$, we have to show that the operator norm $\left\| Q\left( G_\tau \VG_\tau \right) \right\|_{L^2(\mu)}$ is a continuous function of $\tau>0$. This is in turn equivalent to showing that $\tau \mapsto Q\left( G_\tau \VG_\tau \right)$ is a continuous map in the $L^2(\mu)$ operator norm topology. Note that this continuity is not affected by the addition of a constant term to the polynomial $Q$. Thus, without loss of generality, we may assume that $Q$ is a homogeneous polynomial of the form $Q(x) = \alpha x^2 + \beta x $. By Theorems~\ref{thm:Markov} and~\ref{thm:A}, $ G_\tau $ and $ A_\tau $ are both Hilbert-Schmidt integral operators with kernels $ p_\tau $ and $ p'_\tau$, respectively. Since the composition of a bounded operator with a Hilbert-Schmidt operator is again a Hilbert-Schmidt operator, it follows that 
\[ Q\left( G_\tau \VG_\tau \right) = \alpha G_\tau \circ \VG_\tau \circ G_\tau \circ \VG_\tau + \beta G_\tau\circ \VG_\tau \]
is Hilbert-Schmidt. As a result, because the Hilbert-Schmidt norm induces a stronger topology than the $L^2(\mu)$ operator norm, it is sufficient to prove the stronger claim that $\tau \mapsto Q\left( G_\tau \VG_\tau \right)$ is a continuous map in the Hilbert-Schmidt norm topology. 

By \eqref{eqn:HilSchm}, the Hilbert-Schmidt norm of the kernel integral operator $Q\left( G_\tau \VG_\tau \right)$ is just the $L^2(\mu\times\mu)$ norm of its kernel. Thus, denoting this kernel by $q_\tau:M\times M\to \real$, the task now is to show that $\tau\mapsto \left\| q_\tau \right\|_{L^2(\mu\times\mu)}$ is a continuous function of $\tau$, or, equivalently, that $\tau \mapsto q_\tau$ is continuous in the $L^2(\mu\times\mu)$ norm topology. That this is indeed the case follows from the claims below.
\begin{enumerate}[(a)]
\item \emph{$\tau \mapsto p_\tau$ and $\tau \mapsto p'_\tau$ are continuous in the $L^2(\mu\times\mu)$ norm topology.} Indeed, by \eqref{eqn:def:schemeIII} and \eqref{eqn:PsiTauJ},
\[ \left\| p_\tau - p_{\tau'} \right\|_{L^2(\mu)}^2 = \sum_{j=0}^{\infty} \left| \lambda_{\tau,j}^{1/2} - \lambda_{\tau',j}^{1/2} \right|^2 \leq \sum_{j=0}^{L} \left| \lambda_{\tau,j}^{1/2} - \lambda_{\tau',j}^{1/2} \right|^2 + \sum_{j=L+1}^{\infty} \lambda_{\tau,j} + \sum_{j=L+1}^{\infty} \lambda_{\tau',j}, \]
so that for $L$ sufficiently large, the last two terms can be made arbitrarily small, whereas for every fixed $L$ the term $\sum_{j=0}^{L} | \lambda_{\tau_j}^{1/2} - \lambda_{\tau',j}^{1/2} |^2$ converges to $0$ as $\tau \to \tau'$. This establishes $L^2(\mu\times\mu) $ continuity of $ \tau \mapsto p_\tau $. The claim for $ \tau \mapsto p'_\tau $ follows analogously.
\item \emph{If $a_\tau, b_\tau : M \times M \to \real$ are two kernel families depending continuously on $\tau$ with respect to $L^2(\mu\times\mu) $ norm, then their composition, $ c_\tau(x,y) = \int_M a_\tau(x,z)b_\tau(z,y)\, d\mu(z) $, is also continuous.} This claim can be verified via a standard calculation in analysis, which will be omitted here. 
\end{enumerate} 
The continuity of $ \tau \mapsto q_\tau$ then follows from these results since $q_\tau $ is equal to a sum of various compositions of $ p_\tau $ and $ p_\tau' $. This completes the proof of Claim~(iv).

Finally, to prove Claim~(v), fix $f= \sum_{j=0}^{\infty} a_j \phi_j \in D(V)$, and observe the following: \begin{enumerate}[(a)]
\item \emph{ $G_\tau f$ is a family of functions in $D(V)$, converging, as $\tau \to 0^+$ to $f$. } The convergence follows from Theorem \ref{thm:Markov}(iii). Moreover, since $G_\tau$ has a $C^1$ kernel $p_\tau$, the $G_\tau f$ have $C^1$ representatives. Thus, $G_\tau f$ lies in $D(V)$, as claimed.
\item \emph{$P_\tau f$ is a Cauchy sequence in $\RKHS$.} To verify this, fix a $\tau_0>0$. Then, for every $\tau,\tau'\in (0,\tau_0)$,
\[ \| P_{\tau'}f - P_{\tau}f \|_{\RKHS}^2 = \sum_{j=0}^{\infty} \left( \lambda_{\tau',j}^{1/2}-\lambda_{\tau,j}^{1/2} \right)^2 |a_j|^2 = \sum_{j=0}^{L} \left( \lambda_{\tau',j}^{1/2}-\lambda_{\tau,j}^{1/2} \right)^2 |a_j|^2 + \sum_{j=L+1}^{\infty} \left( \lambda_{\tau',j}^{1/2}-\lambda_{\tau,j}^{1/2} \right)^2 |a_j|^2, \]
and therefore, since $\lambda_{\tau,j} \in [0,1)$ for every $\tau>0$ and $j\in\num_0$, we obtain
\[ \limsup_{\tau_0\to 0^+} \| P_{\tau'}f - P_{\tau}f \|_{\RKHS}^2 \leq \limsup_{\tau_0\to 0^+} \sum_{j=0}^{L} \left( \lambda_{\tau',j}^{1/2}-\lambda_{\tau,j}^{1/2} \right)^2 |a_j|^2 + 2\sum_{j=L+1}^{\infty} |a_j|^2 = 2\sum_{j=L+1}^{\infty} |a_j|^2 . \]
The above inequality holds for every $L\in\num$, and the last term vanishes as $L\to \infty$, proving the claim.
\item \emph{$\VG_\tau f$ is a Cauchy sequence in $L^2(\mu)$}. To verify this, note that $VP^*:\RKHS \to L^2(\mu)$ is a bounded operator by Theorem \ref{thm:B}(i), and therefore, since $P_\tau f$ is a Cauchy sequence in $\RKHS$, $A_\tau f = VP^*_\tau P_\tau f = V P^* ( P_\tau f )$ is a Cauchy sequence in $L^2(\mu)$.
\end{enumerate}

We have thus shown that $G_\tau f$ is a family of functions in $D(V)$ which converges to $f$, and their images under $V$, namely $V(G_\tau f) = \VG_\tau f$ is a Cauchy sequence. Since $V$ is a closed operator, the limit of this Cauchy sequence is equal to $Vf$. Thus, $\VG_\tau f$ converges to $Vf$, and since $ f $ was arbitrary, it follows that $A_\tau $ converges to $V$ pointwise on $D(V)$. In addition, because $ G_\tau $ is uniformly bounded and converges to the identity, we have $ G_\tau A_\tau = \tilde V_{2\tau} $, and thus $\tilde V_{\tau}$, converges pointwise to $V$ on $D(V)$. Note that we have used Claim~(ii) to deduce equality of $ G_\tau A_\tau $ and $ V_{2\tau} $. Finally, the pointwise convergence of $B_\tau = G_\tau V $ to $V$ follows directly from the pointwise convergence of $ G_\tau $ to the identity, and the result for $ \mathcal{U}_\tau^* W_\tau \mathcal{U}_\tau = \tilde V_\tau $ is obvious. This completes the proof of Proposition~\ref{prop:Semigroup}.

\end{proof}

\paragraph{Proof of Theorem \ref{thm:Main}} First, Proposition~\ref{prop:Semigroup}(iii) established that $W_\tau $ and $B_\tau $ are trace class. Claim~(i) of the theorem follows from Theorem \ref{thm:B}(ii), Claim~(ii) follows from Theorem~\ref{thm:C}, and Claim~(iii) follows from Theorem~\ref{thm:E}(i) and (viii). Aside from the convergence of $P_\tau^*Z(W_\tau)\Nyst_\tau$ to $Z(V)$ for bounded continuous (as opposed to holomorphic) functions, Claims (iv)--(vii) will follow from Theorems \ref{thm:E}, \ref{thm:F} and Proposition~\ref{prop:Semigroup}(iv) if we can show that $p_\tau$ satisfies Assumption \ref{assump:A4}. Theorem \ref{thm:Markov}(iii) establishes the condition in this assumption that $G_\tau$ converges pointwise to the identity. In order to verify Assumption~\ref{assump:A4}, it thus remains to be shown that, as $\tau\to 0^+$, $\symV_\tau$ converges pointwise to $V$ on $D(V^2)$. This follows immediately from Proposition~\ref{prop:Semigroup}(v), where we have shown the stronger result that $ \tilde V_\tau $ converges to $V$ pointwise on the whole of $D(V)$.

What remains to complete the proof of Theorem~\ref{thm:Main} is to show that $ P_\tau^* Z(W_\tau) \Nyst_\tau $ converges strongly on $H_\infty$ to $Z(V) $ for bounded continuous $Z$. By~\eqref{eqNystGInv}, for every $ f \in H_\infty $ we have
\begin{displaymath}
P_\tau^* Z(W_\tau) \Nyst_\tau f = P\tau^* \mathcal{U}_\tau \mathcal{U}_\tau^* Z( W_\tau ) \mathcal{U}_\tau G_\tau^{-1/2} f = G_\tau^{1/2} Z(\tilde V_\tau) G_{\tau}^{-1/2} f,
\end{displaymath}
and therefore
\begin{align*}
P_\tau^* Z(W_\tau) \Nyst_\tau f - Z(V) f &= G_\tau^{1/2} Z(\tilde V_\tau) G_{\tau}^{-1/2} f - Z(V) f \\
&= G_\tau^{1/2} Z(\tilde V_\tau) ( G_{\tau}^{-1/2} - \Id ) f + ( G_\tau^{1/2} Z(\tilde V_\tau) - Z(V) ) f.
\end{align*}
By Proposition~\ref{prop:Semigroup}(i) and the fact that $ G_\tau^{1/2} Z(\tilde V_\tau) $ is a uniformly bounded family of operators converging pointwise to $Z(V)$, as $ \tau \to 0^+$, each of the terms in the right-hand side of the last equation converges to 0. This shows that $ P_\tau^* Z(W_\tau ) \Nyst_\tau f \xrightarrow{s} Z(V) $ on $H_\infty$, completing the proof of Theorem~\ref{thm:Main}. \qed

Before proving Corollary~\ref{corr:APS}, we will state and prove a proposition on the $\epsilon$-approximate spectrum of $U^t$. One of the important claims we make is that, suitably restricted to the space $P^*\RKHS = D(\mathcal{N}) \subset L^2(\mu)$, $e^{t B_\tau} $ converges in norm to $U^t$, as opposed to merely strongly as shown in Theorem~\ref{thm:Main}(vi). In particular, we will consider the quantity $Q(t, \tau ) = \lVert (U^t - e^{t B_\tau} ) P^* \rVert$ for $t\in\real, \tau>0$, where $ \lVert \cdot \rVert$ denotes $\RKHS \to L^2(\mu)$ operator norm. 

\begin{prop}\label{prop:AppEigen}
Let Assumptions~\ref{assump:A1}, \ref{assump:A2} hold, with $r=2$. Then the function $Q$ is continuous, vanishes at $ t = 0 $ for every $ \tau \in ( 0, \infty) $, and converges to 0 as $ \tau \to 0^+$ for every $ t\in \real$. Moreover, for every eigenfunction $ \zeta_\tau$ of $W_\tau$ with eigenvalue $ i\omega_\tau $ and every $ t \in \real$, $e^{i\omega_\tau t}$ lies in the $\epsilon$-approximate point spectrum of $U^t$, with
\[ \epsilon = Q(t,\tau) \sqrt{\Dirich(\tilde z_\tau) + 1}, \quad \tilde z_\tau := P^*_\tau \zeta_\tau / \lVert P^*_\tau \zeta_\tau \rVert_{L^2(\mu)}, \quad \lVert U^t \tilde z_\tau - e^{i\omega_\tau t} \tilde z_\tau \rVert_{L^2(\mu)} < \epsilon. \]
Moreover, for every fixed $\epsilon>0$, $R( \epsilon,\tau)$ defined in Corollary \ref{corr:APS} diverges as $ \tau \to 0^+ $.
\end{prop}
\begin{proof} By arguments analogous to those used to prove Proposition~\ref{prop:Semigroup}(iv), the map $ ( t, \tau ) \mapsto ( U^t - e^{t B_\tau} ) P^* $ is continuous in the Hilbert-Schmidt norm topology of operators from $\RKHS$ into $L^2(\mu)$ at every $(t,\tau) \in \real \times \real_+$. This implies continuity of $ ( t, \tau ) \mapsto (U^t - e^{t B_\tau})$ in the operator norm topology, and thus continuity of $Q$. That $Q(t,\tau)$ vanishes as $\tau \to 0^+$ at fixed $t$ follows from the fact that $ U^t - e^{t B_\tau}$ converges pointwise to 0, and $P^*$ is compact. That $Q(0,\cdot) = 0 $ is obvious. Next, to verify that $e^{i\omega_\tau t}$ lies in the $\epsilon$-approximate point spectrum of $U^t$ with $\epsilon= Q_t(\tau) \sqrt{\Dirich(\tilde z_\tau) + 1}$, we use Theorem~\ref{thm:C}(ii) to compute
\begin{align}
\nonumber\lVert U^t \tilde z_\tau - e^{i\omega_\tau t} \tilde z_\tau \rVert_{L^2(\mu)} &= \frac{\lVert U^t P^*_\tau \zeta_\tau - e^{i\omega_\tau t}P^*_\tau \zeta_\tau\rVert_{L^2(\mu)}}{\lVert P^*_\tau \zeta_\tau \rVert_{L^2(\mu)}} 
= \frac{\lVert U^t P^*_\tau \zeta_\tau - P^*_\tau e^{t W_\tau} \zeta_\tau \rVert_{L^2(\mu)}}{\lVert P^*_\tau \zeta_\tau \rVert_{L^2(\mu)}} \\
\nonumber&= \frac{\lVert (U^t - e^{t B_\tau}) P^*_\tau \zeta_\tau \rVert_{L^2(\mu)}}{\lVert P^*_\tau \zeta_\tau \rVert_{L^2(\mu)}} 
= \frac{\lVert (U^t - e^{t B_\tau}) P^* \zeta_\tau \rVert_{L^2(\mu)}}{\lVert P^* \zeta_\tau \rVert_{L^2(\mu)}}\\
\label{eqAPSBound} & \leq \lVert (U^t - e^{tB_\tau}) P^* \rVert \frac{\lVert \zeta_\tau \rVert_{\RKHS}}{ \lVert P^*\zeta_\tau \rVert_{L^2(\mu)}} = Q( t,\tau) \sqrt{\Dirich(\tilde z_\tau) + 1}, 
\end{align}

Finally, fix an $ \epsilon > 0 $. It follows by continuity of $Q$, that for every $ T >0$ and $\tau>0$ small enough, then for every $ t \in [-T,T]$, $ Q(t, \tau) < \epsilon $. This implies that $R(\epsilon,\tau ) > T $ for small enough $\tau$. Since $T $ was arbitrary, we can conclude that $R(\epsilon,\tau)$ diverges as $ \tau \to 0^+$. 
\end{proof}

\paragraph{Proof of Corollary~\ref{corr:APS}} The first inequality follows from the definition of $Q(t,\tau)$ and $R(\epsilon,\tau)$, in conjunction with~\eqref{eqAPSBound}. Next, to prove Claim~(i), it is sufficient to show that $e^{i\omega t} $ lies in the spectrum of $U^t$ for every $ t \in \real$. To that end, we use the triangle inequality and the fact that $ \lVert \tilde z_\tau \rVert_{L^2(\mu)} = 1$ to obtain the bound
\begin{equation}
\label{eqAPSBound2}
\lVert U^t \tilde z_\tau - e^{i\omega t } \tilde z_\tau \rVert_{L^2(\mu)} \leq \lVert U^t \tilde z_\tau - e^{i \omega_\tau t} \tilde z_\tau \rVert_{L^2(\mu)} + \lvert e^{i \omega_\tau t} - e^{i \omega t } \rvert, \quad \forall \tau \in \real_+.
\end{equation}
Now, because $ \lim_{\tau \to 0^+} \omega_\tau = \omega $, there exists $ \tau_0 > 0 $ such that for all $ \tau \in ( 0, \tau_0)$, $\lvert e^{i \omega_\tau t} - e^{i \omega t} \rvert < \epsilon/ 2$. Moreover, because $T(\epsilon,\tau)$ is unbounded, there exists $ \tau_1 \in ( 0, \tau_0 ] $ such that $t$ lies in the interval $ (-T(\epsilon/2,\tau), T(\epsilon/2,\tau) )$ from Claim~(i) for all $ \tau \in ( 0, \tau_1) $. As a result, the bound in~\eqref{eqAPSBound} becomes
\begin{displaymath}
\lVert U^t \tilde z_\tau - e^{i \omega t} \tilde z_\tau \rVert_{L^2(\mu)} \leq \epsilon / 2 + \epsilon /2 = \epsilon, \quad \forall \tau \in (0, \tau_1). 
\end{displaymath}
We therefore conclude that $ U^t - e^{i\omega t} $ has no bounded inverse, i.e., $ e^{i\omega t}$ lies in the spectrum of $U^t$, as claimed.

Claim~(ii) will be proven by contradiction. In particular, assume that there exists a sequence $\tau_j>0$ monotonically converging to $0$ as $j\to\infty$, and $\delta>0$ such that for every $j\in\num$, $\tilde{z}_{\tau_j}$ is at distance at least $\delta$ from the 1-dimensional eigenspace $\mathcal{Z}$ of $V$ corresponding to $i\omega$. Here, as a measure of distance of a vector $ z \in L^2(\mu) $ from $\mathcal{Z}$ we use $d(z,\mathcal{Z}):= \inf\{ \|z-z'\|_{L^2(\mu)} : z'\in \mathcal{Z}\}$. Since $\lVert \tilde z_{\tau_j} \rVert_{L^2(\mu)} = 1$, it follows from the boundedness of $\Dirich(\tilde z_{\tau_j})$ that $ \lVert \tilde z_{\tau_j} \rVert_{\Nyst}$ is bounded. Therefore, by compactness of the embedding of $D(\Nyst)$ into $L^2(\mu)$, $ \tilde z_{\tau_j}$ has a subsequence converging to some vector $ z \in L^2(\mu)$. By assumption on the $\tau_j$, $d(z,\mathcal{Z})$ is greater than $\delta$. We will complete the proof by showing  that $z$ lies, in fact, in $\mathcal{Z}$, leading to a contradiction. To that end, note that the condition that $\Dirich(\tilde z_\tau)$ is bounded, together with the fact that $R(\epsilon,\tau)$ diverges from Proposition~\ref{prop:AppEigen}, implies that $T(\epsilon,\tau)$ diverges. Thus, the conclusion of Claim~(i) holds, and $z$ satisfies $ \lVert U^t z - e^{i\omega t} z \rVert_{L^2(\mu)} < \epsilon$ for every $ \epsilon > 0$. We therefore conclude that $U^t z = e^{i \omega t} z$ for every $t\in\real$, i.e., that $z$ lies in $\mathcal{Z}$, in contradiction with the assumption that $ d(z,\mathcal{Z})> \delta$. This completes the proof of Claim~(ii). \qed

\paragraph{Proof of Corollary \ref{thm:predic}} Since $ \{ \phi_0, \phi_1, \ldots \} $ is an orthonormal basis of $L^2(\mu)$, for $L$ large enough, $f_L := \Pi_L f$ satisfies $\left\| f - f_L \right\|_{L^2(\mu)} < \epsilon/2$. Moreover, since $ \{ U^t \}_{t\in\real}$ is a unitary group, the inequality $\left\| U^t f - U^t f_L \right\|_{L^2(\mu)} < \epsilon/2$ is preserved for all $ t \in \real$. Moreover, $f_L$ lies in $H_\infty$ as it is a finite linear combination of the $ \phi_j$. Now define $\hat f_\epsilon = \Nyst f_L$, so that $\hat f_\epsilon\in \RKHS_\infty$, and $P^* \hat f_\epsilon = f_L$. An application of Theorem~\ref{thm:Main}(v) with $Z(i\omega) = e^{i\omega t} $ then shows that, as $ \tau \to 0^+ $, $ \lVert U^t f_L - P^*_\tau e^{tW_\tau } \hat f_\epsilon \rVert_{L^2(\mu)} $ converges to zero, where the convergence is uniform for $ t \in \mathcal{T} $ by continuity of the map $ t \mapsto U^t f_L - P_\tau^* e^{t W_\tau} \hat f_\epsilon$. Therefore, there exists $ \tau_0>0 $ such that for all $ \tau \in (0, \tau_0) $ and $ t \in \mathcal{T} $, $\lVert U^t f_L - P^*_\tau e^{i tW_\tau} \hat f_\epsilon \rVert_{L^2(\mu)} < \epsilon/2 $. Corollary \ref{thm:predic} is then proved by the bound
\[ \left\| U^t f - P_\tau^* e^{it\KVKst_\tau} \hat f_\epsilon \right\|_{L^2(\mu)} < \left\| U^t f - U^t f_L \right\|_{L^2(\mu)} + \left\| U^t f_L - P_\tau^* e^{it\KVKst_\tau} \hat f_\epsilon \right\|_{L^2(\mu)} <\epsilon/2 + \epsilon/2 = \epsilon.\qed \]

%-_-_-_-_-_-_-_-_-_-_-_-_-_-_-_-_-_-_-_-_-_-_-_-_-_-_-_-_-_-_-_-_-_-_-_-_-_-_-_-_-_-_-_-_-_-_-_-_-_-_-_-_-_-_-_-_-_-_-_-_-_-_-_-_-_-_-_-_-_-_-_-_-_-_-_-
\section{Data-driven approximation }\label{sect:numerics}

We now take up the problem of approximating the operators in Theorems~\ref{thm:Markov} and~\ref{thm:Main} from a finite time series of observed data and without prior knowledge of the dynamical flow $ \Phi^t$. Specifically, we consider that available to us is a time series $ F(x_0), F(x_1), \ldots, F(x_{N-1}) $, consisting of the values of an observation function $ F : M \to Y $ that takes values in a data space $Y$, sampled at a fixed time interval $ \Delta t > 0$ along an orbit $ x_0, x_1, \ldots, x_{N-1} $ of the dynamics. As already alluded to in Section~\ref{sec:intro}, besides the lack of knowledge of the dynamical flow map $ \Phi^t $, this task presents a number of obstacles, including:
\begin{enumerate}[(i)]
    \item In general, one does not have direct access to the ergodic invariant measure $\mu$ and the associated $L^2(\mu)$ space, but is limited to working with the sampling measure $ \mu_N = \sum_{n=0}^{N-1}\delta_{x_n} /N$ supported on the finite trajectory $\{x_0 , \ldots, x_{N-1}\}$. In fact, even if $\mu$ were explicitly known, its support $X$ would typically be a non-smooth subset of the ambient manifold $M$, of zero Lebesgue measure (e.g., a fractal attractor), significantly hindering the construction of orthonormal bases of $L^2(\mu)$ by restriction of smooth basis functions defined on $M$.

    \item In many experimental scenarios, the sampled states will not lie exactly on the invariant set $X$, as it is not feasible to achieve complete convergence of the trajectory to that set.
    \item Measurements are not taken continuously in time, preventing direct evaluation of the action of the dynamical vector field $\vec V$ on functions.

\end{enumerate}

To address the first two issues, we take advantage of the fact that in many ergodic dynamical systems encountered in applications, the statistical properties of observables with respect to the sampling measures associated with a suitable class of initial points $x_0 $ coincide with those of the invariant measure \cite{SRB_young}, as discussed below.

\paragraph{Basin of a measure} The basin of an invariant measure $\mu$ is the set of initial points such that the sampling measures $ \mu_N$ on the trajectories starting from them converge weakly to $\mu$. More specifically, it is the set of points $x_0\in M$ such that for every continuous function $ f \in C^0(M) $,
\[ \lim_{N\to\infty} \int_M f \, d\mu_N = \int_M f \, d\mu, \quad \mu_N = \frac{1}{N} \sum_{n=0}^{N-1} \delta_{x_n}, \quad x_n = \Phi^{n \, \Delta t}(x_0). \]
This set will be denoted $B_\mu$. If $\mu$ is ergodic, as assumed throughout this work, then $\mu$-a.e.\ point in $M$ lies in its basin. The invariant measure $ \mu $ is said to be \emph{physical} if $B_\mu$ has nonzero measure with respect to some reference measure in the ambient manifold $M$. For instance, in typical experimental scenarios, initial points are drawn from some distribution equivalent to a smooth volume measure on $M$. In such cases, physicality of $ \mu$ ensures convergence of the data-driven techniques for a ``large'' set of initial conditions. While, in what follows, we will not require that $\mu$ be physical as an explicit assumption, it should be kept in mind that some type of physicality is oftentimes an implied assumption in practical applications.

\paragraph{Finite-difference approximation} Following \cite{Giannakis15,Giannakis17,DasGiannakis_delay_2019}, to address the discrete-time sampling of the data, we approximate the action of the dynamical vector field $ \vec V $ on $C^r $ functions using finite differences. As a concrete example, a scheme appropriate to the $C^1$ regularity in Theorem~\ref{thm:Main} and Assumption~\ref{assump:A3} is a central finite-difference scheme $ \vec V_{\Delta t} : C^0(M) \to C^0(M) $, given by
\begin{equation}\label{eqFD}
\vec V_{\Delta t} f(x) = \frac{f(\Phi^{\Delta t}(x)) - f(\Phi^{-\Delta t}(x))}{2\, \Delta t}. 
    %\vec V_{\Delta t} f(x) = \frac{f(x) - f(\Phi^{-\Delta t}(x))}{ \Delta t}. 
\end{equation}
By compactness of $M$, for any $ f \in C^1(M)$, the error $ \lVert \vec V_{\Delta t} f - \vec V f\rVert_{C^0(M)}$ of this scheme vanishes as $\Delta t \to 0$, and is $o(\Delta t)$ and $O((\Delta t)^2)$ if $f$ lies in $C^2(M)$ or $C^3(M)$, respectively.

The assumptions underlying our data-driven approximation schemes are as follows.

\begin{Assumption}\label{assump:A3}
    The dataset $ \{ y_0, \ldots, y_{N-1} \} $ consists of the values $ y_n = F(x_n) $ of an injective, $C^1$ observation map $F:M\mapsto Y $ into a manifold $Y$, sampled along a trajectory $ x_0, \ldots, x_{N-1} $, $ x_n = \Phi^{n\,\Delta t}(x_0) $, starting from a point $x_0 \in B_\mu $ which is not a fixed point of the dynamics. Moreover: 
\begin{enumerate}[(i)]
\item The sampling interval $ \Delta t $ is such that $ \mu $ is an ergodic invariant measure of the map $ \Phi^{\Delta t} : M \to M $. 
\item $ \kappa : Y \times Y \mapsto \real $ is a $C^1$ symmetric, strictly positive-definite kernel with $ \kappa > 0 $.
\end{enumerate}
\end{Assumption}
\blue{Note that Assumption~\ref{assump:A3}(i) is satisfied iff $\omega \, \Delta t$ is not a multiple of $2\pi$ for any Koopman eigenfrequency $\omega$. For dynamics on a separable space, there can only be countably many such $\omega$, and thus Assumption~\ref{assump:A3}(i) is satisfied for every $\Delta t$ in a full-measure, co-countable subset of the real line.} The manifold $Y$ will be referred to as the data space. While it usually has the structure of a linear space (e.g., $ Y = \real^m$), in a number of scenarios $Y$ can be nonlinear (e.g., directional measurements with $Y = S^2$). 

The techniques described below will be based on the kernel $k : M \times M \to \real$, 
\begin{equation}\label{eqKPullback}
k(x,x') := \kappa\left( F(x), F(x') \right), 
\end{equation}
induced from the kernel $ \kappa $ on data space. Note that $ k(x,x') $ can be evaluated given the data points $F(x) $ and $F(x') $, without explicit knowledge of the underlying dynamical states $ x $ and $ x'$. Moreover, the assumptions on $\kappa$ and $F$ in Assumption~\ref{assump:A3} ensure that $k$ is also a $C^1$ symmetric, strictly positive-definite kernel. We will discuss how to construct $\kappa$ when the injectivity condition on $F$ is not satisfied below.

\paragraph*{Data-driven Hilbert spaces} Since the starting point $x_0 $ is not a fixed point, and $\mu$ is an ergodic invariant measure of $ \Phi^{\Delta t}$, all sampled states $x_0, \ldots, x_{N-1} $ are distinct. Therefore $L^2(\mu_N)$, is an $N$-dimensional Hilbert space, equipped with the inner product $\langle f,g\rangle_{\mu_N} := \sum_{n=0}^{N-1} f^*(x_n) g(x_n)/ N$. This space consists of equivalence classes of complex-valued functions on $M$ having common values at $x_0,\ldots, x_{N-1}$ (i.e., the support of $\mu_N$). It is clear that $L^2(\mu_N)$ is isomorphic to the space $\cmplx^N$ equipped with a normalized Euclidean inner product. Note that one issue with establishing convergence of data-driven approximation techniques in this setting is that there is no obvious way of comparing functions in $L^2(\mu_N)$ and $L^2(\mu)$. Here, we avoid this issue by performing our approximations in suitable RKHSs, whose elements can be projected into both $L^2(\mu_N)$ and $L^2(\mu)$. The main elements of our approach, which closely parallel the theoretical results in Section~\ref{sect:Assump}, are (i) construction of a family of $L^2(\mu_N) $-Markov kernels with its associated semigroup and RKHSs, $ \RKHS_{\tau,N}$; (ii) construction skew-adjoint operators $ W_{\tau,N} $ on $\RKHS_{\tau,N}$ approximating the compactified generator $W_\tau$, and evaluation of the spectral decomposition and functional calculus of these operators; and (iii) prediction of observables by exponentiation of the data-driven generators. We will now describe these procedures, and then, in Theorem~\ref{thm:data_predic}, establish their convergence in the limit of large data. Pseudocode implementing our approach is included in Algorithms~\ref{alg:data_basis}--\ref{alg:pred} in \ref{sec:algo}. 

\paragraph{Markov kernels} Using the bistochastic normalization procedure described in Section~\ref{secReview_RKHS} with $\nu$ set to the sampling measure $\mu_N$ and $k$ to the pullback kernel from~\eqref{eqKPullback}, we construct a $C^1$, $L^2(\mu_N)$-strictly-positive, Markov ergodic kernel $p_N:M \times M \to \real$. We then apply the construction in \eqref{eqn:def:schemeIII} with $\nu=\mu_N$ to obtain a family of kernels $p_{\tau,N} : M \times M \to \real$, $ \tau > 0 $, which are also $L^2(\mu_N)$-strictly-positive and Markov ergodic. Associated with $ p_N $ and $p_{\tau,N}$ are RKHSs $\RKHS_N$ and $\RKHS_{\tau,N}$, respectively, as well as the corresponding integral operators $P_{N} : L^2(\mu_N) \to \RKHS_{N}$, $P_{\tau,N} :L^2(\mu_N) \to \RKHS_{\tau,N}$, $G_N = P_N^* P_N$, and $G_{\tau,N} = P_{\tau,N}^* P_{\tau,N}$. In accordance with Theorem~\ref{thm:Markov}, the latter form $L^2(\mu_N)$-ergodic Markov semigroups for each $N$, with associated eigenvalues $ 1 = \lambda_{\tau,N,0} > \lambda_{\tau,N,1} \geq \cdots \geq \lambda_{\tau,N,N-1} > 0 $, $L^2(\mu_N)$-orthonormal eigenfunctions $\{ \phi_{\tau,N,0},\ldots,\phi_{\tau,N,N-1}\}$, and $\RKHS_{\tau,N}$-orthonormal functions $ \{ \psi_{\tau,N,0}, \ldots, \psi_{\tau,N,N-1} \}$ (the latter, defined analogously to the $\psi_{\tau,j}$ in~\eqref{eqn:def:schemeIII}). The RKHSs $\RKHS_{\tau,N}$ also have associated Nystr\"om extension operators, $\Nyst_{\tau,N} : L^2(\mu_N) \to \mathcal{H_N} $. Note that because the $L^2(\mu_N)$ are finite-dimensional spaces, and the eigenvalues $\lambda_{\tau,N,j}$ are strictly positive, the Nystr\"om operators $\Nyst_{\tau,N}$ are everywhere-defined. 

\paragraph{Data-driven generator} Next, we construct finite-rank approximations of the compactified generator $ W_\tau$. For that, note that every finite-difference scheme $ \vec V_{\Delta t}$ for the dynamical vector field induces a corresponding operator $ \tilde V_{N,\Delta t}$ on $L^2(\mu_N)$. For instance, the central finite-difference scheme in~\eqref{eqFD} leads to 
\[
     \tilde V_{N,\Delta t} f(x_n) = \frac{f(x_{n+1}) - f(x_{n-1})}{2 \Delta t}, \quad n \in \{ 1, \ldots N-2 \}, \quad \tilde V_{N,\Delta t} f(x_0) = \tilde V_{N,\Delta t}(x_{N-1})= 0.
\]
While this operator is generally not skew-adjoint, it can be employed to construct a skew-adjoint operator $ V_{N,\Delta t} :L^2(\mu_N) \to L^2(\mu_N) $ by antisymmetrization, namely, 
\begin{equation}\label{eqFD2}
V_{N,\Delta t} = \frac{\tilde V_{N,\Delta t} - \tilde V_{N,\Delta t}^*}{2}.
\end{equation}
The latter is a data-driven approximation of $ V$, which adheres to our general scheme of approximating $V$ using skew-adjoint operators. Note that $ V_{N,\Delta t} $ is fully characterized through its matrix elements $ \langle \phi_{\tau,N,i}, V_{N,\Delta t} \phi_{\tau,N,j} \rangle_{\mu_N}$ in the $ \phi_{\tau,N,j} $ basis of $L^2(\mu_N)$, which are in turn computable by applying \eqref{eqFD2} to the eigenfunction time series $ \phi_{\tau,N,j}(x_n) $. 

Next, using $ V_{N,\Delta t} $, we construct the skew-adjoint operators $\KVKst_{\tau,N,\Delta t}$ on $\RKHS_{\tau,N}$, defined as
\begin{displaymath}
\KVKst_{\tau,N,\Delta t} = P^*_{\tau,N} V_{N,\Delta t} P_{\tau,N}.
\end{displaymath}
It follows by definition of the $ \psi_{\tau,N,j}$ basis functions of $\RKHS_{\tau,N}$ that the matrix elements of $W_{\tau,N,\Delta t} $ are related to those of $V_{N,\Delta t}$ by
\begin{equation}
\label{eqWMat}
\langle \psi_{\tau,N,i}, W_{\tau,N,\Delta t} \psi_{\tau,N,j} \rangle_{\RKHS_{\tau,N}} = \lambda_{\tau,N,i}^{1/2} \langle \phi_{\tau,N,i}, V_{N,\Delta t} \phi_{\tau,N,j} \rangle_{\mu_N} \lambda_{\tau,N,j}^{1/2}.
\end{equation}
Note that as $i$ and $j$ grow, the matrix elements of $ W_{\tau,N,\Delta t}$ diminish in magnitude compared to those of $V_{N,\Delta t}$ due to the decay of the eigenvalues. This is a manifestation of the RKHS regularization resulting from conjugation of $V_{N,\Delta t} $ by $P_{\tau,N}$. 

\begin{rk*} \blue{As is well known, finite-difference schemes are prone to errors if the sampling interval $\Delta t$ is not sufficiently small, or high-frequency noise is present in the data. The extent to which the data-driven generator matrix in~\eqref{eqWMat} is susceptible to these issues ultimately depends on the kernel $k_\tau $, as it governs the data-driven basis functions $ \psi_{\tau,N,j}$ appearing in the approximation. For instance, in \cite[Theorem~22]{Giannakis17} it was shown that for quasiperiodic systems, incorporating delay-coordinate maps in the construction of the kernel can remove temporally i.i.d.\ noise of arbitrarily large variance, allowing~\eqref{eqWMat} to be evaluated with ``clean'' eigenfunctions. While in systems with continuous spectrum that technique may have limitations (as adding delays would eventually suppress the kernel eigenfunctions spanning the continuous spectrum subspace $H_c$ \cite{DasGiannakis_delay_2019}), delay-coordinate maps should still be a useful tool for enhancing the noise robustness of the approximation in~\eqref{eqWMat}. Methods for controlling errors with respect to $ \Delta t $ would include composing the kernel with averaging operators to suppress high frequencies, or performing differentiation in the Fourier domain using spectral tapering methods. While exploring the efficacy of such approaches lies beyond the scope of this work, for the purposes of the methods presented in this paper, any convergent approximation of $V$ can be employed in place of $\vec V_{\Delta t}$ from~\eqref{eqFD}.}
\end{rk*}

\paragraph{Spectral truncation} In what follows, we will perform coherent pattern extraction and prediction using various spectrally truncated observables and operators. For that, we will need the orthogonal projections $ \Pi_{N,L} : L^2(\mu_N) \to L^2(\mu_N) $ and $ \Pi_{\tau,N,L} : \RKHS_{\tau,N} \mapsto \RKHS_{\tau,N}$ mapping into $ \spn\{ \phi_{N,0}, \ldots, \phi_{N,L-1} \} $ and $ \spn \{ \psi_{\tau,N,0}, \ldots, \psi_{\tau,N,L-1} \} $, respectively. With some abuse of notation, in what follows $\iota : B(M) \to L^2(\mu)$ and $ \iota_N : B(M) \to L^2(\mu_N) $ will be the canonical inclusion and restriction maps, respectively, on the space $B(M) $ of bounded, complex-valued, Borel functions on $M$, equipped with the supremum norm. With these definitions, given an observable $ f \in B(M) $ that is to be predicted, we will treat it by first mapping it into the spectrally truncated observable 
\begin{equation}
\label{eqFNL}
f_{N,L} = \Pi_{N,L} \iota_N f \in L^2(\mu_N), \quad 1 \leq L \leq N.
\end{equation}
Moreover, the operators used for coherent pattern extraction and prediction will be spectrally truncated analogs of $ W_{\tau,N,\Delta t} $, namely 
\begin{equation}
\label{eqWNL}\KVKst_{\tau,N,\Delta t}^{(L)} := \Pi_{\tau,N,L} \blue{ \KVKst_{\tau, N ,\Delta t} } \Pi_{\tau,N,L}, \quad 1 \leq L \leq N. 
\end{equation}
The reason for these spectral truncations will become clear below. Note that in applications the parameters $L$ in~\eqref{eqFNL} and~\eqref{eqWNL} need not be equal. Moreover, since $\Pi_{\tau,N,N} = \Id$,  $W_{\tau,N,\Delta t}^{(N)} $ is equal to $W_{\tau,N,\Delta t}$.

\paragraph{Coherent pattern extraction} For any given $L $, $\KVKst_{\tau,N,\Delta t}^{(L)}$ is a skew-symmetric operator of rank at most $L$. In particular, it is diagonal in an orthonormal basis of eigenfunctions $ \zeta_{\tau,N,\Delta t,j}^{(L)} \in \RKHS_{\tau,N}$, $ j \in \{ 0, \ldots, L-1 \}$, corresponding to purely imaginary eigenvalues $ i \omega_{\tau,N,\Delta t,j}^{(L)}$, i.e., 
\begin{equation}
\label{eqZetaNDeltat}
\KVKst_{\tau,N,\Delta t}^{(L)} \EigenW^{(L)}_{\tau,N,\Delta t,j} = i \omega^{(L)}_{\tau,N,\Delta t,j} \EigenW^{(L)}_{\tau,N,\Delta t,j}.
\end{equation}
The eigenfunctions $ \zeta_{\tau,N,\Delta t, j}^{(L)}$ will act as data-driven coherent observables, approximating the eigenfunctions $ \zeta_{\tau,j}$ of $W_\tau$. It should be noted that $ \zeta_{\tau,N,\Delta t,j}^{(L)}$ is a continuous function, constructed from the training data $ F( x_0 ), \ldots, F( x_{N-1} ) $, which can be evaluated at any $ x \in M$ from the corresponding value $ F( x ) $ of the observation map in $Y$. This procedure is known as out-of-sample evaluation.

The reason for working with $ W_{\tau,N,\Delta t}^{(L)}$, as opposed to the bare data-driven generator $W_{\tau,N,\Delta t}$, is twofold. First, in what follows, we will be interested in establishing a form of spectral convergence for the data-driven generators in the limit of large data---keeping $L$ fixed while increasing $N$ will allow us to ensure uniform convergence of the $ \psi_{\tau,N,\Delta t,j} $ with $ j \leq L - 1 $ to the corresponding $ \psi_{\tau,j} $. Moreover, working at a fixed $L \ll N$ allows to control the computational cost of data-driven approximations of $W_\tau$. In fact, following the computation of the $L \times L$ matrix representing $W_{\tau,N,\Delta t}^{(L)}$, the cost of acting with this operator on observables becomes \blue{independent of the much larger data size $N$.}

\paragraph{Functional calculus and forecasting} By skew-adjointness, the functional calculi of  $W_{\tau,N,\Delta t}^{(L)}$ can be conveniently constructed by applying any given function $Z: i\real \to \cmplx$ to their eigenvalues, and projecting to the corresponding eigenspaces. That is, 
\[Z( \KVKst_{\tau,N,\Delta t}^{(L)} ) = \sum_{j=0}^{L-1} Z( i \omega^{(L)}_{\tau,N,\Delta t,j} ) \langle \EigenW^{(L)}_{\tau,N,\Delta t,j}, \cdot \rangle_{\RKHS_{\tau,N}} \EigenW^{(L)}_{\tau,N,\Delta t,j}. \]
Given such a bounded continuous function $Z$ and a continuous observable $ f \in C^0(M) $, our approximation for $ Z(V) \iota f $ is the $\RKHS_{\tau,N}$ function 
\begin{equation}\label{eqn:def:pred_data}
g_{\tau,N,\Delta t, L, L'} = Z(\KVKst_{\tau,N,\Delta t}^{(L)}) \Nyst_{\tau,N} f_{N,L'},
\end{equation}
where \blue{$f_{N,L'}$} is given by~\eqref{eqFNL}, and $ L, L' $ are chosen such that $ 1 \leq L' \leq L \leq N $. Here, the role of the constraint $L' \leq L $ is to control the error in the dynamical evolution of $ f_{N,L'} $ by the operator $e^{t W_{\tau,N}^{(L)}}$ as $L$ increases, keeping $L'$ fixed. As with the eigenfunctions in~\eqref{eqZetaNDeltat}, $ g_{\tau,N,\Delta t, L, L'} $ can be evaluated at an arbitrary state $ x \in M $, given knowledge of $ F( x ) \in Y $. The relationships between the various maps employed in the construction of this approximation are depicted diagrammatically in Figure~\ref{figDiagram}.

\begin{figure}
\centering\begin{tikzcd}
B(M) \arrow{r}{\iota_N} \arrow{d}[swap]{\iota} \arrow[bend right = 15]{rrd}{\Nyst_{\tau,N,L'}} &L^2(\mu_N) \arrow{r}{\Pi_{N,L'}} &\spn\{\phi_{N,0},\ldots,\phi_{N,L'-1}\} \arrow{d}{\Nyst_{\tau,N}} \\
L^2(\mu) \arrow{d}[swap]{Z\left(V\right)} &\ &\RKHS_{\tau,N} \arrow{d}{Z\left( \KVKst^{(L)}_{N,\tau} \right)} \\
L^2(\mu) &L^2(\mu) \arrow[dashed]{l}[swap]{\mbox{error}} &\RKHS_{\tau,N} \arrow{l}[swap]{\iota} 
\end{tikzcd}
\caption{\label{figDiagram}Diagram illustrating the construction of the data-driven forecast function $g_{\tau,N,\Delta t,L,L'}$ from~\eqref{eqn:def:pred_data} and its relationship to $Z(V)\iota f$. Starting with a bounded observable $ f \in B(M) $, the left loop in the diagram leads to $Z(V) \iota f$, and the right loop to $ \iota g_{\tau,N,\Delta t, L, L'}$. Note that $Z(\KVKst_{\tau,N,\Delta t}^{(L)})$ maps into the $L$-dimensional subspace of $\RKHS_{\tau,N}$ spanned by $\{\psi_{\tau,N,0},\ldots,\psi_{\tau,N,L-1} \}$. The dashed arrow indicates discrepancy (error) between the data-driven function $ g_{\tau,N,\Delta t,L,L'}$ and the true observable, $Z(V) \iota f$. The composition of maps $\Nyst_{\tau,N} \circ \Pi_{N,L'} \circ \iota_N$ has been demarcated separately as an operator $\Nyst_{\tau,N,L'}$. This operator represents an entire data-driven procedure which takes as input a $B(M)$ function $f$, projects it onto its first $L'$ components of a basis for $L^2(\mu_N)$, and then outputs the Nystr\"om extension in $\RKHS_{\tau,N}$. This output can then be the input of any operator $Z( \KVKst^{(L)}_{N,\tau})$, as above.} 
\end{figure}
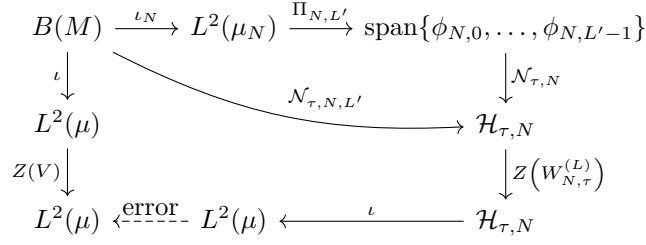

With these constructions, we have the following convergence result.

\begin{theorem}[Data-driven approximation]\label{thm:data_predic}
Let Assumptions \ref{assump:A1}, \ref{assump:A2}, and \ref{assump:A3} hold. Then:
\begin{enumerate}[(i)]
\item Every eigenfrequency $ \omega_{\tau,j} $ of $ W_\tau $, $ \tau > 0 $, can be consistently approximated by the eigenfrequencies $ \omega_{\tau,N,\Delta t,j}^{(L)} $ of $ W_\tau^{(L)}$, in the sense that 
$%\begin{displaymath}
\lim_{L \to \infty} \lim_{\Delta t\to 0, N\Delta t\to\infty} \omega^{(L)}_{\tau,N,\Delta t,j} = \omega_{\tau,j} .
$%\end{displaymath}
\item For every eigenfunction $ \zeta_{\tau,j} $ of $ W_\tau $ corresponding to $ \omega_{\tau,j} $, there exist eigenfunctions $\zeta_{\tau,N,\Delta t}^{(L)} $ of $ W_{\tau,N,\Delta t}^{(L)} $ corresponding to $ \omega_{\tau,N,\Delta t, j}^{(L)} $ such that
$%\[ 
\lim_{L\to\infty}\lim_{\Delta t\to 0, N\Delta t\to\infty} \; \| \EigenW^{(L)}_{\tau,N,\Delta t,j} - \EigenW_{\tau, j} \|_{C^0(M)} =0.
$%\]
\item For every bounded, continuous function $Z:i\real\to\cmplx$ and every bounded observable $f\in B(M)$, 
\begin{equation}\label{eqn:data_driven}
\lim_{L'\to\infty} \lim_{\tau\to 0^+} \lim_{L\to\infty} \; \lim_{\Delta t\to 0^+, N\Delta t\to\infty} \; \| Z(V) \iota f - \iota g_{\tau,N,\Delta t, L, L'} \|_{L^2(\mu)} =0, 
\end{equation}
where $ g_{\tau,N,\Delta t, L, L'} \in \RKHS_{\tau,N} $ is the data-driven approximation from~\eqref{eqn:def:pred_data}. 
\end{enumerate}
\end{theorem}

\begin{rk*} \blue{Theorem~\ref{thm:data_predic} establishes convergence of the data-driven approximation $ g_{\tau,N,\Delta t, L, L'}$, constructed for $L' \leq L$. Since $L'$ is the last asymptotic control parameter taken to $\infty$ in \eqref{eqn:data_driven},  an alternative formulation would be to keep $L'$ constant, and state
\[ \lim_{\tau\to 0^+} \lim_{L\to\infty} \; \lim_{\Delta t\to 0^+, N\Delta t\to\infty} \; \| Z(V) \iota f - \iota g_{\tau,N,\Delta t, L, L'} \|_{L^2(\mu)} =0, \quad \forall f\in \spn\left\{ \phi_0, \ldots, \phi_{L'-1} \right\} .  \]
In this formulation, $L'$ controls the dimension of the space of response observables on which we perform prediction, whereas $L$ controls the dimension of the hypothesis space \cite{CuckerSmale01} in which the forecast function lies. Theoretically, we fix $L'$ to attain convergence since $ \spn\{ \phi_0, \ldots, \phi_{L'-1} \} $ need not be invariant under $Z(V)$. In a numerical application with a fixed training dataset, the parameters $L$ and $L'$ can be tuned independently in a cross-validation step aiming to balance bias errors (increasing with decreasing $L,L')$ and generalization errors (increasing with increasing $L,L'$).}
\end{rk*}
%, the value of the resolution parameter $L$ is restricted by the data size $N$. On the other hand, the observable $f$ is typically not restricted to the span of a finite number of eigenfunctions, so there is no predetermined value of $L'$. In such a scenario, it is advantageous to utilize all of the $L$ eigenfunctions available by taking $L'=L$. 

An application of Theorem~\ref{thm:data_predic}(ii) for $ Z( i\omega) = e^{i\omega t} $ leads to the following corollary, establishing the convergence of the data-driven forecast functions for $ U^t f $. 

\begin{cor}
\label{corPredic}
For every $f \in B(M)$, the function $ f_{\tau,N,\Delta t, L, L'}^{(t)} \in \mathcal{H_{\tau,N}} $ defined as 
\begin{displaymath}
    f_{\tau,N,\Delta t, L, L'}^{(t)} = e^{t W_{\tau,N,\Delta t}^{(L)}} \Nyst_{\tau,N} \Pi_{N,L'} \iota_N f
\end{displaymath}
is an approximation of $ U^t f $, satisfying 
\[ \lim_{L'\to\infty} \lim_{\tau\to 0} \lim_{L\to\infty} \; \lim_{\Delta t\to 0^+, N\Delta t\to\infty} \; \| U^t f - f_{\tau,N,\Delta t, L, L'}^{(t)} \|_{L^2(\mu)} =0. \]
Moreover, the map $t \mapsto f_{\tau,N,\Delta t, L, L'}^{(t)}$ is continuous, and the convergence is uniform for $ t $ lying in compact intervals.
\end{cor}
\begin{rk*}The order in which the limits in Theorem~\ref{thm:data_predic} and Corollary~\ref{corPredic} are taken is important. In particular, the first limits taken are those of $N$ and $\Delta t$. This corresponds to the limit of large data, i.e., infinitely many samples taken at arbitrarily small sampling interval. As stated above, in order to control sampling errors and ensure spectral convergence of the data-driven operators, the limit of large data must be taken at a fixed resolution $L$. After this, the limit $L\to\infty $ is taken to facilitate a finite-rank approximation of $W_\tau $ and $Z( W_\tau) $. Next, the limit $\tau\to 0^+$ is taken as the $0$-time limit of the Markov semi-group $G_{\tau}$, leading to convergence of $W_\tau $ and $Z(W_\tau) $ to $V$ and $Z(V) $, respectively, in the sense of Theorem~\ref{thm:Main}. Finally, in Theorem~\ref{thm:data_predic}(iii) and Corollary~\ref{corPredic}, the limit $L'\to\infty$ is taken to facilitate convergence to the spectrally truncated observable $ \Pi_{L'} \iota f \in L^2(\mu) $ to $\iota f $. The latter limit is analogous to an $ \epsilon \to 0^+ $ limit of the tolerance $\epsilon $ in Corollary~\ref{thm:predic}. 
\end{rk*}

Before proving Theorem~\ref{thm:data_predic}, we will state an auxiliary lemma. In what follows, $ \Pi_{\tau,L} $ will denote the orthogonal projection on $\RKHS_\tau $, mapping into $ \spn \{ \psi_{\tau,0}, \ldots, \psi_{\tau,L-1} \} $. 

\begin{lemma}\label{lem:DataPredic}
Under the assumptions of Theorem~\ref{thm:data_predic}, the following hold:
\begin{enumerate}[(i)]
    \item The eigenvalues $ \lambda_{\tau,N,j} $ of $ G_{\tau, N} $ converge to those of $ G_\tau $, i.e., for every $ \tau \geq 0 $ and $ j \in \mathbb{N}_0 $, $ \lim_{N\to\infty} \lambda_{\tau,N,j} = \lambda_{\tau,j} $. Moreover, for every $\RKHS_{\tau} $ basis function $ \psi_{\tau,j} $ there exists a sequence of $\RKHS_{\tau,N} $ basis functions $ \psi_{\tau,N,j} $ converging to it in $C^0(M) $ norm in the same limit.
\item For every $ \tau > 0 $, the matrix elements of $ W_{\tau,N,\Delta t} $ from~\eqref{eqWMat} converge to the corresponding matrix elements of $ W_{\tau} $, i.e., 
$%\[
\lim_{\Delta t \to 0^+, N \Delta t \to \infty} \langle \psi_{\tau,N,i},W_{\tau,N,\Delta t} \psi_{\tau,N,j} \rangle_{\RKHS_{\tau,N}} = \langle \psi_{\tau,i}, W_\tau \psi_{\tau,j} \rangle_{\RKHS_\tau}.
$%/]
\item As $ L \to \infty $, the finite-rank, skew-adjoint operators $ W_{\tau}^{(L)} := \Pi_{\tau,L} W_\tau \Pi_{\tau,L} $ converge to $ W_\tau $ in Hilbert-Schmidt norm, and thus in $ \RKHS_\tau $ operator norm. 
\end{enumerate}
\end{lemma}

\begin{proof}
Claim~(i) was proved in \cite{DasGiannakis_delay_2019}, following the approach of \cite{VonLuxburgEtAl08}, for Markov kernels constructed via the kernel normalization procedure introduced in the diffusion maps algorithm \cite{CoifmanLafon06}. The result for the bistochastic Markov kernels from~\eqref{eqn:def:bistochastic} follows analogously. 

To verify Claim~(ii), note first that the matrix elements $ \langle \phi_{N,i}, V_{N,\Delta t} \phi_{N,j} \rangle_{L^2(\mu_N)} $ converge to $ \langle \phi_i, V \phi_j \rangle_{L^2(\mu)} $ by convergence of the finite-difference approximation in~\eqref{eqFD} for $C^1$ functions, in conjunction with the fact that the measure $ \mu $ is physical; see \cite{DasGiannakis_delay_2019} for further details. The convergence of the $ \langle \psi_{\tau,N,i},W_{\tau,N,\Delta t} \psi_{\tau,N,j} \rangle_{\RKHS_{\tau,N}} $ to $ \langle \psi_{\tau,i}, W_\tau \psi_{\tau,j} \rangle_{\RKHS_\tau} $ then follows from this result in conjunction with Claim~(i). 

Claim~(iii) follows from the fact that $ \{ \psi_{\tau,ij} := \langle \psi_{\tau,j}, \cdot \rangle_{\RKHS_\tau} \psi_{\tau,i}: i, j \in \mathbb{N}_0 \} $ is an orthonormal basis of the Hilbert space of Hilbert-Schmidt operators on $ \RKHS_\tau $, and in this basis, every Hilbert-Schmidt operator $ T : \RKHS_\tau \to \RKHS_\tau $ can be decomposed as $ T = \sum_{i,j=0}^\infty \langle \psi_{\tau,i}, T \psi_{\tau,j} \rangle_{\RKHS_\tau} \psi_{\tau,ij} $. In this expansion, the partial sums $ \sum_{i,j=0}^{L-1} \langle \psi_{\tau,i}, T \psi_{\tau,j} \rangle_{\RKHS_\tau} \psi_{\tau,ij} $ are equal to $ \Pi_{\tau,L} T \Pi_{\tau,L} $, and converge in Hilbert-Schmidt norm. Applying these results for $ T = W_\tau $ leads directly to the claim. 
\end{proof}

%\paragraph{Remark} The choice of $L$ is independent of the prediction time $t$, which is an important requirement for a numerical implementation of $\hat{f}_{\tau,N,\Delta t, L, L'}$, since one can only compute only a finite number of eigenfunctions with reasonable accuracy. Notice that as $\tau$ is changed, the eigenfunctions $\EigenGV_{\tau,j}$ and $\EigenGV'_{\tau,j}$ just get scaled, so they need not be computed separately for every $\tau>0$. 

\paragraph{Proof of Theorem~\ref{thm:data_predic}} 
Because, by Lemma~\ref{lem:DataPredic}(iii), $ W_\tau^{(L)} $ is a sequence of compact operators converging in operator norm to the compact operator $ W_\tau $, it follows that for every $ j \in \mathbb{N}_0$ such that $ \omega_{\tau,j} \neq 0 $, the eigenvalues $ i \omega_{\tau,j}^{(L)} $ of $ W_\tau^{(L)} $ converge to $ i \omega_j $. Moreover, since, as follows directly from their definition, all $ W_\tau^{(L)} $ have an eigenvalue at zero corresponding to constant eigenfunctions, we conclude that the convergence $ \omega_{\tau,j}^{(L)} \xrightarrow[L\to\infty]{} \omega_{\tau,j} $ holds for all $ j \in \mathbb{N}_0 $. The convergence of the eigenvalues implies in turn that for every eigenfunction $ \zeta_{\tau,j} $ of $W_\tau $ there exists a sequence of eigenfunctions $ \zeta_{\tau,j}^{(L)} $ of $ W_\tau^{(L)} $ converging to it in $\RKHS_\tau$ norm as $ L \to \infty $. Claims~(i) and~(ii) will then follow if it can be shown that the eigenvalues of $ W_{\tau,N,\Delta t }^{(L)} $ converge to those of $ W_\tau^{(L)} $, and the corresponding eigenfunctions converge in $C^0(M) $ norm. 

The convergence of the eigenvalues $ i \omega_{\tau,N,\Delta t}^{(L)} $ to $ i \omega_{\tau}^{(L)} $ follows from the convergence of the matrix elements of $ \KVKst^{(L)}_{\tau,N,\Delta t} $ to $ W_\tau^{(L)} $, as established in Lemma~\ref{lem:DataPredic}(ii). The existence of eigenfunctions $ \zeta_{\tau,N,\Delta t}^{(L)} $ of $ W_{\tau,N,\Delta t}^{(L)} $ converging to $ \zeta_{\tau,j}^{(L)} $ in $C^0(M) $ norm follows from the fact that both $ \zeta_{\tau,N,\Delta t}^{(L)} $ and $ \zeta_{\tau,j}^{(L)} $ are expressible as finite linear combinations of the $ \psi_{\tau,N,j} $ and $ \psi_{\tau,j} $, namely 
\begin{displaymath}
    \zeta_{\tau,N,\Delta t,j}^{(L)} = \sum_{l=0}^{L-1} c_{\tau,N,\Delta t, l,j}^{(L)} \psi_{\tau,N,l}, \quad \zeta_{\tau,j}^{(L)} = \sum_{l=0}^{L-1} c_{\tau, l,j}^{(L)} \psi_{\tau,l}, 
\end{displaymath} 
where $ \vec c_{\tau,N,j}^{(L)} = ( c_{\tau,N,\Delta t, 0,j}^{(L)}, \ldots, c_{\tau,N,\Delta t, L-1,j}^{(L)} )^\top $ and $ \vec c_{\tau,j}^{(L)} = ( c_{\tau, 0,j}^{(L)}, \ldots, c_{\tau, L-1,j}^{(L)} )^\top $ are eigenvectors of the $L\times L $ matrices representing $ W_{\tau,N,\Delta t}^{(L)} $ and $ W_{\tau}^{(L)} $, respectively. By Lemma~\ref{lem:DataPredic}(i), the $ \psi_{\tau,N,j} $ converge to $ \psi_{\tau,j} $ in $ C^0(M) $ norm, and moreover for every eigenvector $ \vec c_{\tau,j}^{(L)} $ there exist $ \vec c_{\tau,N,j}^{(L)} $ converging to it in any vector norm. We therefore conclude that $ \zeta_{\tau,N,\Delta t,j}^{(L)} $ converges in $C^0(M) $ norm to $ \zeta_{\tau,j}^{(L)} $, proving Claims~(i) and~(ii). 

Turning to Claim~(iii), we will verify that the limits in~\eqref{eqn:data_driven} hold in a sequential manner. First, defining $f_{L'} = \Pi_{L'} \iota f$, note that because $Z(V) $ is a bounded operator and the $ \Pi_{L'} $ converge pointwise to the identity, $ \lim_{L'\to\infty} \left\| Z(V)\iota f- Z(V) f_{L'}\right\|_{L^2(\mu)} =0 $. Thus, to verify \eqref{eqn:data_driven}, it suffices to show that
\begin{equation}
\label{eqn:data_driven2}\lim_{\tau\to 0^+} \lim_{L\to\infty} \; \lim_{\Delta t\to 0^+, N\Delta t\to\infty} \; \| Z(V) f_{L'} - \iota g_{\tau,N,\Delta t, L, L'} \|_{L^2(\mu)} =0.
\end{equation}
Second, observe that $f_{L'}$ lies in $ H_\infty$, and thus, by Theorem \ref{thm:Main}(v), $\lim_{\tau\to 0^+} \left\| Z(V) f_{L'} - \iota Z(W_{\tau}) \Nyst_\tau f_{L'} \right\|_{L^2(\mu)} = 0$. As a result, to prove~\eqref{eqn:data_driven2}, it is enough to show that
\begin{equation}
\label{eqn:data_driven3}\lim_{L\to\infty} \; \lim_{\Delta t\to 0^+, N\Delta t\to\infty} \; \| \iota Z(W_{\tau}) \Nyst_\tau f_{L'} - \iota g_{\tau,N,\Delta t, L, L'} \|_{L^2(\mu)} =0.
\end{equation}
Next, by Lemma \ref{lem:DataPredic}(iii), $\lim_{L\to\infty} \| Z(W_{\tau}) \Nyst_\tau f_{L'} - Z(W_{\tau}^{(L)}) \Nyst_\tau f_{L'} \|_{\RKHS_\tau} = 0$, which implies that \eqref{eqn:data_driven3} holds if it can be shown that 
\begin{equation}
\label{eqn:data_driven4}
\lim_{\Delta t\to 0^+, N\Delta t\to\infty} \; \| \iota Z(W_{\tau}^{(L)}) \Nyst_\tau f_{L'} - \iota g_{\tau,N,\Delta t, L, L'} \|_{L^2(\mu)} = 0. 
\end{equation}
The latter will follow in turn if it can be established that the vectors $ g_{N,\Delta t} = Z(\KVKst_{\tau,N,\Delta t}^{(L)}) \Nyst_{\tau,N} \Pi_{N,L'} f_{N,L'} $, with $f_{N,L'} $ given by~\eqref{eqFNL}, converge to $ g = Z(W_{\tau}^{(L)}) \Nyst_\tau \Pi_{L'} f_{L'} $ in $C^0(M)$ norm. This fact follows from arguments similar to the proof of Claim~(i). That is, writing $ g_{N,\Delta t} = \sum_{j=0}^{L-1}c_{N,\Delta t,j} \psi_{\tau,N,j}$ and $ g = \sum_{j=0}^{L-1} c_{j} \psi_{\tau,j} $, one can verify the claimed convergence from the facts that (a) the functions $ \psi_{\tau,N,j} $ converge to $ \psi_{\tau,j} $ in $C^0(M)$ norm; (b) the expansion coefficients $ c_{N,\Delta t,j}$ and $c_j$ are determined from the action of $L\times L'$ matrices representing $Z(\KVKst_{\tau,N,\Delta t}^{(L)}) \Nyst_{\tau,N} \Pi_{N,L'} $ and $Z(\KVKst_{\tau}^{(L)}) \Nyst_{\tau} \Pi_{L'} $ on the $L'$-dimensional vectors representing $f_{N,L'}$ and $f_{L'}$, respectively, all of which converge in the appropriate limit by Lemma~\ref{lem:DataPredic}. The sequence of limits in~\eqref{eqn:data_driven2}--\eqref{eqn:data_driven4} then leads to \eqref{eqn:data_driven}, proving Claim~(iii). This completes the proof of the theorem. \qed

\paragraph{Approximation errors} As stated above, the convergence results in Theorem~\ref{thm:data_predic} and Corollary~\ref{corPredic} require that the limit of large data \blue{has to} be taken before the limits involving the spectral truncation ($L$ and $L'$) and RKHS regularization ($\tau$) parameters. Yet, in practical applications, one typically works with a fixed number of samples $N$ and sampling interval $\Delta t$, and is faced with the question of tuning $L$, $L'$, and $\tau$ so as to achieve optimal performance. In particular, even though from a theoretical standpoint one would like to employ arbitrarily large $L, L'$ and arbitrarily small $\tau$, such a choice would invariably lead to overfits of the training data and/or numerical instability. In the context of prediction (i.e., Theorem~\ref{thm:data_predic}(iii) and Corollary~\ref{corPredic}), appropriate parameter values can be determined using cross-validation, i.e., by setting aside a portion of the available training data as verification data, and choosing $L$, $L'$, and $\tau$, as well as other parameters (e.g., bandwidth parameters of Gaussian kernels as in~\eqref{eqn:def:Gauss_ker} ahead), so as to maximize prediction skill in the verification dataset. In the context of spectral estimation (i.e., in the present work, coherent pattern extraction), parameter selection is more challenging, as typically there is no a priori known ground truth that can employed for cross-validation. Instead, one way to proceed is through a posteriori analysis of the results, seeking to identify eigenvalues and eigenfunctions of $W_{\tau,N,\Delta t}^{(L)}$ with minimal risk of being affected by sampling errors.

One such a posteriori metric is the Dirichlet energy of the eigenfunctions, $ \Dirich_N(P^*_N\zeta_{\tau,N,\Delta t,j}^{(L)})$, induced on $L^2(\mu_N)$ by the RKHS $\RKHS_N$ according to~\eqref{eqDirichlet}. On the basis of well known results from statistical learning theory \cite{CuckerSmale01} (generally established for i.i.d.\ data, though analogous results are expected to hold for equidistributed data in an ergodic sense, as in the present work), this functional is a useful proxy for the sensitivity of $ \zeta_{\tau,N,\Delta t,j}^{(L)}$ to sampling errors. In Corollary~\ref{corr:APS}, we established that the Dirichlet energy is also useful for identifying dynamical coherence. As a result, $\Dirich_N(P^*_N\zeta_{\tau,N,\Delta t,j}^{(L)})$ is a natural quantity to monitor for the purpose of identifying robust, data-driven coherent observables. Still, the raw Dirichlet energy does not take into account another source of error in our data-driven approximations, namely that we are approximating the unbounded generator $V$ by a finite-difference operator $V_{N,\Delta t} $ of the form in~\eqref{eqFD2}. Such operators, and as a result $W_{\tau,N,\Delta t}^{(L)}$, have $L^2(\mu_N)$ operator norm of at most $1/\Delta t$, placing an effective Nyquist limit on the eigenfrequencies $ \omega_{\tau,N,\Delta t,}^{(L)}$ that can be recovered at a given sampling interval $\Delta t$. In particular, eigenfrequencies close to that limit are expected to have high sensitivity to $\Delta t$.
The above suggests assessing the robustness of the data-driven eigenfunctions $\zeta_{\tau,N,\Delta t,j}^{(L)}$ using a functional that depends on both the Dirichlet energy and eigenfrequency. In the experiments presented in Section~\ref{sect:examples} ahead, we will employ the frequency-adjusted Dirichlet energy on $L^2(\mu_N)$ given by
\begin{equation}
    \Dirich_{N,\Delta t}(f) = \Dirich_{N}(f)\left(1- \frac{ ( \Delta t \lVert V_{N,\Delta t} f \rVert_{L^2(\mu_N)})^2}{\lVert f \rVert^{2}_{L^2(\mu_N)} } \right)^{-1}, \quad \forall f \in L^2(\mu_N) \setminus \{0\}, \quad \text{and} \quad \mathcal{D}_{N,\Delta t}(0) = 0.
\label{eqDirichletDt}
\end{equation}
By construction, this functional takes small values on functions with low roughness (in the sense of Dirichlet energy $\mathcal D_N(f)$), and thus reduced risk of sensitivity to sampling errors. \blue{Moreover, the term $ \lVert V_{N,\Delta t} f \rVert^2_{L^2(\mu_N)}$ can be thought of as a spectral energy in the frequency domain for the time series $f(x_n)$ sampled discretely at times $ t_n = n \, \Delta t $. The term $ ( 1 -  \lVert V_{N,\Delta t } f \rVert^2_{L^2(\mu_N)} / \lVert f \rVert^2_{L^2(\mu_N)} )^{-1}$ then penalizes $f$ whose frequency spectral energy is comparable to $ 1 / \Delta t^2$.}  In what follows, we will order by convention all data-driven eigenfunctions $\zeta_{\tau,N,\Delta t, j}^{(L)}$ in order of increasing $\Dirich_{N,\Delta t}(P^*_N \zeta_{\tau,N,\Delta t, j}^{(L)} )$.
We end this section with a discussion on how to obtain the kernel $ \kappa $ on data space.

\paragraph{Choice of kernel} First, note that the injectivity assumption on the observation map $F$ is with minimal loss of generality. In particular, according to the theory of delay-coordinate maps of dynamical systems \cite{SauerEtAl91}, under mild assumptions, the map $ F_Q : M \to Y^Q$, $ Q \in \mathbb{N}$, defined as
\[ F_Q(x) = \left( F(x), F(\Phi^{-\Delta t}x), \ldots, F(\Phi^{-(Q-1)\Delta t}x) \right), \]
is injective for large-enough $Q$. Moreover, $F_Q(x_n)$ can be evaluated for all states $x_n$ with $ n \in \{ Q - 1, \ldots, N \} $ given the values of $ F $ on a finite trajectory $x_0, x_1, \ldots x_{N-1}$ alone. Thus, delay-coordinate maps are a useful remedy when the observation map is non-injective, which is frequently the case \blue{with experimental or observational data acquired from high-dimensional systems (e.g., engineering or geophysical fluid flows)}.

Assuming then that the observation map $ F $ is injective, one can implement the techniques described in this section with any $C^1$ strictly positive-definite kernel on $Y$. As a concrete example for the case $ Y = \real^m$, we mention here the radial Gaussian kernels, 
\begin{equation}\label{eqn:def:Gauss_ker}
\kappa( y, y ) = \exp \left( - \frac{ d^2( y, y' ) }{ \epsilon } \right),
\end{equation} 
where $ d : Y \times Y \to \real $ is the standard Euclidean metric on $ \real^m $, and $ \epsilon $ a positive bandwidth parameter. Such kernels are popular in manifold learning techniques \cite{BelkinNiyogi03,CoifmanLafon06} due to their ability to approximate heat kernels and the spectrum of the Laplace-Beltrami operator in the $\epsilon \to 0^+$ limit. Here, we do not assume that the support $X$ of the invariant measure has manifold structure, so generally we do not have a heat kernel interpretation. Nevertheless, radial Gaussian kernels are known to be strictly positive-definite on arbitrary subsets of $\real^m$ \cite{Genton01}, which is sufficient for our purposes. The numerical experiments in Section~\ref{sect:examples} will be carried out with a variable-bandwidth variant of~\eqref{eqn:def:Gauss_ker}, whose construction and basic properties are described in \ref{appVB}.

%-_-_-_-_-_-_-_-_-_-_-_-_-_-_-_-_-_-_-_-_-_-_-_-_-_-_-_-_-_-_-_-_-_-_-_-_-_-_-_-_-_-_-_-_-_-_-_-_-_-_-_-_-_-_-_-_-_-_-_-_-_-_-_-_-_-_-_-_-_-_-_-_-_-_-_-
\section{Examples and discussion}\label{sect:examples}

In this section, we apply the procedure described in Section \ref{sect:numerics} to ergodic dynamical systems with different types of spectra. The objective is to illustrate the results of Theorems \ref{thm:Markov}, \ref{thm:Main}, \ref{thm:data_predic} and Corollaries~\ref{corr:APS}, \ref{thm:predic}, and demonstrate that the framework is effective in identifying coherent observables and performing prediction in quasiperiodic and mixing systems. We consider the following three examples:    

\begin{enumerate}[(i)]
    \item A linear, ergodic flow $ \Phi^t : \mathbb{T}^2 \to \mathbb{T}^2 $ on the 2-torus, 
\begin{displaymath}
    \Phi^t(\theta_1, \theta_2) = ( \theta_1 + \alpha_1 t, \theta_2 + \alpha_2 t ) \mod 2 \pi, 
    \label{eqn:2D_oscill}
\end{displaymath}
where $ \alpha_1 $ and $ \alpha_2 $ are rationally independent frequencies, set to $ 1 $ and $ 30^{1/2} $, respectively. The observation map $ F: \mathbb{T}^2 \to \real^3 $ is given by the standard embedding of the 2-torus into $\real^3$,  
\begin{displaymath}
    %\label{eqn:2D_oscill_obs}
    F(\theta_1, \theta_2) = (F_1, F_2, F_3 ) = \left( (1+R\cos \theta_1) \cos \theta_1, (1+R \cos \theta_2) \sin \theta_1, \sin \theta_2 \right), 
\end{displaymath}
where we set $ R = 1/2$. 
\item The L63 system \cite{Lorenz63},  generated by the $C^\infty$ vector field $ \vec V$ on $\real^3$, whose components $ ( V_1, V_2, V_3)$ at $(x,y,z)\in\real^3$ are given by
\begin{displaymath}
V_1 = \sigma(y-x), \quad V_2 = x(\rho-z) -y, \quad V_3 = xy-\beta z.
\end{displaymath}
We use the standard parameter values $ \beta = 8/3 $, $\rho = 28$, $\sigma = 10 $,  and take $F:\real^3\to\real^3$ to be the identity map. 
\item The R\"ossler system \cite{Rossler76} on $\real^3$, generated by the smooth vector field $\vec V $, with components $(V_1, V_2, V_3) $ at $(x,y,z)$ given by
    \begin{displaymath}
        V_1 = -y -z, \quad V_2 = x+ a y, \quad V_3 = b + z(x-c). 
    \end{displaymath}
    We use the standard parameter values $ a = 0.1 $, $ b = 0.1 $, $ c = 14$, and as in the case of the L63 system, set $F$ to the identity map. 
    \end{enumerate}

\paragraph{Methodology} The following steps describe sequentially the entire numerical procedure carried out for each system. Additional algorithmic details, including pseudocode are included in~\ref{appVB} and~\ref{sec:algo}. 
\begin{enumerate}
    \item Numerical trajectories $x_0, x_1, \ldots, x_{N-1} $ of length $N$, with $x_n = \Phi^{n\Delta t}(x_0)$,  were generated using a sampling interval $\Delta t>0$. In the case of the torus rotation, $ \Delta t $ was set to $2 \pi/ 500 \approx 0.013$. The sampling interval in the L63 and R\"ossler experiments was 0.01 and 0.04, respectively. In all three experiments, the \blue{nominal number of samples was $N=\text{64,000}$. In the torus case, we also show eigenfunction results for $ N = 6400 $ to assess sensitivity to sampling errors.}  The trajectories for the torus experiments were computed analytically. The L63 and R\"ossler experiments utilize numerical trajectories generated in Matlab, using the \texttt{ode45} solver. These trajectories start from arbitrary initial conditions in $\real^3$, followed by a spinup period of $N\,\Delta t $ time units before collecting the actual ``production'' data.  
    \item The observation map $F$ described for each system was used to generate the respective time series $ F( x_0 ), F( x_1 ), \ldots, F( x_{N-1} ) $. For our choices of $F $, all of these time series take values in $\real^3$. In addition, we generated time series $ f( x_0 ), f(x_1), \ldots, f(x_{N-1}) $ for various other continuous, real-valued observables for use in forecasting experiments (described below). 
        
    \item Data-driven eigenpairs $ (\lambda_{N,j}, \phi_{N,j} ) $ with $ j \in \{ 0, \ldots, L - 1 \} $ were computed by applying Algorithm~\ref{alg:data_basis} to the dataset $F(x_0), \ldots, F(x_{N-1})$. Throughout, we used the variable-bandwidth Gaussian kernel described in~\ref{appVB}, in conjunction with the bistochastic normalization procedure from Section~\ref{secReview_RKHS}. In addition, we tuned the kernel bandwidth $\epsilon$ using an automatic procedure; see~\ref{appVB} for further details and references. The number of eigenfunctions employed in our experiments ranged from $L=500$ to 1000; i.e., $L \ll N $ in all cases. As described in~\ref{appVB}, the eigenpairs $ (\lambda_{N,j}, \phi_{N,j} ) $ for the bistochastic kernels employed here can be determined from the singular values and left singular vectors of a non-symmetric $N \times N $ kernel matrix, without explicit formation of the Markov kernel matrix itself. We followed that approach here, using Matlab's \texttt{svds} iterative solver to perform the singular value decomposition (SVD). All pairwise distances in data space required for kernel evaluation were computed by brute force (as opposed to using approximate nearest-neighbor search) in Matlab, retaining $5000 \ll N $ nearest neighbors per datapoint.  

    \item Using the eigenpairs $ ( \lambda_{N,j}, \phi_{N,j} ) $ from \blue{Step~3} as inputs, Algorithm~\ref{alg:generator} was applied to form the $L\times L $ operator matrices for $ W^{(L)}_{\tau,N,\Delta t} $, and compute the corresponding eigenfrequencies $\omega_{\tau,N,\Delta t,j}^{(L)} $ and eigenfunctions $ \zeta_{\tau,N,\Delta t,j}^{(L)} \in \RKHS_{\tau,N}$. Throughout, we used the central finite-difference scheme in~\eqref{eqFD} (which, in this case, is $O( (\Delta t)^2)$-accurate) to compute the matrix elements of $W^{(L)}_{\tau,N,\Delta}$, and Matlab's \texttt{eig} solver to compute the $( \omega_{\tau,N,\Delta t}^{(L)}, \zeta_{\tau,N,\Delta t}^{(L)})$ eigenpairs. In order to investigate the dependence of the spectra of $W_{\tau,N,\Delta t}^{(L)}$ on $ \tau $ (particularly from the perspective of Corollary~\ref{corr:APS}), we computed eigenfrequencies for logarithmically spaced values of $\tau $, and examined the $ \tau \mapsto \omega_{j,\tau,N,\Delta t}^{(L)}$ dependence through scatterplots. Moreover, for each eigenfunction, we computed its frequency-adjusted Dirichlet energy $\Dirich_{N,\Delta t}( P^*_N \zeta_{\tau,N,\Delta t,j}^{(L)})$ from~\eqref{eqDirichletDt}, and ordered the eigenpairs $(\omega_{\tau,N,\Delta t}^{(L)}, \zeta_{\tau,N,\Delta t}^{(L)})$ in order of increasing $\Dirich_{N,\Delta t}(P^*_N\zeta_{\tau,N,\Delta t}^{(L)})$.  
            
    \item Forecasting experiments were performed by constructing the data-driven functions  $ f_{\tau,N,\Delta t, L, L'}^{(t_m)}$ via Algorithm~\ref{alg:pred} for lead times $t_m := m\, \Delta t $, $ m \in \num_0 $, in an interval  $[0, m_\text{max} \, \Delta t ] $. In all cases, we used $L = L'$. \blue{Initial conditions for the forecasts were generated  from a time series $F(\hat x_0), \ldots, F(\hat x_{\hat N -1}) $ of the observation map, sampled on a dynamical trajectory $ \hat x_0, \ldots,\hat x_{\hat N + m_\text{max}-1} $ of length $\hat{N} + m_\text{max} $, independent of the training data. The values $U^{t_m} f(\hat x_n) = f(\hat x_{n+m}) $ of the forecast observable on this trajectory were then used as verification data to assess the out-of-sample predictions $ f_{\tau,N,\Delta t,L,L'}^{(t_m)}( F( \hat x_n ) ) $. The latter, were evaluated using Algorithm~\ref{alg:OUS}.}  In all experiments, $\hat{N}$ was equal to $N$. Forecast errors were assessed through the $L^2$ norm associated with the sampling measure $ \hat \mu_{\hat N} = \sum_{n=0}^{\hat N -1 } \delta_{\hat x_n} / \hat N $. Specifically, for a given lead time $t_m$ we compute a normalized \blue{root mean square error (RMSE)} metric, 
        \begin{displaymath}
            \blue{\varepsilon(t_m) = \lVert U^{t_m} f -   f_{\tau,N,\Delta t, L, L'}^{(t_m)} \rVert_{L^2(\hat \mu_{\hat N})} / \lVert f \rVert_{L^2(\hat \mu_{\hat N})}}, 
        \end{displaymath}
        such that values $ \varepsilon(t_m) \ll 1$ correspond to skillful forecasts, whereas $ \varepsilon(t_m) \simeq 1$ indicates loss of skill. \blue{The metric $ \varepsilon(t_m)$ is an empirical estimator of the normalized expected error with respect to the invariant measure, $ \lVert U^{t_m} f -   f_{\tau,N,\Delta t, L, L'}^{(t_m)} \rVert_{L^2(\mu)} / \lVert f \rVert_{L^2(\mu)} $, to which it converges almost surely as $ \hat N \to \infty $.}  
\end{enumerate}

We now present and discuss the experimental results for each system. Hereafter, for notational simplicity, we will drop $N$, $L$, and $\Delta t $ subscripts and superscripts from data-driven eigenfrequencies, eigenfunctions, and operators. We will also use $\Dirich_{\tau,j}$ as a shorthand notation for the frequency-adjusted Dirichlet energy from~\eqref{eqDirichletDt} of the $j$-th eigenfunction of $W_\tau$.     

\begin{figure}
    \includegraphics[width=.33\linewidth]{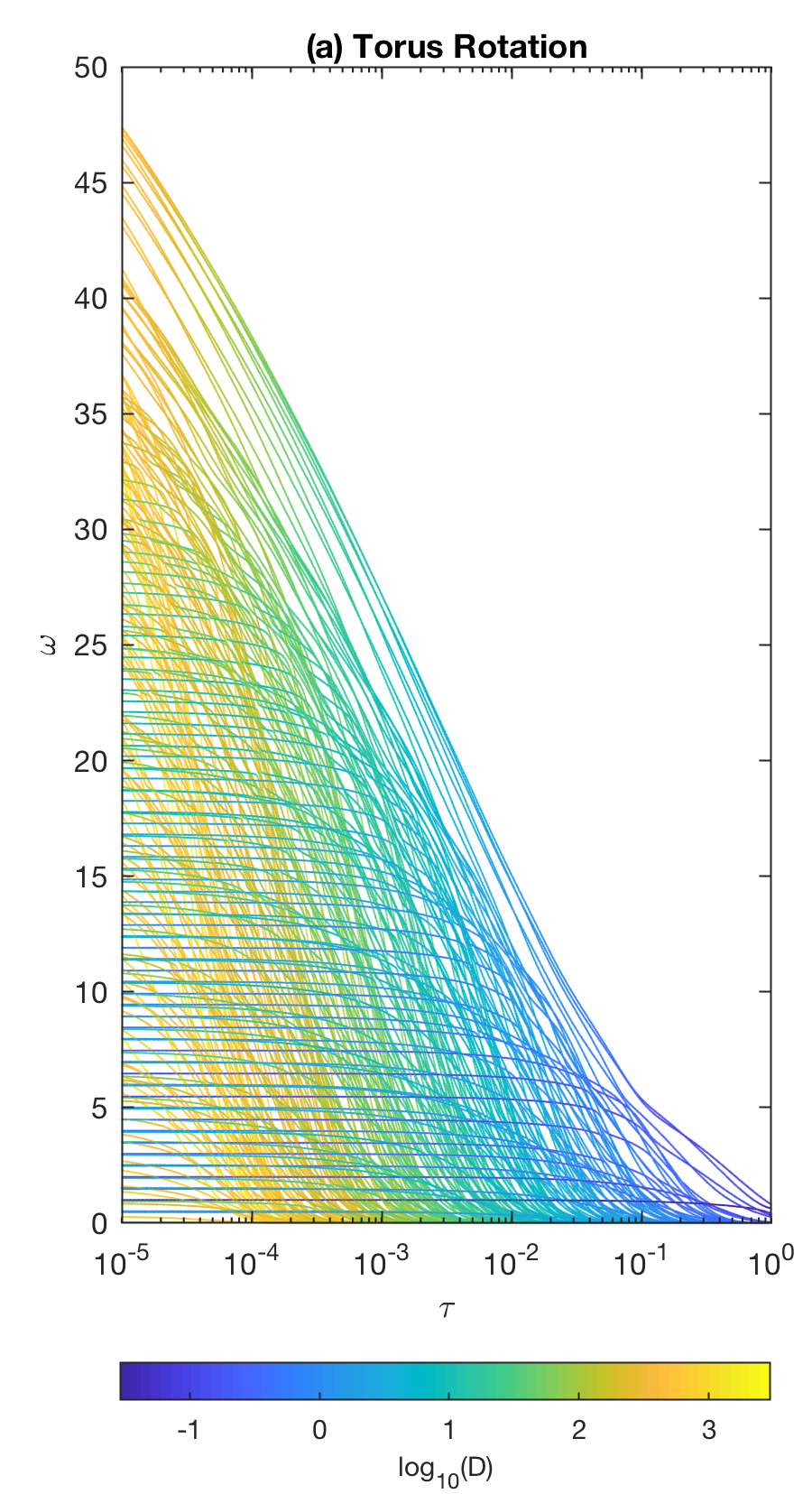}
    \includegraphics[width=.33\linewidth]{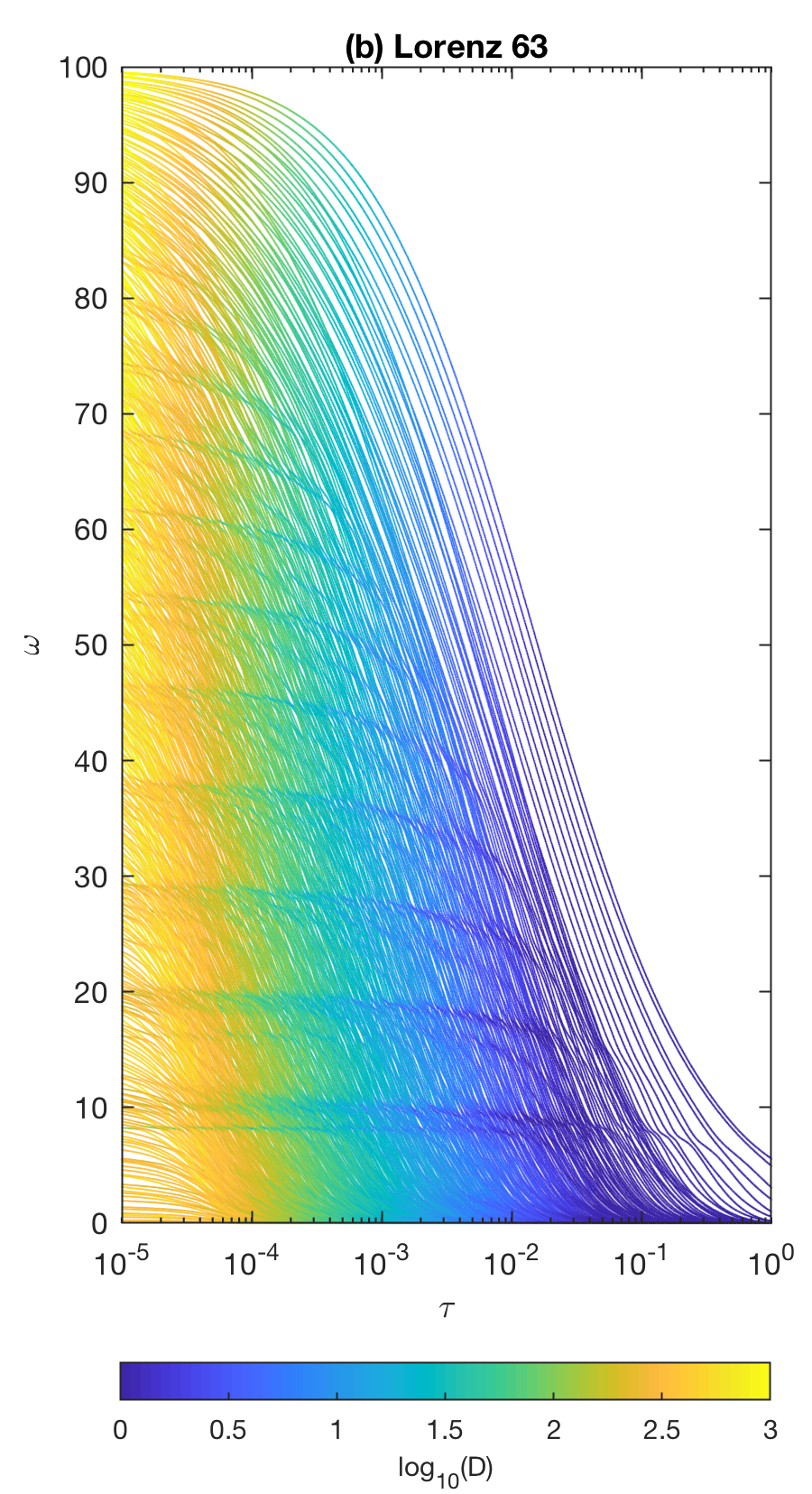}
    \includegraphics[width=.33\linewidth]{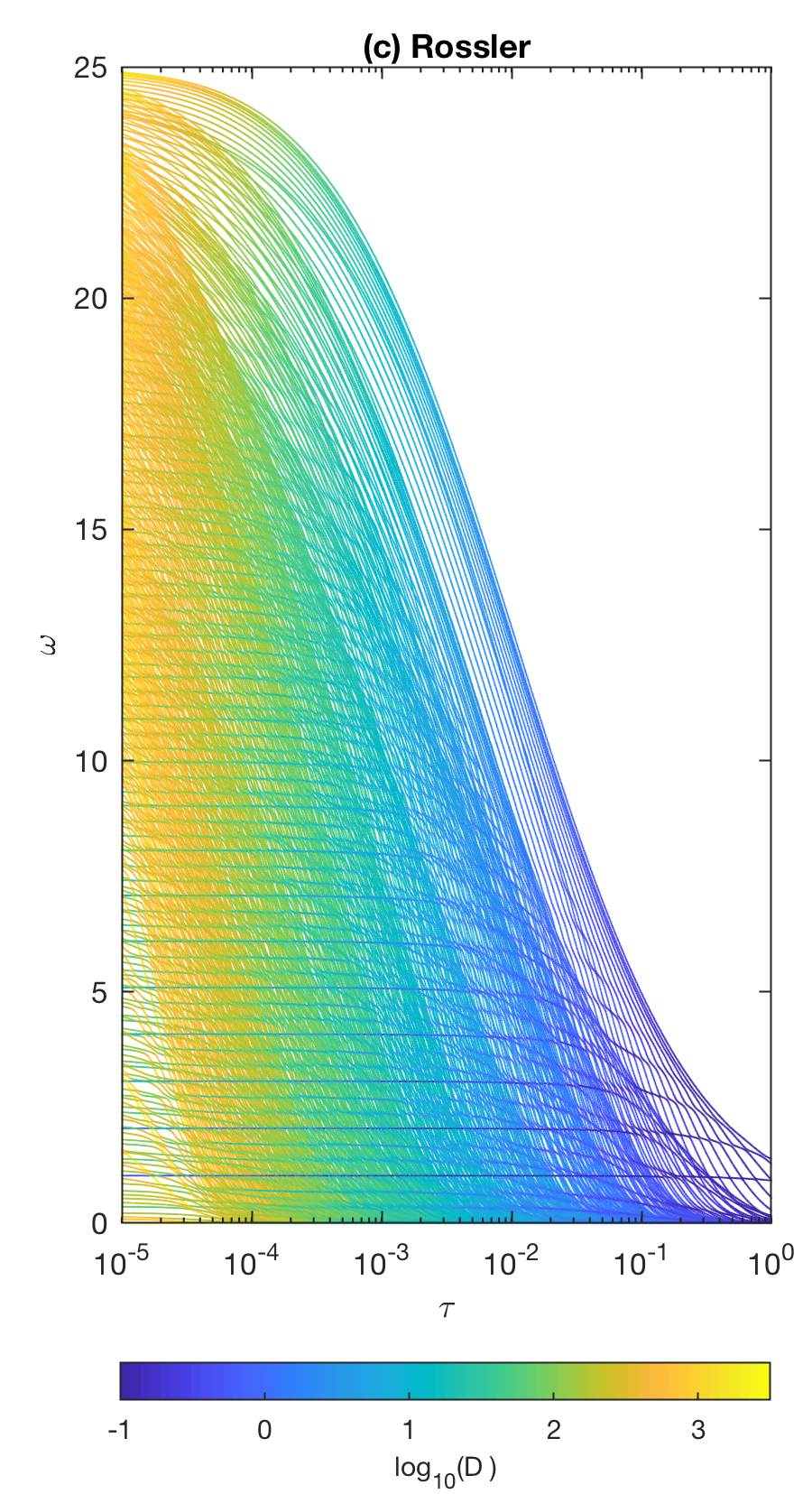}
    \caption{Eigenfrequencies $\omega_j $ of the data-driven generators $\KVKst_{\tau}$ as a function of $\tau$, for (a) the linear torus flow; (b) the L63 system; and (c) the R\"ossler system. Colors represent the logarithms of the frequency-adjusted Dirichlet energies from~\eqref{eqDirichletDt} of the corresponding eigenfunctions. Only positive frequencies are shown, as the $ \omega <0 $ parts of the spectra are mirror images of the $ \omega >0$ parts by skew-adjointness and reality of $W_\tau$.}
	\label{fig:SpecTau}
\end{figure}

\paragraph{Linear flow on the 2-torus} For any choice of rationally independent frequencies $ \alpha_1$ and $\alpha_2$, the system has a unique Borel ergodic invariant probability measure $ \mu $, which coincides with the Haar measure on $\mathbb{T}^2$. Thus, in the notation of Assumption~\ref{assump:A1}, the state space $ \mathcal{M}$, the forward-invariant compact manifold $M$, and the support of the invariant measure $X$ are all equal to $\mathbb{T}^2$. The basin of the invariant measure $B_\mu $ from Section~\ref{sect:numerics} is also equal to $\mathbb{T}^2$.  For this invariant measure, the Koopman group on $L^2(\mu)$ has pure point spectrum, consisting of eigenfrequencies of the form $j_1 \alpha_1 + j_2 \alpha_2$,  $j_1,j_2\in\mathbb{Z}$, corresponding to the eigenfunctions $e^{i(j_1\theta_1+j_2\theta_2)}$. The latter form an orthonormal basis of $L^2(\mu)$, so that the point and continuous spectrum subspaces in the invariant splitting in~\eqref{eqHDecomp} are $H_p= L^2(\mu)$ and  $H_c = \{0\}$, respectively. Note that because $\alpha_1$ and $ \alpha_2 $ are rationally independent, the set of eigenfrequencies lies dense in $\real$, which implies that the support of the PVM $E$ of this system (in this case, the closure of its set of eigenvalues) is equal to the whole real line. This makes the problem of numerically distinguishing eigenfrequencies from non-eigenfrequencies non-trivial, despite the simplicity of the underlying dynamics.

Figure~\ref{fig:SpecTau}(a) shows a scatterplot of the eigenfrequencies $\omega_{\tau,j}$ and the corresponding Dirichlet energies $ \Dirich_{\tau,j}$, computed for $L=500$ and values of $ \tau $ logarithmically spaced in the interval $ [ 10^{-5}, 1 ]$. There, the behavior of the numerically computed eigenfrequencies are broadly consistent with the results in Theorem~\ref{thm:Main} and Corollary~\ref{corr:APS}. In particular, the eigenfrequencies are seen to form continuous curves parameterized by $ \tau $ (consistent with Theorem~\ref{thm:Main}(vii) and Proposition~\ref{prop:Semigroup}(iv)), and the Dirichlet energy delineates the curves that have numerically converged over the examined values of $\tau$ (i.e., as  $ \tau $ approaches $10^{-5}$), from those that have not (as expected from Corollary~\ref{corr:APS}). Notice, in particular, that the task of visually identifying continuous eigenfrequency curves in Figure~\ref{fig:SpecTau}(a)  would be significantly more difficult without color-coding by Dirichlet energy.    

According to Corollary~\ref{corr:APS}, the eigenfrequency curves of $W_\tau $ with bounded Dirichlet energies should approximate Koopman eigenfrequencies, and the corresponding eigenfunctions should approximate Koopman eigenfunctions. Indeed, as illustrated in Figure~\ref{fig:ZetaTorus}, the leading data-driven eigenfrequencies $\omega_{\tau,j}$  for $ \tau = 10^{-5}$ agree with the theoretical eigenfrequencies to two to four significant figures. Moreover, the corresponding eigenfunctions agree well with the Koopman eigenfunctions of this system; that is, $\zeta_{\tau,j}$ in Figure~\ref{fig:ZetaTorus} have the structure of Fourier modes on the 2-torus, with near-exact sinusoidal time series at the corresponding eigenfrequencies. 

\begin{figure}
    \includegraphics[width=\linewidth]{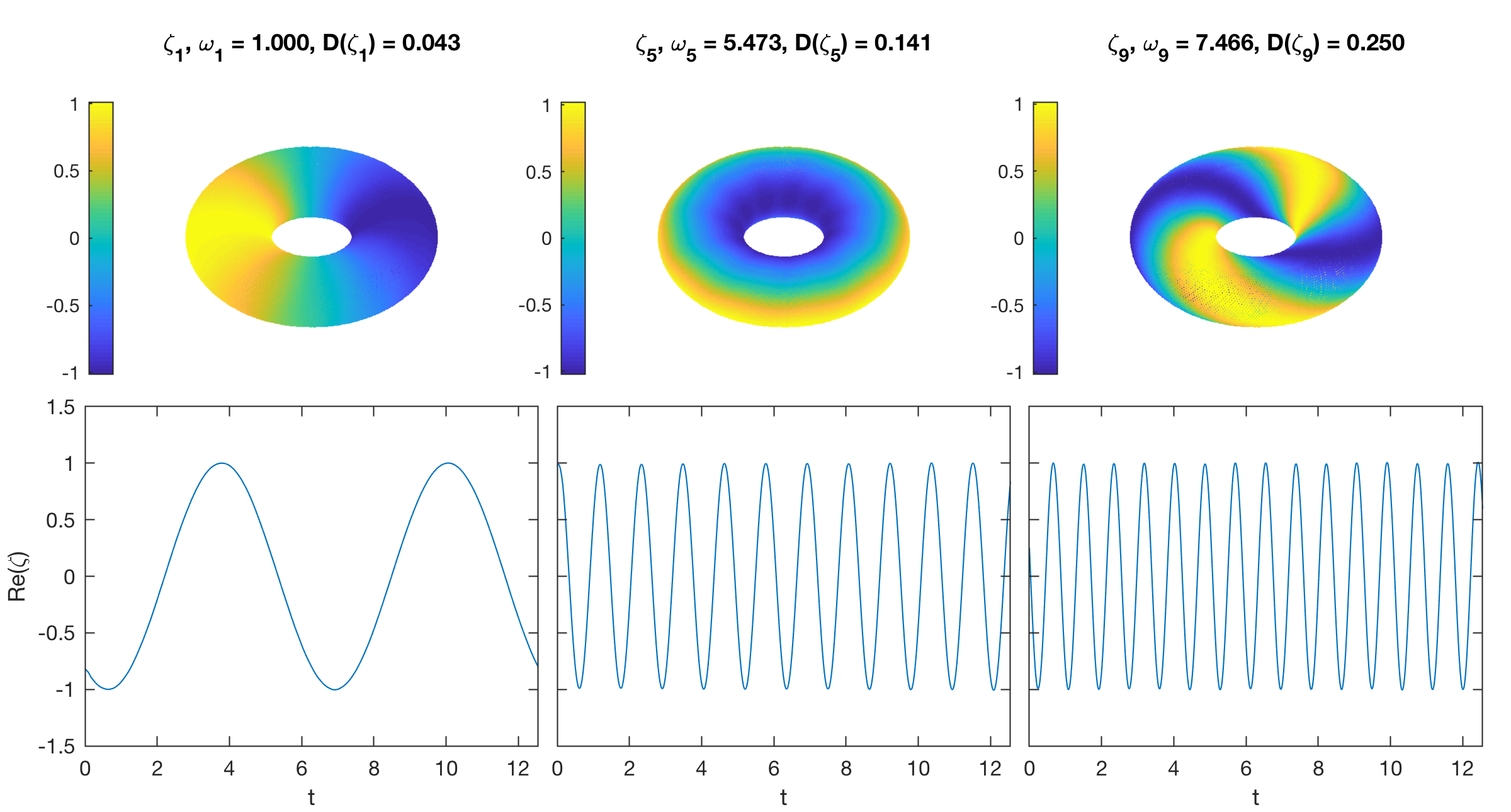}
    \caption{Representative eigenfunctions $\zeta_{\tau,j}$ of the data-driven generator $\KVKst_{\tau}$ with $ \tau = 10^{-5}$ for the linear flow on  the 2-torus. Top row: Scatterplots of $ \Real( \zeta_{\tau,j} ) $ on the training dataset embedded in $\real^3$. Bottom row: Time series $ t_n \mapsto \Real(\zeta_{\tau,j}(x_n))$ of the eigenfunctions, sampled along a portion of the dynamical trajectory in the training data. The numerical eigenfrequencies $ \omega_{\tau,j}$  and frequency-adjusted Dirichlet energies $\Dirich_{\tau,j}$ are also indicated. The eigenfrequencies with $ j = 1 $, 5, and 9 shown here agree with the theoretical eigenfrequencies $  \alpha_1 = 1 $, $ \alpha_2 \approx 5.477 $, and $ 2 \alpha_1 + 1 \alpha_2 \approx 7.477 $ to within four, three, and two significant figures, respectively.}
	\label{fig:ZetaTorus}
\end{figure}

\blue{Next, to assess the significance of RKHS regularization in spectral approximation of the generator, in Figure~\ref{figZetaTauTorus} we compare the real and imaginary parts of numerical eigenfunctions $\zeta_{\tau,j}$ of $W_\tau$ with eigenfunctions obtained from a ``naive'' approximation of the generator with $\tau = 0 $. In all cases, we select the eigenfunction whose corresponding eigenfrequency is closest to the generating eigenfrequency $ \alpha_1 = 1$, and plot the real and imaginary parts of $ \zeta_{\tau,j}$ as a scatterplot in the complex plane. For an exact approximation of a normalized Koopman eigenfunction, the plotted points should lie in the unit circle. According to the results in the figure, the naive approximation performs comparably to the regularized approximation for the experiment with $N = \text{64,000}$ and $ L=500 $ basis functions, but the quality of the approximation has considerably higher sensitivity to the number of samples and/or basis functions employed. Indeed, as is evident from Figures~\ref{figZetaTauTorus}(b) and~\ref{figZetaTauTorus}(e), decreasing the number of samples to $ N = 6400$ imparts a significant amount of high-frequency noise in the eigenfunction obtained from the naive approximation, whereas the eigenfunction based on RKHS regularization is comparatively more stable. This can be understood from the fact that the quality of the data-driven basis functions $\phi_{N,j}$ generally decreases with decreasing $N$, and the amount of quality degradation is higher the smaller the corresponding eigenvalue $\lambda_{N,j} $ is. As a result, without regularization to suppress the basis functions corresponding to small $ \lambda_{N,j}$, the quality of the naive approximation also degrades}. 

\blue{Increasing $L$ at fixed $N$ is also expected to adversely affect the naive approximation, in this case not only due to the basis function errors just mentioned, but also because the spectrum of $V$ is a dense subset of the imaginary line. That is, even with ``perfect'' basis functions, increasing $L$ without regularization will result in the creation of near-degenerate eigenfrequencies around any reference eigenfrequency in the spectrum of the approximate generator, leading to high sensitivity to perturbations. Figures~\ref{figZetaTauTorus}(c) and~\ref{figZetaTauTorus}(g) demonstrate that increasing $L$ from 500 to 1000 in the $N=6400 $ experiments results in considerable degradation of the eigenfunctions obtained from the naive approximation, whereas the eigenfunctions of the RKHS-regularized generator behave stably under this parameter change. In the case of the larger, $N= \text{64,000}$ dataset, a similar behavior is observed in Figures~\ref{figZetaTauTorus}(d) and~\ref{figZetaTauTorus}(g) by increasing $L$ to 10,000.}

\begin{figure}
    \includegraphics[width=\linewidth]{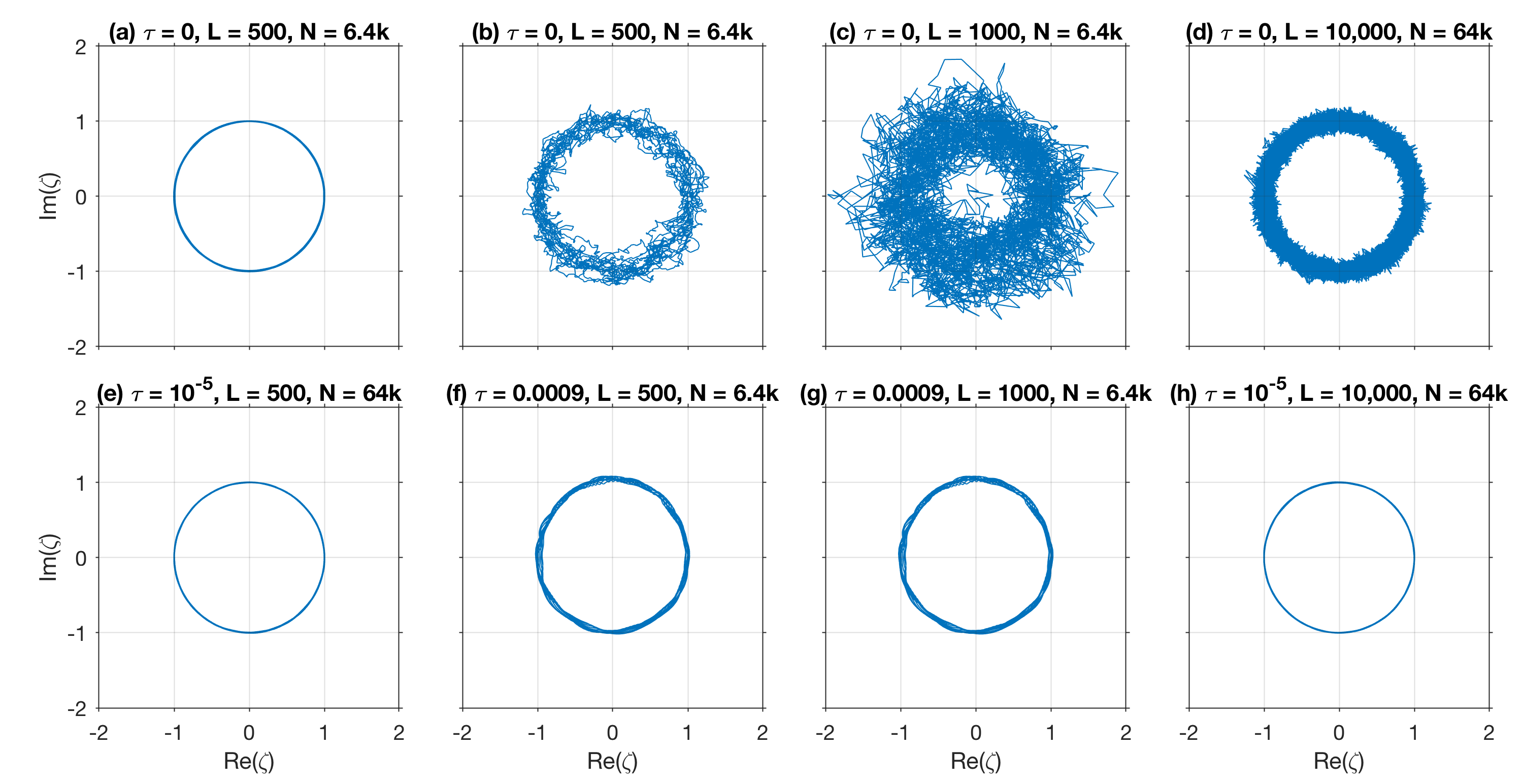}
    \caption{\blue{Real and imaginary parts of numerical Koopman eigenfunctions for the torus flow obtained from data-driven approximations of the generator without regularization (a--d) and the RKHS regularization $W_\tau $ (e--h), for different dataset sizes $N$ and values of the regularization and spectral resolution parameters $\tau $ and $L$. The eigenfunctions depicted here are those whose corresponding eigenfrequency in the data-driven spectrum is closest to the theoretical eigenfrequency $ \alpha_1 = 1$. For an exact approximation of a normalized Koopman eigenfunction, the numerical eigenfunctions should take values in the unit circle in the complex plane.}}
    \label{figZetaTauTorus}
\end{figure}

Turning now to forecasting, in Figure~\ref{fig:T2DPred} we show prediction results for the components $F_1$ and $F_3$ of the torus embedding into $\real^3$, as well as the observable $e^{F_1+F_3}$, which has a non-polynomial dependence on the components of the observation map (and in this case, the Koopman eigenfunctions). In all three cases, we use $ \tau = 10^{-5}$ and $L=500$, and examine lead times $t$ in the interval $ [ 0, 3000 \, \Delta t ] \approx [ 0, 39 ] $, which is approximately 26 times longer than the ``fast'' characteristic timescale $ 2  \pi/ \alpha_2 \approx 1.15$ of the system. Over that interval, the normalized forecast errors $\varepsilon(t)$ exhibit a linear error growth, remaining below 0.1 in the case of $F_1$, $F_3$, and below 0.2 in the case of $e^{F_1+F_3}$. The somewhat lower forecast skill for $e^{F_1+F_3}$ is consistent with the fact that infinitely many Koopman eigenfrequencies are required to fully capture the dynamical evolution of this observable. As mentioned in Section~\ref{sec:intro}, an advantageous aspect of the RKHS framework presented here over previous forecasting techniques operating on $L^2$ spaces \cite{BerryEtAl15,Giannakis17} is that it produces pointwise-evaluatable prediction functions, as opposed to expectation values with respect to probability measures with $L^2$ densities (which must be supplied by the user as initial conditions). As illustrated in Figure~\ref{fig:T2DPred}, the forecasts accurately reproduce the dynamical evolution of all three observables examined here.

\begin{figure}
    \includegraphics[width=\linewidth]{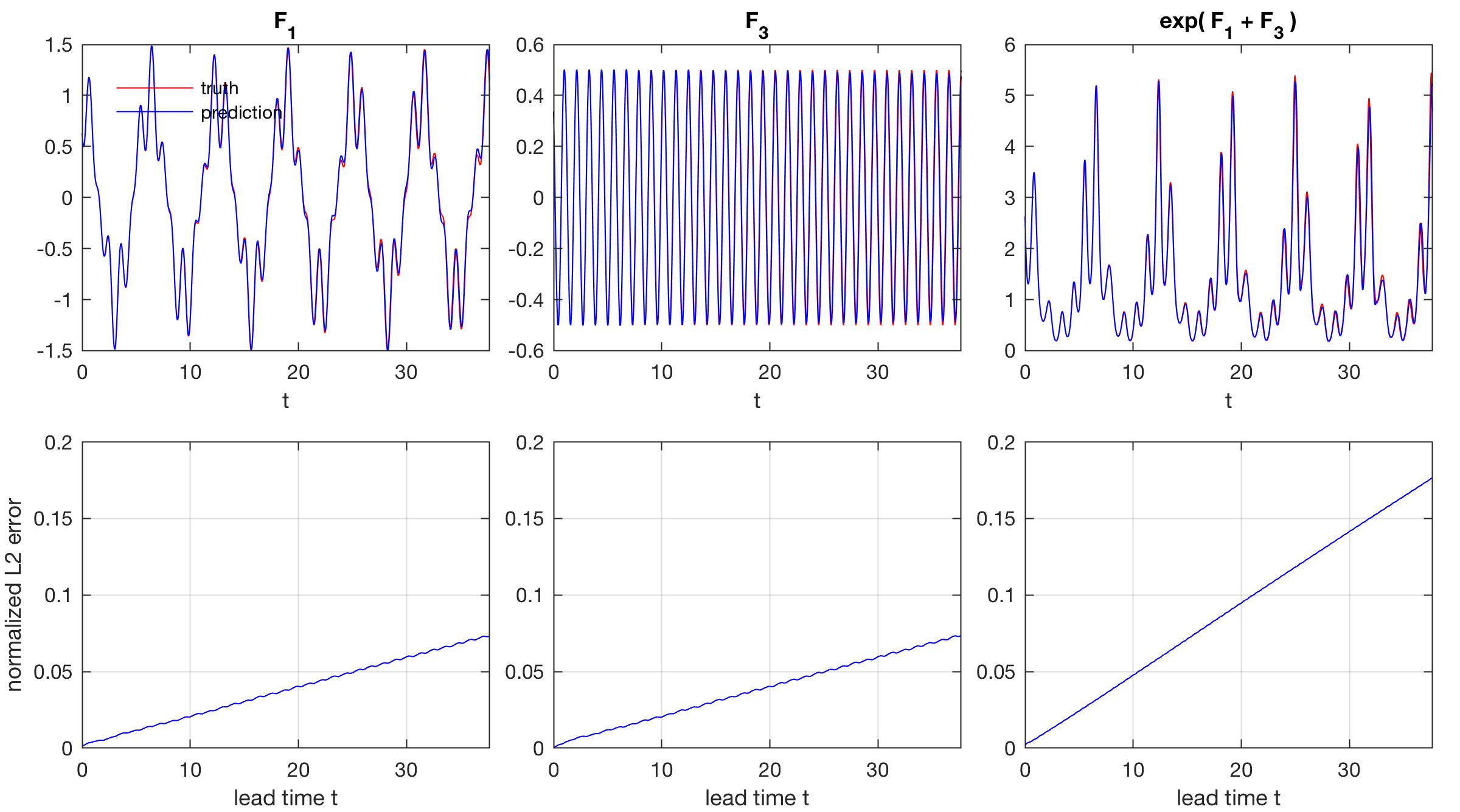}
    \caption{Data-driven prediction of the components $F_1$ and $F_3$ of the embedding $F$ of the 2-torus into $\real^3$ (left and center columns), and the non-polynomial observable $\exp(F_1+F_3)$ (right column) for the linear torus flow, using the operator $e^{t \KVKst_\tau}$ with $ \tau = 10^{-5}$. Top row: Comparison of the true and predicted signals as a function of lead time $t$ for a fixed initial condition in the verification dataset. Bottom row: Normalized RMSE $\varepsilon(t)$ as a function of lead time.}
	\label{fig:T2DPred}
\end{figure}

\paragraph{Lorenz 63 system} For our standard choice of parameters, the L63 system is known to have a compact attractor $ X \subset\mathcal{M} $ in the state space $ \mathcal{M} = \real^3$ \cite{Tucker99} with fractal dimension $\approx 2.06$ \cite{LorentzFract}, supporting a physical invariant measure $ \mu$, which has a single positive Lyapunov exponent $ \Lambda \approx 0.91$ \cite{Sprott03}. Due to dissipative dynamics, the attractor is contained within absorbing balls \cite{LawEtAl14}, playing here the role of the forward-invariant compact manifold $M \supset X$.  The system is also rigorously known to be mixing \cite{LuzzattoEtAl05}, which implies that its associated Koopman unitary group on $L^2(\mu)$ has no nonzero eigenfrequencies. Thus, the $H_p$ subspace for this system is the one-dimensional space consisting of constant functions, while $H_c $ contains all non-constant $ f \in L^2(\mu)$ with $ \int_M f \, d\mu = 0 $.

Figure~\ref{fig:SpecTau}(b) shows the dependence of the eigenfrequencies of $W_{\tau}$ for the L63 system, as well as the corresponding Dirichlet energies, on $ \tau \in [ 10^{-5}, 1 ] $, computed using $L=750$ basis functions. As one might expect, the behavior of this spectrum is qualitatively different from that of the quasiperiodic torus flow in Figure~\ref{fig:SpecTau}(a). That is, instead of the eigenfrequency curves of low Dirichlet energy interleaved with higher-Dirichlet-energy curves in Figure~\ref{fig:SpecTau}(a), the eigenfrequencies in Figure~\ref{fig:SpecTau}(b) exhibit an apparent continual growth in Dirichlet energy as $\tau $ decreases to 0. This behavior is consistent with Corollary~\ref{corr:APS}, according to which if the Dirichlet energy were to saturate along a sequence of eigenfrequencies of $W_\tau$ as $ \tau \to 0^+$, and that sequence had a nonzero limit, then that limit would necessarily be a nonzero Koopman eigenfrequency. As stated above, the latter is not possible for the L63 system. Nevertheless, upon visual inspection, one can identify in Figure~\ref{fig:SpecTau}(b) frequency bands characterized by smaller Dirichlet energy than the surrounding frequencies; e.g., frequency bands centered at $ \omega \simeq 8 $, 10, 20, 27, as well as higher frequencies. According to Corollary~\ref{corr:APS}, the corresponding eigenfunctions of $W_{\tau}$ are good candidates for coherent observables, evolving as approximate Koopman eigenfunctions, as we now verify.   

Representative eigenfunctions $ \zeta_{\tau,j}$ chosen from these frequency bands for $ \tau = 10^{-4}$, and visualized as scatterplots on the L63 attractor, as well as time series on the sampled dynamical trajectory, are displayed in Figure~\ref{fig:ZetaL63}. At least at the level of time series, the qualitative features of these eigenfunctions can be interpreted as generalizations of the Koopman eigenfunctions associated with the point spectra of measure-preserving ergodic dynamical systems. That is, similarly to Koopman eigenfunctions, the eigenfunctions of $W_{\tau}$ in Figure~\ref{fig:ZetaL63} are narrowband signals, evolving at a characteristic frequency determined from the corresponding eigenvalue $ \omega_{\tau,j} $, and with $\simeq 90^\circ$ phase difference between their real and imaginary (not shown) parts. However, unlike true Koopman eigenfunctions, the oscillatory signals associated with $ \zeta_{\tau,j}$ exhibit pronounced amplitude modulations, giving them the appearance of wavepackets. \blue{If single-frequency, constant-amplitude, sinusoidal time series  are to be thought of as hallmark features of Koopman eigenfunctions in measure-preserving systems, it appears that the eigenfunction time series of $W_{\tau}$ shown in Figure~\ref{fig:ZetaL63} lose the constancy of the amplitude, while maintaining a narrowband frequency character with high phase coherence between real and imaginary parts. In other words, these eigenfunctions reveal observables of the L63 system with an approximately cyclical behavior, despite mixing dynamics.} 

\begin{figure}
    \includegraphics[width=\linewidth]{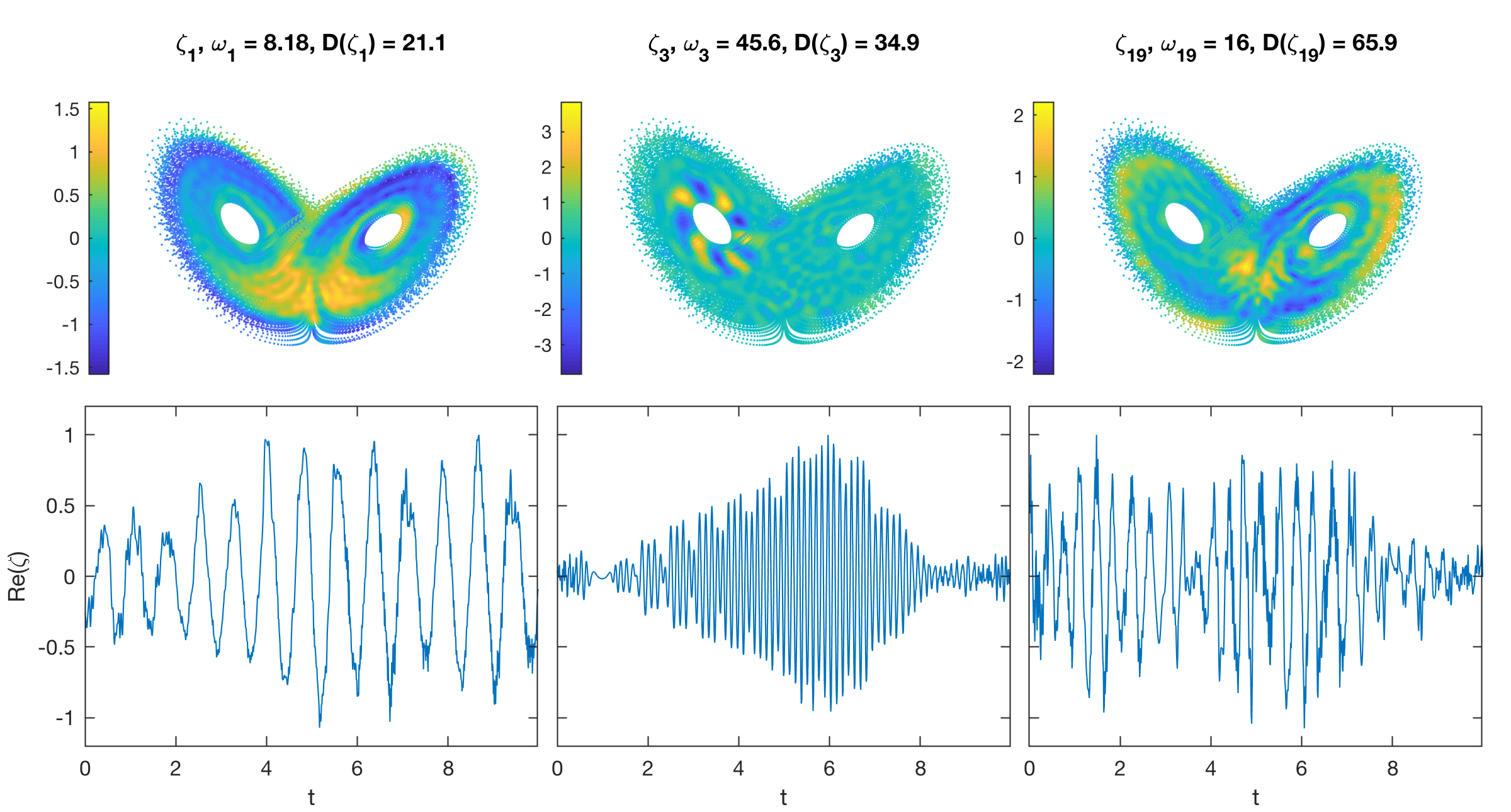}
    \caption{As in Figure~\ref{fig:ZetaTorus}, but for eigenfunctions of the data-driven generator $\KVKst_{\tau}$ with $ \tau = 10^{-4}$ for the L63 system. The eigenfunction time series in the lower panels have been scaled by their maximum absolute values so as to fit within the same axis limits. Observe the qualitatively different geometrical structure of the eigenfunctions on the Lorenz attractor. Despite these differences, the corresponding eigenfunction time series have the structure of amplitude-modulated wavetrains with a fairly distinct carrier frequency and lower-frequency modulating envelopes.}
	\label{fig:ZetaL63}
\end{figure}

Despite the qualitative similarities of the corresponding time series, it is evident from Figure~\ref{fig:ZetaL63} that the geometrical structure of the eigenfunctions of $W_{\tau}$ on the L63 attractor may exhibit significantly different characteristics. For example, eigenfunction $\zeta_{\tau,3}$ shown there (which corresponds to  fairly high eigenfrequency, $\omega_{\tau,3} \approx 46$) appear to be strongly localized on one of the two lobes of the L63 attractor, whereas eigenfunctions $\zeta_{\tau,1}$ and $\zeta_{\tau,19}$ (corresponding to lower eigenfrequencies, $ \omega_{\tau,1} \approx 8.2$ and $\omega_{\tau,19} \approx 16$, respectively) are supported on both lobes. Moreover, the level sets of $ \zeta_{\tau, 3}$ are arranged in predominantly transverse directions to the dynamical flow, whereas those of $ \zeta_{\tau,1}$ and $\zeta_{\tau,19}$ appear to be more parallel relative to the orbits of the dynamics. These differences are consistent with the fact that $ \omega_{\tau,3}$ is appreciably larger than $\omega_{\tau, 1}$, as a more transverse arrangement of level sets relative to the orbits of the dynamics means that more contour crossings per unit time take place. It should be noted that an analogous eigenfunction to $\zeta_{\tau,3}$, but supported in the opposite lobe of the L63 attractor is also present in the spectrum of $W_{\tau}$ (not shown here). It is also worthwhile noting that eigenfunction $\zeta_{\tau,1}$ bears some qualitative similarities with the pattern depicted in \citep[][Figure~13]{KordaEtAl18}. Based on its corresponding eigenfrequency and level-set structure, eigenfunction $\zeta_{\tau,19}$ resembles a second harmonic of $\zeta_{\tau,1}$.   

Next, we consider forecasting experiments for the three components $(F_1, F_2, F_3) $ of the observation map $F$, which coincide with the components of the L63 state vector in $\real^3$. We evaluate data-driven forecast functions for these observables at lead times in the interval $ [ 0, 500 \, \Delta t ] = [ 0, 5 ] $, using the regularization parameter $\tau = 10^{-5}$ and $L=750$ basis functions. Representative forecast trajectories and the corresponding normalized $L^2$ errors are displayed in Figure~\ref{fig:L63Pred}. Unlike the linear error growth seen in Figure~\ref{fig:T2DPred} for the torus experiments, the L63 forecasts exhibit an exponential-like initial error growth, lasting for lead times up to $ t \simeq 0.7$, and followed by a more gradual increase. The initial error growth period is somewhat shorter, though of the same order of magnitude, than the $e$-folding timescale associated with the system's positive Lyapunov exponent, i.e., $1/\Lambda \approx 1.1 $. In the case of observables $F_1$ and $F_2$, the normalized $L^2$ error $\varepsilon(t)$ is seen to saturate around 1.4 as $t$ approaches 5. Observable $F_3$ exhibits a somewhat slower error growth than $F_1 $ and $F_2$, which may be a manifestation of dynamical symmetry of the L63 system under the transformation $(x,y,z) \in \real^3 \mapsto ( -x, -y, z )$, making $F_3 $ a more predictable observable. 

To interpret the long-time behavior of the error $\varepsilon(t)$, note that the Koopman operator of a mixing dynamical system such as L63 has the property that, as $ t\to \infty$,  $ \langle g, U^t f \rangle_\mu$ converges to $ \langle g, 1 \rangle_\mu \langle 1, f \rangle_\mu$. Based on this, it is possible to verify that, in this limit, the normalized $L^2$ error $ \lVert (U^t - e^{t \tilde V_\tau} ) f \rVert_{L^2(\mu)} / \lVert f \rVert_{L^2(\mu)} $ associated with the quasiperiodic, unitary evolution group generated by $ \tilde V_\tau := \mathcal{U}^*_\tau W_\tau \mathcal{U}_\tau $ converges to $ \sqrt{2}$. Now, our RKHS-based prediction scheme does not employ $ \tilde V_\tau $ directly, and as follows from Corollary~\ref{thm:predic}, its error is governed by the non-unitary group generated by $B_\tau$, viz., $ \lVert U^t f - e^{t B_\tau} P^*_\tau f_\epsilon \rVert_{L^2(\mu)} / \lVert f \rVert_{L^2(\mu)} $. Nevertheless, for sufficiently small $ \tau $ and $\epsilon$, $e^{tB_\tau} P^*_\tau f_\epsilon $ can be made arbitrarily close to $e^{t\tilde V_\tau } f $, uniformly over compact time intervals. In that case, $ \varepsilon(t)$ would saturate close to $ \sqrt{2}$, as observed in Figure~\ref{fig:L63Pred}.  It is worthwhile noting that, in the presence of mixing, the forecast functions derived from expectation values must necessarily converge to a constant equal to the mean with respect to the invariant measure of the dynamics. For instance, as $ t \to \infty$, a forecast of the form $ \mathbb{E}_\rho U^t f = \langle \rho, U^t f \rangle_\mu $ \cite{BerryEtAl15,Giannakis17}, where $ \rho $ is a probability density in $L^2(\mu)$,  satisfies $ \mathbb{E}_\rho U^t f \to \langle \rho, 1 \rangle_{\mu} \langle 1, U^t f \rangle_{\mu}= \int_M f \, d\mu $. Such forecasts have asymptotic relative error $\epsilon(t)$ equal to 1, i.e., smaller relative error than the quasiperiodic unitary evolution models constructed here, but arguably a constant prediction does not provide a realistic representation of the underlying dynamics. Indeed, as illustrated by the forecast trajectories in Figure~\ref{fig:L63Pred}, the RKHS-based framework produces non-trivial, L63-like dynamics even at late times, when initial-value predictability has been lost. In that regard, the data-driven forecasts presented here are more akin to a ``simulation'' of L63 dynamics, as opposed to estimation of expectation values and/or other statistics.

\begin{figure}
    \includegraphics[width=\linewidth]{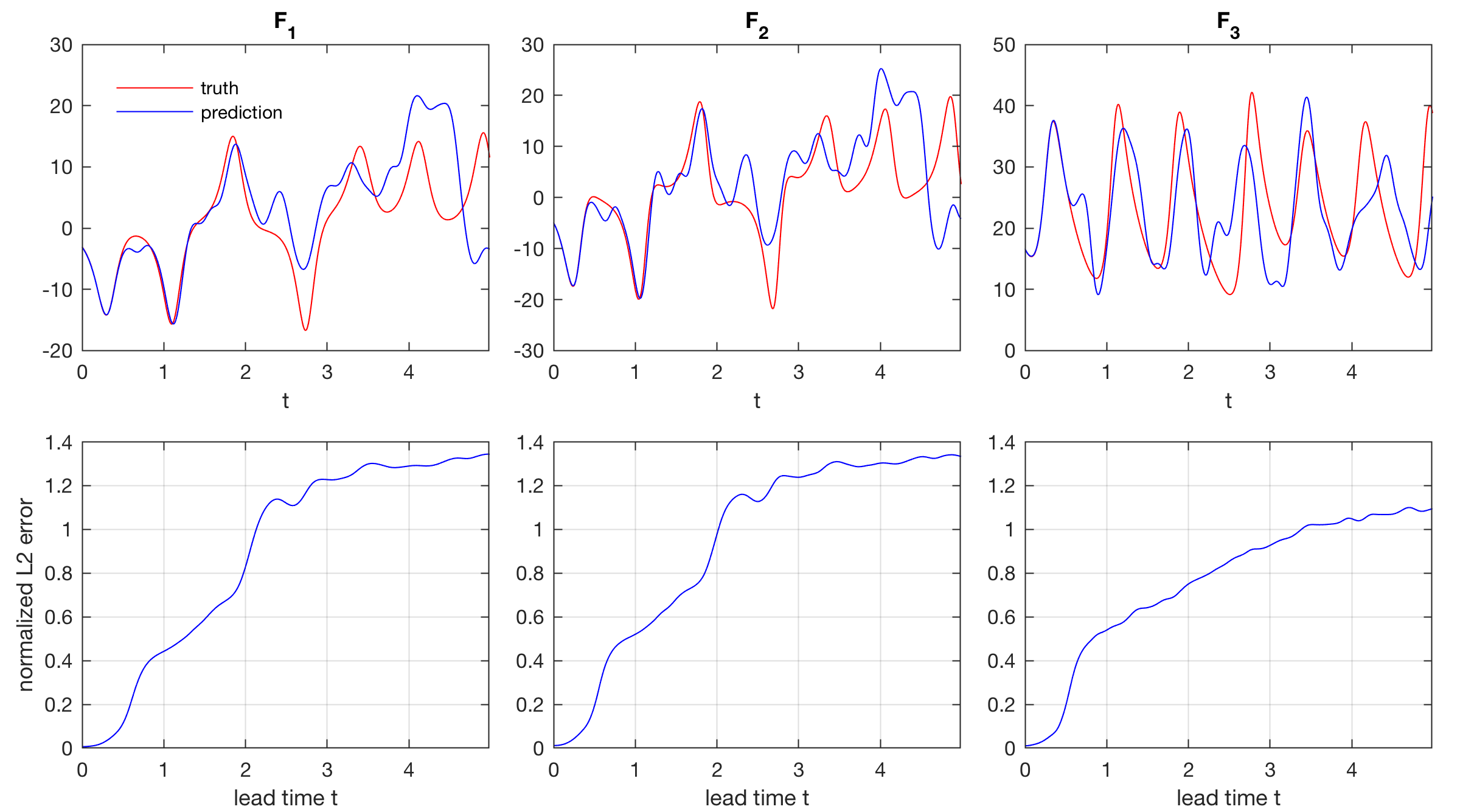}
    \caption{As in Figure~\ref{fig:T2DPred}, but for data-driven prediction of the components $F_1$, $F_2$, and $F_3$ of the L63 state vector.} 
	\label{fig:L63Pred}
\end{figure}

\paragraph{R\"ossler system} 
The R\"ossler system is sometimes viewed as a simplified analog of the L63 system, as it only has a single quadratic nonlinearity, as opposed to two nonlinearities in the L63 system. Yet, despite the simplicity of its governing equations, it exhibits complex dynamical characteristics, some of which are not seen in the L63 system. For the standard choice of parameters listed above, one well known such feature is an outward spiraling motion in the $z=0$ plane about an unstable fixed point at $(x,y,z)=(0,0,0) \in \real^3$, which undergoes intermittent bursts to large positive $z$ values when the radial coordinate $ r = \sqrt{x^2 + y^2} $ has become sufficiently large. This behavior produces a stiff signal in the $z$ coordinate, as well as banding of trajectories in state space, which are challenging to model with data-driven approaches. Another notable aspect of the R\"ossler system is that chaotic behavior predominantly takes place in the $(r,z)$ coordinates, whereas the evolution of the azimuthal angle in the $z=0 $ plane proceeds at a near-constant angular frequency, approximately equal to 1 in natural time units. The R\"ossler system is also known to possess a single positive Lyapunov exponent, approximately equal to 0.071 \cite{Sprott03}. While, to our knowledge, theorems on the existence and measure-theoretic mixing properties of the R\"ossler system analogous to \cite{LorentzFract,LuzzattoEtAl05} for the L63 system have not been established, the system has been studied extensively through analytical and numerical techniques, supporting the hypothesis that the R\"ossler system is indeed mixing, albeit at a slow rate \cite{PeiferEtAl05}.

In light of the above, it is perhaps not too surprising that the dependence of the eigenfrequencies of $W_\tau$ for this system, depicted in Figure~\ref{fig:SpecTau}(c) for $L = 750$, exhibits features reminiscent of both the torus and L63 spectra in Figures~\ref{fig:SpecTau}(a) and~\ref{fig:SpecTau}(b), respectively. That is, the spectrum of $W_\tau$ for the R\"ossler system exhibits bands of eigenfrequency curves with an apparent continual increase of Dirichlet energy with decreasing $ \tau$, as in L63,  but superposed to these curves is a set of eigenfrequencies at approximately integer multiples of a base frequency $ \alpha \simeq 1$, and  with near-constant corresponding Dirichlet energies, as in the linear torus flow. A visualization of corresponding eigenfunctions from the latter group, e.g., eigenfunctions $\zeta_{\tau,1}$ and $\zeta_{\tau,3}$ in Figure~\ref{fig:ZetaRossler} computed for $\tau = 10^{-5}$, reveals that these frequencies are indeed associated with highly coherent observables, which are predominantly functions of the azimuthal phase angle, and evolve near-periodically at integer multiples of the base frequency $ \alpha$.  Meanwhile, another group of eigenfrequencies of $W_\tau$, whose corresponding Dirichlet energies undergo a moderate increase  with decreasing $ \tau $, exhibit manifestly radial variability in state space and amplitude-modulated time series, reminiscent of the eigenfunctions recovered in the L63 system. Such an eigenfunction is $\zeta_{\tau,39}$ shown in Figure~\ref{fig:ZetaRossler}, whose corresponding eigenfrequency, $ \omega_{\tau,39} \approx 0.36 $, is smaller than the base frequency $\alpha$. \blue{On the basis of the eigenfrequency $\omega_{\tau,39}$, we can identify a characteristic timescale $2\pi/\omega_{\tau,39} \approx 17.5$ of coherent radial oscillations of the R\"ossler system.}  

\begin{figure}
    \includegraphics[width=\linewidth]{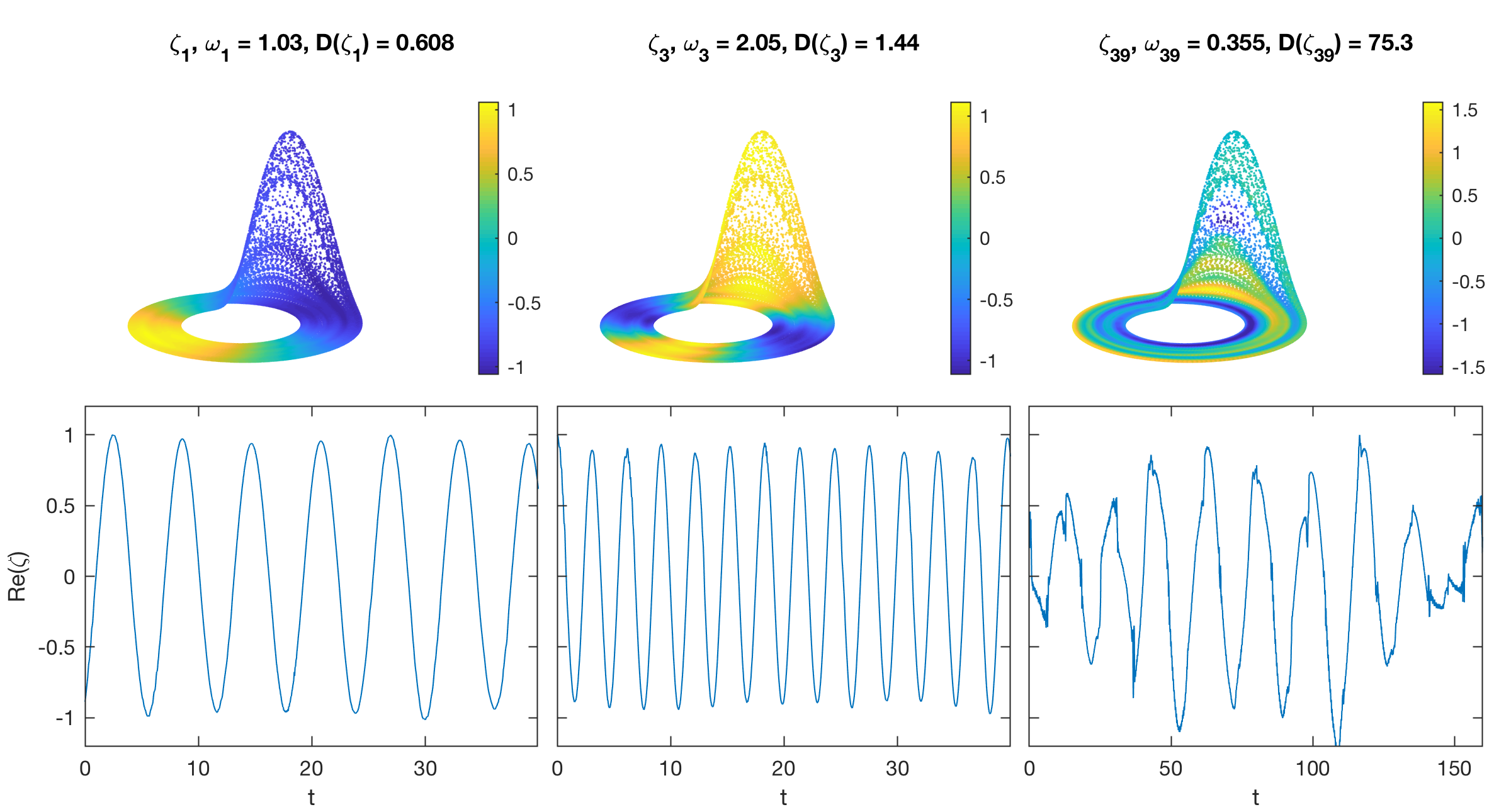}
    \caption{\label{fig:ZetaRossler}As in Figure~\ref{fig:ZetaTorus}, but for eigenfunctions of $W_\tau$, $ \tau = 10^{-5}$, for the R\"ossler system.}
\end{figure}

Next, Figure~\ref{fig:RosslerPred} shows forecasting results for the components $(F_1,F_2,F_3)$ of the R\"ossler state vector over lead times $ t \in [ 0, 2000 \, \Delta t ] = [ 0, 80 ] $, computed for $ \tau = 10^{-5}$ and $L=1000$. Due to the dynamical behavior of the R\"ossler system outlined above, one would expect that predicting $F_3 $ is significantly more challenging than predicting $F_1 $ or $F_2$, and this is indeed reflected in the results in Figure~\ref{fig:RosslerPred}. In particular, consistent with the near-linear evolution of the azimuthal phase angle in the $z=0$ plane, prediction of the observation map components $F_1$ and $F_2 $ remains skillful for the entire forecast interval examined, with the normalized error $\varepsilon(t)$ exhibiting a gradual increase to $ 0.25 $ by $ t \simeq 80 $. An inspection of the individual forecast trajectories shown in Figure~\ref{fig:RosslerPred} indicates that the errors in these forecasts are predominantly amplitude errors (as opposed to phase errors), likely caused by chaotic dynamics of the radial coordinate $r$. On the other hand, forecasts of the $F_3$ component exhibit a significantly more rapid error growth, reaching $ \varepsilon(t) \simeq 1.15$ as $t$ approaches 80.  This error can be understood from the highly stiff, intermittent nature of $F_3 $, exhibiting infrequent excursions to large positive values and virtually no negative values. As is evident from the forecast trajectory in Figure~\ref{fig:RosslerPred}, the data-driven forecasts are generally successful in capturing the timing of the $F_3$ bursts (likely aided by the high coherence of the azimuthal phase angle), but for lead times $ t \gtrsim 20 $, they struggle to reproduce the amplitude of the bursts and the non-negativity of the $F_3$ signal. In separate calculations, we have verified that the non-negativity of the forecast signal over a given time interval can be improved by increasing the number $L$ of basis functions. 

\begin{figure}
    \includegraphics[width=\linewidth]{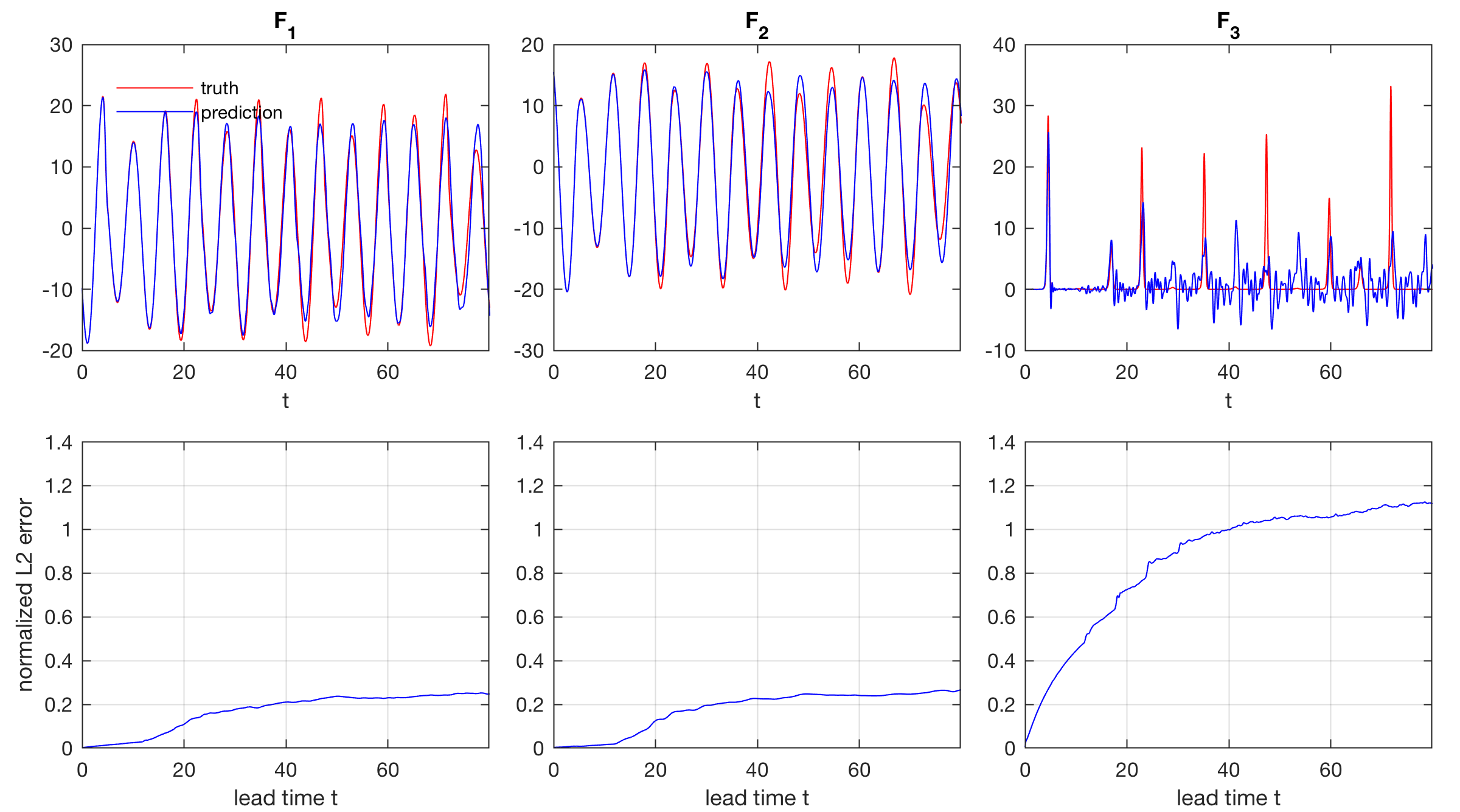}
\caption{As in Figure~\ref{fig:L63Pred} but for forecasts of the state-vector components of the R\"ossler system, using $ \tau = 10^{-5} $ and $L=7000$.} 	\label{fig:RosslerPred}
\end{figure}

%-_-_-_-_-_-_-_-_-_-_-_-_-_-_-_-_-_-_-_-_-_-_-_-_-_-_-_-_-_-_-_-_-_-_-_-_-_-_-_-_-_-_-_-_-_-_-_-_-_-_-_-_-_-_-_-_-_-_-_-_-_-_-_-_-_-_-_-_-_-_-_-_-_-_-_-
\section{\label{secConclusions}Concluding remarks} 

In this paper, we have developed a data-driven framework for spectral analysis of measure-preserving, ergodic dynamical systems, using ideas from RKHS theory. A central element of our approach has been to regularize the unbounded, skew-adjoint generator of the unitary Koopman group of the system by pre- and post-composing it with integral operators associated with reproducing kernels of RKHSs, rendering it into a compact operator. We showed that if this procedure is carried out using a one-parameter family of Markov kernels of appropriate ($C^1$) regularity, the resulting regularized generators form a one-parameter family of trace-class, skew-adjoint integral operators $W_\tau$ on RKHS, converging to the Koopman generator in strong resolvent sense in a limit of vanishing regularization parameter $\tau$. As a result, at every $ \tau > 0$, $W_\tau$ can be spectrally decomposed in terms of a purely atomic projection-valued measure (PVM), with an associated discrete set of eigenfrequencies and an orthonormal basis of eigenfunctions, converging to the PVM of the Koopman generator as $ \tau \to 0^+ $ in an appropriate sense. Notably, this result holds for measure-preserving ergodic systems of arbitrary spectral characteristics (pure point, continuous, mixed), and further allows consistent approximation of the functional calculus of the Koopman generator for bounded continuous functions. In particular, exponentiation of the regularized generator leads to a unitary, quasiperiodic evolution group, $ e^{tW_\tau} $, which can be used as an approximation of the Koopman group of the system to perform forecasting of observables with convergence guarantees. We also showed that the eigenfunctions associated with this group form coherent observables lying in the approximate point spectrum of the Koopman operator, generalizing the coherent patterns associated with Koopman eigenfunctions and the point spectrum of the system.  

Another advantageous aspect of the RKHS framework is that it naturally lends itself to data-driven approximation from time-ordered measurements of the system state taken through injective observation maps, requiring little structural modification of the continuous formulation. In particular, the data-driven approximation schemes employ properties of physical measures to consistently approximate integrals with respect to the invariant measure by time averages, and take advantage of RKHS regularity to approximate the action of the generator on functions by temporal finite differences. Coupled with the ability afforded by RKHSs to perform interpolation and out-of-sample evaluation, this approach leads to data-driven forecast functions for the evolution of observables, as well as coherent eigenfunctions, whose robustness can be assessed a posteriori through a Dirichlet energy criterion. 

We demonstrated the efficacy of this approach through a suite of  coherent pattern extraction and forecasting experiments in the setting of a quasiperiodic flow on the 2-torus and the chaotic L63 and R\"ossler systems. \blue{In the case of the torus rotation, the eigenfrequencies of the RKHS-regularized operator $W_\tau$ correctly identify generating eigenfrequencies of the system, as well as integer combinations of such eigenfrequencies. Meanwhile, in the L63 and R\"ossler settings, eigenfunctions of $W_\tau $ identified via the Dirichlet energy criterion exhibit an approximately periodic evolution, behaving as approximate Koopman eigenfunctions. These eigenfunctions reveal coherent oscillatory observables of these systems with characteristic timescales determined from the corresponding eigenvalues, despite potentially mixing dynamics. Forecasting using the evolution group generated by $W_\tau $ was found to perform well in these systems, with skill likely aided by the presence of approximately periodic eigenfunctions in the respective spectra.} 

Areas of future research stemming from this work include improved representations of the generator through alternative schemes to finite differences, as well as extensions to partially observed systems (i.e., non-injective observation maps). In addition, the fact that the spectral convergence results in Theorem~\ref{thm:Main} require pointwise convergence of the approximating operators only on a core of the generator $V$, yet in Proposition~\ref{prop:Semigroup} we were able to establish pointwise convergence on the full domain $D(V)$, suggests that it may be possible to weaken the $C^1$ regularity assumptions on the kernels and their associated RKHSs underlying Theorem~\ref{thm:Main}. It would also be fruitful to explore formulations of the framework presented here utilizing methods for kernel learning \cite{BerryHarlim18,OwhadiYoo18} to optimize prediction skill of prescribed observables. \blue{Meanwhile, the approximately periodic nature of the identified eigenfunctions in the L63 and R\"ossler systems suggests possible connections between the spectral properties of $W_\tau$ and periodic orbits of the underlying flow in state space expected for non-uniformly hyperbolic dynamics. Finally, a topic of significant interest in both the Koopman and transfer operator literature is spectral analysis of dissipative and/or non-ergodic systems \cite{Mezic05,FroylandEtAl14,Mezic20}. While some of the spectral approximation techniques employed in this work make use of the skew-adjoint structure of the generator of measure-preserving systems (e.g., strong convergence in a core in Lemma~\ref{lem:core_SRC}), and we have also made use of ergodicity to establish correspondences between the spectra of various types of regularized generators (e.g., Lemma~\ref{lem:fsdt}), it would nevertheless be fruitful to explore applications of RKHS theory to spectral analysis of such ``open dynamical systems'', extending the framework developed here for the measure-preserving, ergodic setting.}

\section*{Acknowledgments} Dimitrios Giannakis acknowledges support by ONR YIP grant N00014-16-1-2649, ONR MURI grant N00014-19-1-2421, NSF grants DMS-1521775, DMS-1854383, 1842538, and DARPA grant HR0011-16-C-0116. Suddhasattwa Das was supported as a postdoctoral research fellow from the first two grants. Joanna Slawinska acknowledges support from NSF grants 1551489 and 1842538. The authors would like to thank Igor Mezi\'c for pointing out a possible connection between the results of Theorem~\ref{thm:Main} and the pseudospectrum of the Koopman operator, which led to Corollary~\ref{corr:APS}. 

%-_-_-_-_-_-_-_-_-_-_-_-_-_-_-_-_-_-_-_-_-_-_-_-_-_-_-_-_-_-_-_-_-_-_-_-_-_-_-_-_-_-_-_-_-_-_-_-_-_-_-_-_-_-_-_-_-_-_-_-_-_-_-_-_-_-_-_-_-_-_-_-_-_-_-_-
\appendix 

%-_-_-_-_-_-_-_-_-_-_-_-_-_-_-_-_-_-_-_-_-_-_-_-_-_-_-_-_-_-_-_-_-_-_-_-_-_-_-_-_-_-_-_-_-_-_-_-_-_-_-_-_-_-_-_-_-_-_-_-_-_-_-_-_-_-_-_-_-_-_-_-_-_-_-_-_-_-_-_-_
\section{Variable-bandwidth kernels} \label{appVB}

The numerical experiments in Section~\ref{sect:examples} were performed using variable-bandwidth Gaussian kernels $ \kappa_N : Y \times Y \to \real$ of the form
    \begin{equation}\label{eqKVB}
        \kappa_N( y, y' ) = \exp \left( - \frac{ d^2( y, y' ) }{ \epsilon \sigma_N(y) \sigma_N(y') } \right).
    \end{equation}
    Here, $ \sigma_N : Y \to \real $ is a strictly-positive, $C^1$ function on $ Y$, which generally depends on the training dataset $ \{ y_0, \ldots, y_{N-1} \} $. We indicate this dependence with $N$ subscripts. Intuitively, the role of the bandwidth function $ \sigma_N $ is to correct for variations in the ``sampling density'' of the data. In particular, for a well conditioned kernel integral operator $G_N$, the number of datapoints lying within radius $O(\epsilon^{1/2}) $ balls centered at each datapoint should not exhibit significant variations across the dataset, yet, the standard radial Gaussian kernel from~\eqref{eqn:def:Gauss_ker} has no mechanism for preventing this from happening. For appropriately chosen $ \sigma_N $, the variable-bandwidth kernel in~\eqref{eqKVB} can, in effect, vary the radii of these balls to help improve conditioning. The different bandwidth functions proposed in the literature include near-neighbor distances \cite{ZelnikManorPerona04} and kernel density estimates \cite{BerryHarlim16}. In the numerical experiments of Section~\ref{sect:examples}, we will employ the latter approach, defining 
\begin{equation}\label{eqSigma}
      \sigma_N( y ) = \rho_N^{-1/\tilde m}( y ), \quad \rho_N(y) = \frac{ 1 }{ (\pi \tilde \epsilon )^{\tilde m / 2} } \int_Y e^{-d^2(y, y') / \tilde \epsilon } \, d\tilde \mu_N(y).
\end{equation} 
Here, $ \tilde \mu_N = \sum_{n=0}^{N-1} \delta_{y_n} / N $ is the sampling measure in data space, $ \tilde \epsilon $ a positive bandwidth parameter (different from $\epsilon $ in~\eqref{eqKVB}), and $ \tilde m $ a positive parameter approximating the dimension of $ F( X) $. The parameters $\epsilon $, $ \tilde \epsilon $, and $\tilde m $ are all determined from the data automatically; see \cite{BerryEtAl15,Giannakis17} for descriptions of this procedure. 

If $F(X) $ has the structure of a Riemannian submanifold of $Y$, and the pushforward $ \tilde \mu $ on of the invariant measure on $ Y $ has a smooth density, the functions $ \rho_N$ from~\eqref{eqSigma} are estimates of the sampling density  $ \rho = d\tilde \mu /d\vol$, which converge in the limit of $ N \to \infty $ followed by $ \tilde \epsilon \to 0 $. Thus, with this choice of bandwidth function, the bandwidth of the kernel $ \kappa_N $ from~\eqref{eqKVB} will be large (small) when the sampling density is small (large), achieving the desired balancing of the kernel. More quantitatively, with this choice of bandwidth functions and after suitable normalization, $ \kappa_N $ approximates the heat kernel of a conformally transformed Riemannian metric on $F(X)$, whose volume form has uniform density relative to $ \tilde \mu $ \cite{Giannakis17}. Of course, if $ F(X) $ does not have manifold structure, or $ \rho $ is not smooth, this Riemannian geometric interpretation is not applicable, but the balancing effect of the bandwidth functions on local balls still holds. It should be noted that one can prove spectral convergence results analogous to Lemma~\ref{lem:DataPredic}(i) for the class of $N$-dependent kernels on $M$ induced by $\kappa_N$; see \cite{GiannakisEtAl17} for such a result. Here, we will omit a proof of spectral convergence for the integral operators associated with $ \kappa_N$  in the interest of brevity. It is also important to note that, to our knowledge, it has not been established whether the kernels on $M$ induced by $\kappa_N $, and the kernel that they converge to as $ N \to \infty$, are $ L^2(\mu_N) $- and $ L^2(\mu) $-strictly-positive, respectively. That being said, we did not find evidence of zero eigenvalues of $G_N$ in the experiments of Section~\ref{sect:examples}. 

%-_-_-_-_-_-_-_-_-_-_-_-_-_-_-_-_-_-_-_-_-_-_-_-_-_-_-_-_-_-_-_-_-_-_-_-_-_-_-_-_-_-_-_-_-_-_-_-_-_-_-_-_-_-_-_-_-_-_-_-_-_-_-_-_-_-_-_-_-_-_-_-_-_-_-_-_-_-_-_-_
\section{Pseudocode }\label{sec:algo} 

In this appendix, we provide pseudocode listings for the techniques described in Section~\ref{sect:numerics}. We have split the entire process into four algorithms, the first two of which describe the construction of the data-driven eigenpairs $(\lambda_{N,j}, \phi_{N,j}) $ from Lemma~\ref{lem:DataPredic} and pointwise evaluation of the corresponding basis functions $ \psi_{N,j} $ of $\mathcal{H}_{N}$, respectively. Algorithm~\ref{alg:generator} describes the construction of the data-driven generator $W_{\tau,N,\Delta t}^{(L)}$ from Theorem~\ref{thm:data_predic}(i) and computation of its associated eigenvalues and eigenfunctions. Algorithm~\ref{alg:pred} describes the construction and pointwise (out-of-sample) evaluation of the data-driven forecast function $ f_{\tau,N,\Delta t, L, L'}^{(t)}$ from Corollary~\ref{corPredic}. In what follows, $ \vec 1 $ will denote the $N$-dimensional column vector whose elements are all equal to $1$. Moreover, the indexing of all vector and matrix elements will start from 0. 

We begin by listing Algorithm~\ref{alg:data_basis} for a general kernel $\kappa$ on data space $Y$ of the form in~\eqref{eqKPullback}, evaluated on a time series of the values of the observation map $F$ on a dynamical trajectory $x_0,\ldots,x_{N-1}$ in $M$. As stated in Sections~\ref{sect:numerics} and~\ref{sect:examples}, here we work with the variable-bandwidth Gaussian kernel described in \ref{appVB}. Evaluation of this kernel requires a kernel density estimation step, summarized in \cite[][Algorithm~1]{Giannakis17}. The variable-bandwidth Gaussian kernel also requires specification of the bandwidth parameter $\epsilon$, as well as the bandwidth and dimension parameters $ \tilde \epsilon$ and  $ \tilde m $, respectively, in~\eqref{eqSigma}.  We set these parameters automatically via the procedure described in \citep[][Appendix~A]{BerryEtAl15} and \citep[][Algorithm~1]{Giannakis17}. The main outputs of Algorithm~\ref{alg:data_basis} are the eigenpairs $( \lambda_{N,j}, \phi_{N,j})$ of the Markov operator $ G_N $ associated with the Markov kernel $p_{N}$, obtained via the bistochastic normalization procedure from Section~\ref{secReview_RKHS}. Due to the $L^2(\mu_N) \simeq \cmplx^N$ isomorphism, $ G_N $ can be represented by an $N\times N $ matrix $ \bm G $ with elements $ G_{ij} = p_N( x_{i}, x_{j} ) / N $, and the eigenvectors $ \phi_{N,j} $ by $N$-dimensional column vectors $ \vec \phi_j = ( \phi_{N,j}(x_0), \ldots, \phi_{N,j}(x_{N-1}) )^\top$.  We will abbreviate $ \lambda_{N,j} $ by $ \lambda_j$.  The eigenpairs $ ( \lambda_{j}, \vec \phi_j ) $  can be computed without explicit formation of $\bm G$, owing to the fact that $\bm G = \tilde{ \bm K } \tilde{ \bm K}^\top $, where $ \tilde{\bm K}$ is a non-symmetric $N \times N$ kernel matrix to be defined in Algorithm~\ref{alg:data_basis}. In particular, the  $\lambda_{j}$ are equal to the squared singular values of $\tilde{ \bm K }$, and the $ \vec \phi_j $ are equal to the corresponding left singular vectors. Algorithm~\ref{alg:data_basis} also outputs as auxiliary outputs the corresponding right singular vectors $ \vec \gamma_j \in \mathbb{R}^N$ of $ \tilde{ \bm K } $ and a degree vector $ \vec q \in \mathbb{R}^N$ associated with that matrix; these outputs will be used for pointwise evaluation in Algorithm~\ref{alg:OUS}.

\begin{algorithm}[Data-driven basis]\label{alg:data_basis} \

\begin{itemize}
\item Inputs 
\begin{itemize}
    \item Time series $ F(x_0), \ldots, F(x_{N-1}) $ in data space $Y$
	\item Number $L \leq N $ of eigenpairs to be computed
\end{itemize}
\item Outputs 
    \begin{itemize}
        \item Leading $L$ eigenvalues $ \lambda_0,  \ldots, \lambda_{L-1} $ of $\bm G$ and the corresponding eigenvectors $ \vec \phi_0, \ldots, \vec \phi_{L-1} \in \mathbb{R}^N$   
        \item Degree vector $ \vec q \in \mathbb{ R}^N$
        \item Right singular vectors $ \vec \gamma_0, \ldots, \vec \gamma_{L-1} \in \mathbb{R}^N$
    \end{itemize}
\item Steps
\begin{enumerate}
    \item Compute the $N \times N$ kernel matrix $\bm K$ with $K_{ij} = \kappa\left( F(x_i) , F(x_j) \right) / N$. 
    \item Compute the $N$-dimensional degree vectors $\vec{d} = \bm K \vec 1$ and $\vec{q} = \bm K \bm D^{-1} \vec 1$, where $ \bm D = \diag \vec d $.
    \item Form the $N \times N$ kernel matrix $ \tilde{ \bm K } = \bm D^{-1} \bm K \bm Q^{-1/2}$, with $ \bm Q = \diag \vec q $. 
    \item Compute the $L$ largest singular values $ \sigma_0, \ldots, \sigma_{L-1} $ of $ \tilde{ \bm K } $, and set $ \lambda_j = \sigma_j^2$. Set $ \vec \phi_j $ and $ \vec \gamma_j $ to the corresponding left and right singular vectors, respectively, normalized to unit 2-norm. 
\end{enumerate}
\end{itemize}
\end{algorithm}

Next, Algorithm~\ref{alg:OUS} carries out the task of evaluating the RKHS functions $ \psi_{N,j} \in \mathcal{H}_N$ at an arbitrary collection $ \hat x_0, \hat x_1, \ldots, \hat x_{\hat N-1} $ of points in $M$, given the corresponding values $ F(\hat x_0 ),F(\hat x_1), \ldots, F(\hat x_{\hat N-1}) $ of the observation map $F$. As with Algorithm~\ref{alg:data_basis}, this computation can be performed without explicit formation of a kernel matrix associated with $ p_N$, using instead the singular vectors $ \vec \phi_0, \ldots, \vec \phi_{L-1} $ and $ \vec \gamma_0, \ldots, \vec \gamma_{L-1} $. In what follows, we use the column vectors $ \vec \psi_j = ( \psi_{N,j}(\hat x_0), \ldots \psi_{N,j}( \hat x_{N-1}))^\top \in \mathbb{R}^{\hat N}$ to represent the values of the $ \psi_{N,j} $ at the desired points. Note that in the case of the variable-bandwidth Gaussian kernels from \ref{appVB}, the computation of the $ \vec \psi_j $ requires an additional density estimation step for the out-of-sample data $ F(\hat x_n)$, which is carried out analogously to \citep[][Algorithm~1]{Giannakis17}. Moreover, all kernel parameters $ \epsilon$, $ \tilde \epsilon$, and $ \tilde m $ are the same as those used in Algorithm~\ref{alg:data_basis}.            

\begin{algorithm}[Pointwise evaluation in RKHS]\label{alg:OUS} \
\begin{itemize}
\item Input
\begin{itemize}
    \item Values $ F(\hat x_0), \ldots, F(\hat x_{\hat N-1}) $ of the observation map at the evaluation points 
    \item Eigenvalues $  \lambda_0, \ldots, \lambda_{L-1}  $, eigenvectors $  \vec \phi_0, \ldots, \vec \phi_{L-1}  $, right singular vectors $  \vec \gamma_0, \ldots, \vec \gamma_{L-1}  $, and degree vector $\vec q $ from Algorithm~\ref{alg:data_basis}   
\end{itemize}

\item Output 
    \begin{itemize}
        \item Vectors $ \vec \psi_0, \ldots, \vec \psi_{L-1} \in \mathbb{R}^{\hat N} $ with the values of the RKHS functions $\psi_{N,0}, \ldots, \psi_{N,j}$ at the evaluation points 
    \end{itemize}
   
\item Steps
\begin{enumerate}
    \item Compute the $\hat N \times N$ kernel matrix $ \hat{ \bm K}$ with $ \hat K_{ij} = \kappa\left( F( \hat x_i) , F(x_j) \right) / N$. 
    \item Compute the $\hat N$-dimensional degree vector $\hat{d} = \hat{\bm K} \vec 1$. 
    \item Form the $\hat N \times N $ kernel matrix $ \bar{ \bm K } = \hat{\bm D}^{-1} \hat{\bm K} \bm Q^{-1/2} $,  where $ \hat{\bm D} = \diag \hat d $ and $ \bm Q = \diag \vec q $.
    \item Output $ \vec \psi_j = \bar{ \bm K} \vec \gamma_j $. 
\end{enumerate}
\end{itemize}
\end{algorithm}

Note that when working with Gaussian kernels, as done throughout this paper, we approximate the kernel matrices $ \bm K $, $ \tilde{ \bm K } $, $ \hat {\bm K} $ and $ \bar {\bm K} $ in Algorithms~\ref{alg:data_basis} and~\ref{alg:OUS} by sparse matrices (as is common practice), retaining in each case the $ k_{\text{nn}} $ largest entries  per row. In the numerical experiments of Section~\ref{sect:examples}, $k_\text{nn} $ was approximately $8\%$ of $N$.

We now describe how to construct an $L\times L$ matrix $ \bm W $ representing the data-driven generator $W_{\tau,N,\Delta t}^{(L)}$ in the $\psi_{\tau,N,j}$ basis of $ \mathcal{H}_{\tau,N}$, and use that matrix to compute the $(\omega_{\tau,N,\Delta t,j}^{(L)}, \zeta_{\tau,N,\Delta t,j}^{(L)})$ eigenpairs. We represent each eigenvector $ \zeta_{\tau,N,\Delta t,j}^{(L)} \in \mathcal{H}_{\tau,N}$ by a column vector $ \vec \xi_j = ( \xi_{0,j}, \ldots, \xi_{L-1,j} )^\top \in \cmplx^L$ storing the expansion coefficients of $ \zeta_{\tau,N,\Delta t,j}^{(L)} $ in the $ \psi_{\tau,N,j}$ basis, i.e., $ \zeta_{\tau,N,\Delta t,j}^{(L)} = \sum_{i=0}^{L-1} \xi_{i,j} \psi_{\tau,N,j}$. Given a set $ \{  \hat x_0, \ldots, \hat x_{\hat N-1} \} $ of evaluation points in $M$, the values $ \zeta_{\tau,N,\Delta t,j}^{(L)}(\hat x_n) $ will be represented by the column vectors $ \vec \zeta_j = ( \zeta_{\tau,N,\Delta t, j}^{(L)}(x_0), \ldots, \zeta_{\tau,N,\Delta t, j}^{(L)}( \hat x_{\hat N -1}))^\top \in \cmplx^{\hat N} $. In Algorithm~\ref{alg:generator} below, we describe the construction of $ \bm W$ and the computation of the $ \omega_j $, $ \vec \xi_j $, and $ \vec \zeta_j$,  using the central finite-difference scheme from~\eqref{eqFD2} to approximate the action of the generator. The algorithm can also be implemented using any skew-adjoint finite-difference scheme of appropriate regularity. Moreover, we employ the basis functions and pointwise evaluation procedures from Algorithms~\ref{alg:data_basis} and~\ref{alg:OUS}, associated with the bistochastic kernel normalization in Section~\ref{secReview_RKHS}, but Algorithm~\ref{alg:generator} can be implemented using any other Markov operator meeting the conditions of Theorem~\ref{thm:data_predic}. Algorithm~\ref{alg:generator} also returns the frequency-adjusted Dirichlet energies $ \mathcal{D}_{N,\Delta t}(\zeta_{\tau,N,\Delta, t,j}^{(L)})$ of the eigenfunctions from~\eqref{eqDirichletDt}, abbreviated $\mathcal{D}_j$. We also abbreviate $ \omega_{\tau,N,\Delta t, j}^{(L)} $ by $ \omega_j $.    

\begin{algorithm}[Data-driven generator and its eigendecomposition]\label{alg:generator} \ 
\begin{itemize}
\item Inputs
\begin{itemize}
	\item RKHS regularization parameter $\tau>0$
	\item Time step $\Delta t > 0$
    \item Eigenvalues $  \lambda_0, \ldots, \lambda_{L-1}  $, eigenvectors $  \vec \phi_0, \ldots, \vec \phi_{L-1}  $, right singular vectors $  \vec \gamma_0, \ldots, \vec \gamma_{L-1}  $, and degree vector $\vec q $ from Algorithm~\ref{alg:data_basis}
    \item  Pointwise-evaluated RKHS functions $\vec \psi_0, \ldots, \vec \psi_{L-1} $ from Algorithm~\ref{alg:OUS}
\end{itemize}
\item Outputs
\begin{itemize}
    \item Eigenfrequencies $ \omega_{0}, \ldots, \omega_{L-1} \in \mathbb{R}$,  the corresponding eigenvectors $\vec \xi_0, \ldots, \vec \xi_{L-1} \in \cmplx^L$, and the Dirichlet energies $ \mathcal{D}_0, \ldots, \mathcal{D}_{L-1} \geq 0$ 
    \item Vectors $ \vec \zeta_0, \ldots, \vec \zeta_{L-1} \in \cmplx^{\hat N} $ with the values of the eigenfunctions $ \zeta_{\tau,N,\Delta t,j}^{(L)} $ at the evaluation points
\end{itemize}
\item Steps
\begin{enumerate}
    \item Construct the $L\times L$ diagonal matrix $\tilde{\bm \Lambda} $, with $ \tilde \Lambda_{jj} = e^{\tau(1- \lambda_{j}^{-1})} $, and the $N\times L $ matrix $ \bm \Phi $, whose $j$-th column is equal to $  \vec \phi_{j} $. 
    \item Form the skew-symmetric, tridiagonal, $N\times N$ finite-difference matrix $ \bm V $ with
        \begin{displaymath}
            2\, \Delta t\,\bm V   = \begin{pmatrix}
                0 & \frac{1}{2} &    \\
                - \frac{1}{2} & 0 & 1 & \\
             & - 1 & 0 & 1   \\
            & & \ddots & \ddots & \ddots \\
             & & &  -1 & 0 & 1 \\
             & & & & -\frac{1}{2} & 0 & \frac{1}{2}
         \end{pmatrix}.
    \end{displaymath}
\item Compute the $L\times L$ skew-symmetric matrix ${\bm W} =  {\tilde{ \bm \Lambda}^{1/2} } {\bm\Phi}^\top{\bm V} {\bm\Phi} \tilde {\bm \Lambda}^{1/2}$.
\item Set the eigenfrequencies $ \omega_0, \ldots, \omega_{L-1} $ to the imaginary parts of the eigenvalues of $\bm W$. Set $ \vec \xi_j$ to the corresponding eigenvectors, normalized to unit 2-norm.
\item For each eigenvector $ \vec \xi_j $, compute the Dirichlet energy  
    \begin{displaymath}
        \mathcal{D}_j = \left( \frac{ \lVert \tilde{ \bm \Lambda }^{1/2} \bm \Lambda^{-1/2} \vec \xi_j \rVert_2^2 } {\lVert \tilde{ \bm \Lambda }^{1/2} \vec \xi_j \rVert_2^2} - 1 \right) ( 1 - ( \omega_j \, \Delta t )^2 )^{-1}, \quad \bm \Lambda = \diag( \lambda_0, \ldots, \lambda_{L-1} ).  
    \end{displaymath}
\item Form the $ \hat N \times L $ matrix $ \bm \Psi $, whose $j$-th column is equal to $ \vec \psi_j $, and set $ \vec \zeta_j = \bm \Psi \vec \xi_j$.
    \end{enumerate}
\end{itemize}
\end{algorithm}

Finally, Algorithm~\ref{alg:pred} computes the values of the data-driven forecast function $ f_{\tau,N,\Delta t,L,L'}^{(t)}$ from Corollary~\ref{corPredic} for lead time $ t \geq 0 $ at a set of evaluation points $ \{ \hat x_0, \ldots, \hat x_{\hat N} \} \subset M $, using the output of Algorithm~\ref{alg:generator} and the values $f(x_0), \ldots, f(x_{N-1}) $ of the prediction observable $f $ on the dynamical trajectory $ x_0, \ldots, x_{N-1}$. The forecast values  are output as a column vector $ \hat f = ( f_{\tau,N,\Delta t, L, L'}^{(t)}( \hat x_0), \ldots, f_{\tau,N,\Delta t, L,L'}^{(t)}( \hat x_{\hat N-1}))^\top \in \cmplx^{\hat N}$. Note that a similar approach can be employed to evaluate the approximations in Theorem~\ref{thm:data_predic}(iii) for general bounded continuous functions $Z : i \mathbb{R} \to \cmplx $. 

\begin{algorithm}[Data-driven prediction]\label{alg:pred}  \
\begin{itemize}
\item Inputs
\begin{itemize}
    \item Lead time $ t \geq 0$
	\item Number of basis functions $L'\leq L $
    \item Time series $f(x_0), \ldots, f(x_{N-1}) \in \cmplx$ of the prediction observable 
    \item Eigenvalues $  \lambda_0, \ldots, \lambda_{L-1}  $ and eigenvectors $  \vec \phi_0, \ldots, \vec \phi_{L-1}  $ from Algorithm~\ref{alg:data_basis}
    \item Eigenfrequencies $ \omega_{0}, \ldots, \omega_{L-1}$, eigenvectors $\vec \xi_0, \ldots, \vec \xi_{L-1} $, and pointwise-evaluated eigenfunctions $ \vec \zeta_0, \ldots, \vec \zeta_{L-1}  $ from Algorithm~\ref{alg:generator}
\end{itemize}
\item Outputs
    \begin{itemize}
        \item Column vector $ \hat f \in \cmplx^{\hat N} $ with the values of the forecast function for $U^tf $ at the evaluation points
    \end{itemize}
\item Steps
    \begin{enumerate}
        \item Form the column vector of observable values $ \vec f = ( f( x_0 ), \ldots, f(x_{N-1} ) )^\top \in \cmplx^N$  
        \item Compute the column vector of expansion coefficients $ \vec c = ( c_0, \ldots, c_{L-1} )^\top \in \cmplx^L $, where 
            \begin{displaymath}
                c_j = \begin{cases}
                    \vec \phi_j^\top \vec f / \lambda_j^{1/2}, & j \leq L', \\
                    0, & \text{otherwise}.
                \end{cases}
            \end{displaymath}
        \item Form the $L \times L $ diagonal matrix $ \bm U = \diag( 1, e^{i \omega_1 t}, \ldots, e^{i\omega_{L-1} t} ) $, the $ \hat N \times L $ eigenfunction matrix $ \bm Z $ whose $j$-th column is equal to $ \vec \zeta_j $, and set $\hat f = \bm Z \bm U \vec c $. 
    \end{enumerate}
    \end{itemize}
\end{algorithm}

%-_-_-_-_-_-_-_-_-_-_-_-_-_-_-_-_-_-_-_-_-_-_-_-_-_-_-_-_-_-_-_-_-_-_-_-_-_-_-_-_-_-_-_-_-_-_-_-_-_-_-_-_-_-_-_-_-_-_-_-_-_-_-_-_-_-_-_-_-_-_-_-_-_-_-_-
\bibliography{References}
\end{document}